\documentclass{article}
\usepackage[utf8]{inputenc}
\usepackage{amsmath}
\usepackage{amssymb}
\usepackage{amsthm}
\usepackage{xcolor}
\usepackage{fullpage}
\usepackage{hyperref}
\usepackage{tikz}
\usepackage{soul}
\usepackage{mathrsfs}
\usepackage{enumitem}
\usepackage{calc}
\usepackage{microtype}
\usepackage{accents}
\usepackage{bm}

\newcommand{\vardbtilde}[1]{\tilde{\raisebox{0pt}[0.85\height]{$\tilde{#1}$}}}

\usepackage{accents}
\newlength{\dhatheight}
\newcommand{\doublehat}[1]{%
    \settoheight{\dhatheight}{\ensuremath{\hat{#1}}}%
    \addtolength{\dhatheight}{-0.35ex}%
    \hat{\vphantom{\rule{1pt}{\dhatheight}}%
    \smash{\hat{#1}}}}

\newcommand{\triplehat}[1]{%
    \settoheight{\dhatheight}{\ensuremath{\doublehat{#1}}}%
    \addtolength{\dhatheight}{-0.35ex}%
    \hat{\vphantom{\rule{1pt}{\dhatheight}}%
    \smash{\doublehat{#1}}}}

\newtheorem{theorem}{Theorem}[section]

\newtheorem{corollary}[theorem]{Corollary}
\newtheorem{lemma}[theorem]{Lemma}
\newtheorem{definition}[theorem]{Definition}
\newtheorem{proposition}[theorem]{Proposition}
\newtheorem{remark}[theorem]{Remark}

\newtheorem*{theorem*}{Theorem}

\title{Hilbert transforms and maximal functions along flat curves on the Heisenberg group}
\author{Lingxiao Zhang}

\begin{document}
\maketitle

\begin{abstract}
We establish the $L^p$ boundedness of Hilbert transforms and maximal functions along flat curves in the Heisenberg group. This generalizes the $\mathbb{R}^n$  result by Carbery, Christ, Vance, Wainger, and Watson. What is new about our result compared to the Heisenberg group generalization by Carbery, Wainger, and Wright is that we allow all three components of the curves to vary independently, we keep the original form of the conditions required in the $\mathbb{R}^n$ case, and our method is likely to be generalized to other stratified nilpotent groups. 
\end{abstract}

%\small{\textbf{Keywords:} Hilbert transform along flat curves}

\section{Introduction}
Consider the singular Radon transform 
$$
Rf(x)= \psi(x) \int f(\gamma_t(x))\,K(t)\,dt,
$$
where $\psi(x)\in C_c^\infty(\mathbb{R}^n)$, $K(t)$ is a product kernel on $\mathbb{R}^N$ with compact support, and $\gamma_t(x)$ is a $C^\infty$ function mapping from a neighborhood of $(t,x)=(0,0)\in \mathbb{R}^N\times \mathbb{R}^n$ to $\mathbb{R}^n$ with $\gamma_0(x)\equiv x$. (The product kernel with only one parameter is the Calder\'on-Zygmund kernel. See the definition of the product kernel in Definition 2.1.1 of Nagel, Ricci, and Stein \cite{NRS01}.) We are interested in the case where $\gamma_t(x)$ is not real-analytic. In the case where $\gamma_t(x)$ is real-analytic, Stein, Street, and Zhang \cite{L2,LP,ANALYTIC,ME} proves a necessary and sufficient condition on $\gamma_t(x)$ such that the operator $R$ is bounded on $L^p(\mathbb{R}^n) (1<p<\infty)$, for every $\psi(x)$ with sufficiently small support and every $K(t)$ with sufficiently small support. 

In the case where $\gamma_t(x)$ is not real-analytic, the same condition is no longer necessary. In particular, when $K(t)$ is a Calder\'on-Zygmund kernel on $\mathbb{R}$, the Stein-Street condition says if we write the Taylor expansion of
$$
W(t,x):=\frac{d}{d\epsilon}\Big|_{\epsilon=1} \gamma_{\epsilon t} \circ \gamma_t^{-1}(x) \sim \sum_{n=0}^\infty t^n X_n(x),
$$
finitely many members of $\{\text{the vector fields } X_n \text{ and their iterated commutators}\}$ should span $W(t,x)$, with certain control on the coefficients (see Definition 4.2 of \cite{LP}). However, the following operator
$$
Lf(x)=\int_{-\frac{1}{3}}^{\frac{1}{3}} f(x_1-t, x_2-e^{-\frac{1}{|t|}}\,\text{sgn}\,t)K(t)\,dt,
$$
where $K(t)$ is a Calder\'on-Zygmund kernel on $\mathbb{R}$, does not satisfy the Stein-Street condition, but is still bounded on $L^p(\mathbb{R}^2)$ for every $1<p<\infty$ and every Calder\'on-Zygmund kernel $K(t)$ on $\mathbb{R}$; this is implied by the proof of Carbery, Vance, Wainger, Watson, and Wright \cite{withoutfourier}.

In the non-analytic case, characterizing which singular Radon transforms are bounded on $L^p$ is open. For example, it is an open conjecture by Stein that 
$$
Sf(x)=\int_{-\sigma}^\sigma f\big(x-tV(x)\big)\,\frac{dt}{t}
$$
is bounded on $L^2$ for some $\sigma>0$, as long as the vector field $V(x)$ is Lipschitz. If $V(x)$ is real analytic, by Bourgain \cite{bourgain}, there exists $\sigma>0$ such that $S$ is bounded on $L^p$ for every $1<p<\infty$; this is also implied by Street and Stein \cite{ANALYTIC}. Thus one only needs to deal with the non-analytic case. Considerable effort and progress have been dedicated to this problem, including Bateman and Thiele \cite{THIELE} and Lacey and Li \cite{LACEYLI10}. See \cite{LACEYLI10} for a history. Nonetheless, the validity of the conjecture is still unknown even if we assume $V(x)$ is $C^\infty$. Additional techniques might be needed to address non-analytic singular Radon transforms.

A breakthrough in the non-analytic singular Radon transforms appears in Theorems \ref{R2} and \ref{RN} below due to Carbery, Christ, Vance, Wainger, and Watson. They study the following operators on which we will focus in the paper:
\begin{equation}\label{H}
H_\Gamma f(x) = \int_{-1}^1 f\big(x-\Gamma(t)\big)\,\frac{dt}{t},
\end{equation}
\begin{equation}\label{M}
M_\Gamma f(x)= \sup_{0<\sigma\leq 1} \frac{1}{\sigma} \int_0^\sigma \big|f\big(x-\Gamma(t)\big)\big|\,dt,
\end{equation}
where $\Gamma: [-1,1]\to \mathbb{R}^n$ is a continuous odd curve. Thus $\Gamma(0)=0$. Assume $\Gamma(t)=\big(\alpha_1(t), \alpha_2(t), \ldots, \alpha_n(t)\big)$ with $\alpha_1(t)=t$ and $\Gamma'(0)=(1,0,\ldots, 0)$. This may be viewed as a normalization for curves with $|\Gamma'(0)|\neq0$. We call such operators (local) Hilbert transform and (local) maximal function along the flat odd curve $\Gamma(t)$. 

Note the upper and lower limits of the integral in (\ref{H}) are not essential and can be replaced by any positive number and its opposite number, respectively. Similarly, for (\ref{M}), the supremum can be taken over the interval from $0$ to any positive number. 

Carbery, Christ, Vance, Wainger, and Watson establish an $L^p$ boundedness result in the $\mathbb{R}^2$ case:

\begin{theorem}[\cite{CARBERYCHRIST}]\label{R2}
Suppose $\Gamma(t)=\big(t,\alpha(t)\big)$ is a continuous odd curve in $\mathbb{R}^2$ with $\alpha'(0)=0$, and is $C^2$ for $t\in (0,\infty)$. 
%\color{blue} (If we assume $\alpha_2'(0)=0$ or $\Gamma'(0)=(1,0,\ldots,0)$, this may be viewed as a normalization for $|\Gamma'(0)|\neq 0$ by a change of parametrization of the curve and by a rotation around the origin) \color{black}
Denote
$$
h(t)=t\alpha'(t)-\alpha(t).
$$
If $\alpha''(t)>0$ for $t>0$, and
$$
h'(t)\gtrsim \frac{h(t)}{t}, \quad \text{for } t>0,
$$
then $H_\Gamma$ is bounded on $L^p(\mathbb{R}^2)$ for $1<p< \infty$, and $M_\Gamma$ is bounded on $L^p(\mathbb{R}^2)$ for $1<p\leq \infty$. 
\end{theorem}

Note $\alpha(t)''>0$ for $t>0$ implies that $h(t)>0$ and $h'(t)=t\alpha''(t)>0$ for $t>0$.

Later, Carbery, Vance, Wainger, and Watson prove the $\mathbb{R}^n$ result:

\begin{theorem}[\cite{VANCE}]\label{RN}
Suppose $\Gamma(t)=\big(t,\alpha_2(t), \ldots, \alpha_n(t)\big)$ is a continuous odd curve in $\mathbb{R}^n$ with $\Gamma'(0)=(1,0,\ldots, 0)$, and is $C^n$ for $t\in(0,\infty)$.  Denote
$$
N_j(t)=\det
\begin{pmatrix}
t & \alpha_2(t) & \cdots & \alpha_j(t)\\
1 & \alpha_2'(t) & \cdots & \alpha_j'(t)\\
\vdots & \vdots & \vdots & \vdots\\
0 & \alpha_2^{(j-1)}(t) & \cdots & \alpha_j^{(j-1)}(t)
\end{pmatrix},
\quad 1\leq j\leq n,
$$
$$
D_j(t) = \det
\begin{pmatrix}
1 & \alpha_2'(t) & \cdots & \alpha_j'(t)\\
0 & \alpha_2''(t) & \cdots & \alpha_j''(t)\\
\vdots & \vdots & \vdots & \vdots\\
0 & \alpha_2^{(j)}(t) & \cdots & \alpha_j^{(j)}(t)
\end{pmatrix},
\quad 1\leq j \leq n,
$$
and
$$
h_j(t) = \frac{N_j(t)}{D_{j-1}(t)}, \quad 2\leq j \leq n.
$$
If 
\begin{equation}\label{positived}
D_j(t)>0, \quad 1\leq j \leq n, \quad t>0,   
\end{equation}
and
$$
h_j'(t) \gtrsim \frac{h_j(t)}{t}, \quad 2\leq j \leq n, \quad t>0,
$$
then $H_\Gamma$ is bounded on $L^p(\mathbb{R}^n)$ for $1<p<\infty$, and $M_\Gamma$ is bounded on $L^p(\mathbb{R}^n)$ for $1<p\leq \infty$.
\end{theorem}

By Lemmas 1 and 2 in Nagel, Vance, Wainger, and Weinberg \cite{NVWW}, (\ref{positived}) implies $h_j>0$ and $h_j'>0$ for $t>0$.

When $n=2$, $h_2(t)$ and $D_2(t)$ coincide with $h(t)$ and $\alpha''(t)$ in Theorem \ref{R2}, respectively. Thus Theorem \ref{RN} coincide with Theorem \ref{R2} in the $n=2$ case. When $n>2$, for $t>0$, in place of $\alpha''(t)>0$, the convexity condition becomes $D_j(t)>0$ for each $j$. 
%Note that if we only assume the $D_j(t)$ never vanish for $t>0$, one can make $D_j(t)>0$ for $t>0$ by replacing $\Gamma(t)$ with a new curve $\widetilde \Gamma(t):=(t, \pm \alpha_2(t), \ldots, \pm \alpha_n(t))$. The conclusion of the theorem holds for $\Gamma(t)$ if and only if it holds for $\widetilde \Gamma(t)$. Hence (\ref{positived}) in the theorem can be replaced by $D_j(t)>0$ for $t>0$ for each $j$. 

Singular integrals in a doubling metric space (i.e., a space of homogeneous type in the sense of Coifman and Weiss) are often used to help bound the singular Radon transforms, as in Christ, Nagel, Stein, and Wainger \cite{CNSW} and in Street and Stein \cite{LP}. However, one single doubling metric space does not suffice to achieve Theorems \ref{R2} and \ref{RN} (see the remark following Equation (3.14) in \cite{CARBERYCHRIST}). Instead, Carbery, Christ, Vance, Wainger, and Watson \cite{CARBERYCHRIST} design a sequence of scaling maps and accordingly a sequence of interrelated doubling metrics. See also Carbery, Vance, Wainger, and Wright \cite{notion} for this variation of the notion of a space of homogeneous type.

Unfortunately, there are only a few generalizations using these powerful techniques. The Heisenberg group case is usually the first step toward generalizing to cases that are not translation-invariant. Kim \cite{KIM05} generalized the result to operators along curves $\big(t,\alpha(t), 0\big)$ in the Heisenberg group, using the group Fourier transform. To expedite future generalizations, Carbery, Vance, Wainger, Watson, and Wright \cite{withoutfourier} give a proof of Theorem \ref{RN} avoiding the Fourier transform and using the $TT^*$ method. A substantial progress has been made by Carbery, Wainger, and Wright \cite{CWW}, using the $TT^*$ method. They generalized the result to operators along curves $\big(t,\alpha(t), |t|\alpha(t)\big)$ in the Heisenberg group, under a different condition; see Theorem \ref{CWW} below. 

In the Heisenberg group $\mathbb{H}^1$, (\ref{H}) and (\ref{M}) becomes
\begin{equation}\label{hilbert}
\mathcal{H}_\Gamma f(x)=\int_{-1}^1 f\big(x\cdot \Gamma(t)^{-1}\big) \,\frac{dt}{t},
\end{equation}
\begin{equation}\label{maximal}
\mathcal{M}_\Gamma f(x)=\sup_{0<\sigma\leq 1} \frac{1}{\sigma} \int_0^\sigma \big|f\big(x\cdot \Gamma(t)^{-1}\big)\big|\,dt,
\end{equation}
where the group multiplication and inverse are defined by
$$
x\cdot y= (x_1,x_2,x_3)\cdot (y_1, y_2, y_3)=\big(x_1+y_1, x_2+y_2, x_3+y_3+\frac{1}{2}(x_1y_2-x_2y_1)\big),
$$
and
$$
x^{-1}=(x_1, x_2, x_3)^{-1}=(-x_1, -x_2, -x_3).
$$

\begin{theorem}[Carbery, Wainger, and Wright \cite{CWW}]\label{CWW}
Suppose $\Gamma(t)=\big(t, \alpha(t), |t|\alpha(t)\big)$ is a $C^3$ odd curve in $\mathbb{H}^1$ with $\alpha'(0)=0$. Denote
$$
\lambda(t) =\frac{t\alpha''(t)}{\alpha'(t)}.
$$
If $\alpha''(t)>0$ for $t>0$, and $\lambda$ is decreasing and positive for $t>0$, such that
$$
\lim_{t\downarrow 0} \frac{t\lambda'(t)}{\lambda(t)^2} \log^+\lambda(t) =0,
$$
then $\mathcal{H}_\Gamma$ is bounded on $L^p(\mathbb{H}^1)$ for $1<p<\infty$, and $\mathcal{M}_\Gamma$ is bounded on $L^p(\mathbb{H}^1)$ for $1<p\leq \infty$.
\end{theorem}

The following is the main result of the paper. We allow a third independent component in $\mathbb{H}^1$, and keep the basic form of the conditions in Theorems \ref{R2} and \ref{RN} so as to make the most of their techniques.

\begin{theorem}\label{1}
Suppose $\Gamma(t)=\big(t, \alpha(t), \beta(t)\big)$ is a continuous odd curve in $\mathbb{H}^1$ with $\alpha'(0)=0$, and is $C^3$ for $t \in(0,\infty)$. Denote
$$
N_1(t)=t, \quad N_2(t)=t\alpha'(t)-\alpha(t),\quad D_1(t)=1, \quad D_2(t)=\alpha''(t), \quad h_2(t)=\frac{N_2(t)}{D_1(t)},
$$
$$
\bar \beta(t)=\beta(t)-\frac{1}{2}\int_0^t h_2, \quad \bar N_3(t) = \det 
\begin{pmatrix}
t & \alpha(t) & \bar \beta(t)\\
1 & \alpha'(t) & \bar \beta'(t)\\
0 & \alpha''(t) & \bar \beta''(t)
\end{pmatrix},
\quad \bar h_3(t)= \frac{\bar N_3(t)}{D_2(t)},
$$
and
\begin{align*}
\bar D_3(t) &=\alpha''(t)\bar \beta'''(t)-\alpha'''(t)\bar \beta''(t),\\
\hat D_3(t)&= \alpha''(t)\big(\bar \beta'''(t)+2\alpha''(t)\big)-\alpha'''(t)\big(\bar \beta''(t)+\alpha'(t)\big),\\
\doublehat D_3(t)&= \alpha''(t)\big(\bar \beta'''(t)-\alpha''(t)\big)-\alpha'''(t)\big(\bar \beta''(t)-\alpha'(t)\big),\\
\triplehat D_3(t) &=\alpha''(t)\big(\bar \beta'''(t)-\alpha''(t)\big)-\alpha'''(t)\Big(\bar \beta''(t)-\alpha'(t)+\frac{\alpha(t)}{t}\Big).    
\end{align*}
If $D_2(t), \bar D_3(t), \hat D_3(t), \doublehat D_3(t), \triplehat D_3(t)>0$ for $t>0$, and 
\begin{equation}\label{infinitesimal}
h_2'(t) \gtrsim \frac{h_2(t)}{t}, \quad \bar h_3'(t) \gtrsim \frac{\bar h_3(t)}{t}, \quad \bar h_3(t) \gtrsim \int_0^t h_2, \quad \text{for } t>0,
\end{equation}
then $\mathcal{H}_\Gamma$ is bounded on $L^p(\mathbb{H}^1)$ for $1<p<\infty$, and $\mathcal{M}_\Gamma$ is bounded on $L^p(\mathbb{H}^1)$ for $1<p\leq \infty$.
\end{theorem}

Again, by Lemmas 1 and 2 in Nagel, Vance, Wainger, and Weinberg \cite{NVWW}, $D_2, \bar D_3>0$ for $t>0$ implies $h_2, \bar h_3>0$ and $h_2', \bar h_3'>0$ for $t>0$. 

The theorem still holds if the kernel $\frac{1}{t}$ in (\ref{hilbert}) is replaced by any Calder\'on-Zygmund kernel on $\mathbb{R}$; see Remark \ref{decompose}.

A closely related work using different techniques is by Kim \cite{KIM02}, which studies the $\big(t, \alpha(t), |t|\alpha(t)\big)$ case in the Heisenberg group, under the condition that there exists $C>1$ such that $\alpha'(Ct) \geq 2\alpha'(t)$ for $t>0$. It continues in the line of Carlsson, Christ, Cordoba, Duoandikoetxea, Rubio de Francia, Vance, Wainger, and Weinberg \cite{doublingtime}.

\begin{remark}
Even in the $\big(t,\alpha(t), |t|\alpha(t)\big)$ case in $\mathbb{H}^1$, our theorem covers certain situations not covered in Kim \cite{KIM02} or Carbery, Wainger, and Wright \cite{CWW}. Consider the odd curve given for small positive $t$ by $\alpha(t)= \frac{t}{\log \frac{1}{t}}$, (initially constructed in \cite{CARBERYCHRIST} for the $\mathbb{R}^2$ case). It satisfies the condition in Theorem \ref{1}, and thus the Hilbert transform and maximal function along the curve are bounded on $L^p(\mathbb{H}^1)$ for every $1<p<\infty$. But such curve does not satisfy the condition of \cite{KIM02} or that of \cite{CWW}. See a detailed comparison of related works in history in Sections 1.1 and 1.3 of \cite{CARBERYCHRIST}.
\end{remark}

Our method is likely to be generalized to other stratified nilpotent groups. See Remarks \ref{before} and \ref{after} for what barrier has been eliminated between $\mathbb{R}^n$ cases and cases that are not translation-invariant.

It is worth noting that Theorem \ref{1} does not exhaust all the cases for the Heisenberg group that can utilize the techniques by Carbery, Christ, Vance, Wainger, and Watson. The condition in Theorem \ref{1} suggests that it is the case that ``$\beta(t)$ dominates $\alpha(t)$''. This paper uses this case to present a method that removes obstructions to generalizing the techniques to translation non-invariant cases. Verifying the cases ``$\alpha(t)$ dominates $\beta(t)$'' and ``$\alpha(t)$ and $\beta(t)$ do not dominate each other'', together with cases of other stratified nilpotent groups, is not the purpose of this paper.

In Section \ref{section2}, we set up a sequence of metrics that gives rise to a Calder\'on-Zygmund theory. Such Calder\'on-Zygmund theory leads to a Littlewood-Paley theory in Section \ref{littlewood}, which will be well-suited to address the $L^p$ boundedness of $\mathcal{H}_\Gamma$ and $\mathcal{M}_\Gamma$. Section \ref{section4} is the central part of this paper, where we use the $TT^*$ method. The goal of Section \ref{section4} is to prove Theorem \ref{main}. At the beginning of Section \ref{section4}, we reduce the matter to proving the regularity of certain measures (Proposition \ref{density}). In Subsection \ref{4.1}, we reduce Proposition \ref{density} to an estimation of a Jacobian (Proposition \ref{good}) and the regularity of the Jacobian and a partition of the unity functions (Proposition \ref{estimateofderivative}). We prove Proposition \ref{good} in Subsection \ref{jacobian} and prove Proposition \ref{estimateofderivative} in Subsection \ref{unity}. To maintain the coherence of the proof in Section \ref{section4} and to avoid repeating similar arguments, we omit several cases which we discuss in Subsection \ref{4.4}. Having shown Theorem \ref{main}, our conclusion (Theorem \ref{1}) follows as in Section \ref{finalsection}.

\section{Calder\'on-Zygmund theory}\label{section2}
In this section, we will determine a sequence of metrics $\{\rho_l(x,y)\}_{l\in \mathbb{Z}}$ in the Heisenberg group. A similar argument as in Section 2 of \cite{CARBERYCHRIST} will then establish a Calder\'on-Zygmund theory with respect to these metrics; we will include the proof to convince the reader that an adaptation of the argument in Section 2 of \cite{CARBERYCHRIST} to our metrics in the Heisenberg group remains valid. Such Calder\'on-Zygmund theory will give rise to a corresponding Littlewood-Paley theory in Section \ref{littlewood}, which will be well-suited to address the $L^p$ boundedness of $\mathcal{H}_\Gamma$ and $\mathcal{M}_\Gamma$. 

Identify the Heisenberg group $\mathbb{H}^1$ with $\mathbb{R}^3$ using coordinates $(x_1, x_2, x_3)$, where the group multiplication is
$$
x\cdot y=(x_1, x_2, x_3)\cdot (y_1, y_2, y_3) = \big(x_1+y_1, x_2 + y_2, x_3+y_3 +\frac{1}{2}(x_1y_2-x_2y_1)\big).
$$
A basis of left-invariant vector fields can be taken as
$$
X=\partial_{x_1}-\frac{1}{2}x_2\partial_{x_3}, \quad Y=\partial_{x_2} + \frac{1}{2}x_1\partial_{x_3}, \quad T=\partial_{x_3},
$$
which satisfies
$$
[X, Y]=T, \quad [X,T]=[Y,T]=0.
$$
We have
$$
x=(x_1, x_2, x_3)=e^{x_1X+x_2Y+x_3T}0,
$$
and
$$
x\cdot y=e^{x_1X+x_2Y+x_3T}e^{y_1X+y_2Y+y_3T}0 =x\cdot e^{y_1X+y_2Y+y_3T}0.
$$

\begin{definition}\label{matrix}
Denote $\bar h_\beta(t) = t\bar \beta'(t)-\bar \beta(t)$.
For each $t>0$, define $3\times 3$ matrices 
$$
A(t) = 
\begin{pmatrix}
t & 0 & 0\\
\alpha(t) & h_2(t) & 0\\
\bar \beta(t) & \bar h_\beta(t) & \bar h_3(t)
\end{pmatrix},
$$
and
$$
\widetilde A(t)=
\begin{pmatrix}
t & 0 & 0\\
\alpha(t) & h_2(t) & 0\\
\bar \beta(t) & \bar h_\beta(t) & th_2(t) +\bar h_3(t)
\end{pmatrix}.
$$
\end{definition}

We have
$$
A(t)^{-1}=
\begin{pmatrix}
\frac{1}{t} & 0 & 0\\
-\frac{\alpha(t)}{th_2(t)} & \frac{1}{h_2(t)} & 0\\
\frac{\alpha(t)\bar h_\beta(t)/h_2(t)-\bar \beta(t)}{t\bar h_3(t)}& -\frac{\bar h_\beta(t)}{h_2(t)\bar h_3(t)} & \frac{1}{\bar h_3(t)}
\end{pmatrix},
$$
and
$$
\widetilde A(t)^{-1}=
\begin{pmatrix}
\frac{1}{t} & 0 & 0\\
-\frac{\alpha(t)}{th_2(t)} & \frac{1}{h_2(t)} & 0\\
\frac{\alpha(t)\bar h_\beta(t)/h_2(t)-\bar \beta(t)}{t(th_2(t)+\bar h_3(t))}& -\frac{\bar h_\beta(t)}{h_2(t)(th_2(t)+\bar h_3(t))} & \frac{1}{th_2(t)+\bar h_3(t)}
\end{pmatrix}.
$$
By Proposition 4.2 in \cite{VANCE}, there exist  $\epsilon>0$ and $C_0\geq 1$ such that
\begin{equation}\label{frobenius}
\|A(s)^{-1} A(t)\|\leq C_0 \big(\frac{t}{s}\big)^\epsilon, \quad \text{for } 0<t\leq s,    
\end{equation}
where $\|\cdot\|$ denotes the Frobenius norm of the matrix, i.e., $\|(a_{ij})_{i,j}\|=(\sum_{i,j}a_{ij}^2)^{1/2}$. In particular,
\begin{equation}\label{h2doubling}
\frac{h_2(t)}{h_2(s)} \leq C_0 \big(\frac{t}{s}\big)^\epsilon, \quad \text{ for } 0<t\leq s.
\end{equation}
By enlarging $C_0\geq 1$ in (\ref{frobenius}) if necessary, we have
\begin{equation}\label{rivere}
\|\widetilde A(s)^{-1} \widetilde A(t)\|\leq C_0 \big(\frac{t}{s}\big)^\epsilon, \quad \text{for } 0<t\leq s.  
\end{equation}

\begin{definition}
Associated with each $t>0$, define a basis of left-invariant vector fields on $\mathbb{H}^1$:
$$
X_{t} = tX + \alpha(t) Y +\bar \beta(t)T, \quad Y_{t} = h_2(t) Y + \bar h_\beta(t)T, \quad T_{t} = \big(th_2(t) + \bar h_3(t)\big)T.
$$
Denote
$$
\xi_{t}= \frac{1}{2} \frac{th_2(t)}{th_2(t)+\bar h_3(t)} \leq \frac{1}{2}.
$$
\end{definition}

We have
\begin{equation}\label{commutator}
[X_{t}, Y_{t}]=2\xi_{t}T_{t}.
\end{equation}

\begin{definition}
Associated with each $t>0$, we define the ball with center $x$ and radius $r>0$ as
\begin{align*}
B_{t}(x,r)&=\big\{x\cdot e^{z_1X_{t}+z_2Y_{t}+z_3T_{t}}0:|z_1|,|z_2|<r, |z_3|<r^2\big\}\\
&=\big\{x\cdot \big(\widetilde A(t) z\big):z\in [-r,r]^2\times [-r^2, r^2]\big\},
\end{align*}
where $\widetilde A(t)z$ denotes the matrix $\widetilde A(t)$ multiply the vector $z$. In particular,
\begin{equation}\label{invariantball}
B_{t}(0,r)=x^{-1}\cdot B_{t}(x,r).
\end{equation}
\end{definition}

Recall that $h_2(t), \bar h_3(t)$ are positive increasing functions for $t>0$. Thus 
$$
th_2(t)\bar h_3(t)\to 0, \quad \text{ as }t\to 0, \qquad th_2(t)\bar h_3(t) \to \infty, \quad \text{ as } t \to \infty.
$$

\begin{definition}
Let $\{t_l\}_{l\in \mathbb{Z}}$ be a doubly infinite sequence convergent with $t_l\to 0$ as $l\to \infty$ and $t_l\to \infty$ as $l\to -\infty$, such that
\begin{equation}\label{doubling}
1\leq c_0:=\sup_{l\in \mathbb{N}} \frac{t_lh_2(t_l)\bar h_3(t_l)}{t_{l+1} h_2(t_{l+1}) \bar h_3(t_{l+1})} <\infty.
\end{equation}
This can be achieved by letting $\{t_l\}$ include all the numbers $\{2^{-j}: j\in \mathbb{Z}\}$, and also include finitely many more numbers between $2^{-j-1}$ and $2^{-j}$ if the function $th_2(t)\bar h_3(t)$ grows rapidly between $t=2^{-j-1}$ and $t=2^{-j}$. 
\end{definition}

Let $|\cdot|$ denote the Lebesgue measure of any set in $\mathbb{H}^1\cong \mathbb{R}^3$. We have
\begin{equation}\label{area}
\begin{aligned}
|B_{t}(x,r)|&=\int_{[-r,r]^2\times [-r^2, r^2]} \Big| \det \frac{\partial \big(x\cdot (\widetilde A(t)z)\big)}{\partial z}\Big|\,dz\\
&=\int_{[-r,r]^2\times [-r^2, r^2]} \Big| \det \frac{\partial \big(\widetilde A(t)z\big)}{\partial z}\Big|\,dz\\
&=\int_{[-r,r]^2\times [-r^2, r^2]} t h_2(t)\big(th_2(t)+\bar h_3(t)\big)\,dz\\
&=8r^4 t h_2(t)\big(th_2(t)+\bar h_3(t)\big).
\end{aligned}
\end{equation}
Therefore by (\ref{doubling}), for every $l\in \mathbb{Z}$,
\begin{equation}\label{ratio}
\frac{|B_{t_l}(x,r)|}{|B_{t_{l+1}}(x,r)|} =\frac{t_l h_2(t_l)\big(t_lh_2(t_l)+\bar h_3(t_l)\big)}{t_{l+1} h_2(t_{l+1})\big(t_{l+1}h_2(t_{l+1})+\bar h_3(t_{l+1})\big)} \leq c_0^2.
\end{equation}

For every $r\geq 1$ and every $t'\leq t$, if $y\in B_{t'}(x,r)$, we can write
$$
y=x\cdot \big(\widetilde A(t')z'\big), \quad \text{ for some } z'\in [-r,r]^2\times [-r^2, r^2].
$$
Thus
$$
y=x\cdot \big(\widetilde A(t)  \widetilde A(t)^{-1} \widetilde A(t') z'\big).
$$
Due to (\ref{rivere}), by denoting $z=(z_1, z_2, z_3):=\widetilde A(t)^{-1} \widetilde A(t') z'$, we have
$$
|z_1|\leq C_0 r, \quad |z_2|\leq 2C_0 r, \quad |z_3|\leq 3C_0 r^2 \leq (2C_0r)^2.
$$
Hence $y\in B_{t}(x,2C_0r)$. Therefore for every $r\geq 1$ and $t'\leq t$, 
\begin{equation}\label{smallalsoinbig}
B_{t'}(x,r)\subseteq B_{t}(x,2C_0r). 
\end{equation}
In particular, for every $r\geq 1$ and $l'\geq l$, 
\begin{equation}\label{smallinbig}
B_{t_{l'}}(x,r)\subseteq B_{t_l}(x,2C_0r).
\end{equation}

For any $y\in \bigcap_{l\in \mathbb{Z}} B_{t_l}(x,r)$, fixing an $l_0\in \mathbb{Z}$, we have for every $l\in \mathbb{Z}$,
$$
x^{-1}\cdot y =\widetilde A(t_{l_0})z_0=\widetilde A(t_l)z_l, 
$$
where $z_0, z_l\in [-r,r]^2\times [-r^2, r^2]$. Then by (\ref{rivere}),
$$
|z_0|=|\widetilde A(t_{l_0})^{-1}\widetilde A(t_l)z_l|\leq C_0 \big(\frac{t_l}{t_{l_0}}\big)^\epsilon |z_l| \to 0, \quad \text{as } l\to \infty,
$$
where $|z_0|$ denotes the $l^2$ norm of the vector $z_0\in \mathbb{R}^3$. Thus $z_0=0$ and $y=x$. Hence 
\begin{equation}\label{singleton}
\bigcap_{l\in \mathbb{Z}} B_{t_l}(x,r)=\{x\}.
\end{equation}

For any $y\in \mathbb{H}^1$, fixing an $l_0\in \mathbb{Z}$, denote
$$
z_0 := \widetilde A(t_{l_0})^{-1} \big(x^{-1} \cdot y\big), \quad z_l:= \widetilde A(t_l)^{-1}\widetilde A(t_{l_0})z_0, \quad \text{for } l\in \mathbb{Z}.
$$
By (\ref{rivere}), we have
$$
|z_l| \to 0, \quad \text{as } l\to -\infty.
$$
So we can pick an $l$ such that $z_l\in [-r,r]^2\times [-r^2, r^2]$. Then since
$$
x^{-1}\cdot y = \widetilde A(t_{l_0}) z_0 = \widetilde A(t_l) z_l,
$$
we have $y\in B_{t_l}(x,r)$. Hence
\begin{equation}\label{all}
\bigcup_{l\in \mathbb{Z}} B_{t_l}(x,r) = \mathbb{H}^1.
\end{equation}

\begin{definition}
For each $l\in \mathbb{Z}$, define
$$
\rho_l(x,y)=\sup \{r>0: y\not\in B_{t_l}(x,r)\}, \quad \forall x,y\in \mathbb{H}^1.
$$
Denote
$$
\rho_l(x,\Omega) = \inf_{y\in \Omega} \rho_l(x,y), \quad \forall x\in \mathbb{H}^1, \forall \Omega \subseteq \mathbb{H}^1.
$$
\end{definition}

Thus
$$
B_{t_l}(x,r)=\{y: \rho_l(x,y)<r\},
$$
and by (\ref{invariantball}), we have
\begin{equation}\label{invariantmetric}
\rho_l(x,y)=\rho_l(0,x^{-1}\cdot y).
\end{equation}

\begin{lemma}
For each $l\in \mathbb{Z}$, $\rho_l$ is a metric.
\end{lemma}

\begin{proof}
It is clear that $\rho_l(x,y)=0$ implies $x=y$. Note that
$$
y=x\cdot \big(\widetilde A(t_l) z\big), \quad \text{for some } z\in [-r,r]^2\times [-r^2, r^2],
$$
if and only if 
$$
x=y\cdot \big(\widetilde A(t_l) z\big)^{-1} = y\cdot \big(-\widetilde A(t_l) z\big) = y\cdot \big(\widetilde A(t_l) (-z)\big), \quad \text{for some } z\in [-r,r]^2\times [-r^2, r^2].
$$
Thus $y\in B_{t_l}(x,r)$ if and only if $x\in B_{t_l}(y,r)$. Therefore $\rho_l(x,y)=\rho_l(y,x)$. 

If $y\in B_{t}(x,r_1)$ and $w\in B_{t}(y,r_2)$, then $x^{-1}\cdot y\in B_t(0,r_1)$ and $y^{-1}\cdot w \in B_t(0, r_1)$, and thus there exists $z\in [-r_1, r_1]^2\times [-r_1^2, r_1^2], z'\in [-r_2, r_2]^2\times [-r_2^2, r_2^2]$ such that
\begin{align*}
x^{-1}\cdot w&=(x^{-1}\cdot y) \cdot (y^{-1} \cdot w)= \big(\widetilde A(t)z\big) \cdot \big(\widetilde A(t)z'\big)\\
&=e^{z_1X_{t}+z_2Y_{t}+z_3T_{t}}e^{z_1'X_{t}+z_2'Y_{t} +z_3'T_{t}}0\\
&=e^{(z_1+z_1')X_{t} + (z_2+z_2')Y_{t} + (z_3+z_3'+\xi_{t}(z_1z_2'-z_2z_1'))T_{t}}0,
\end{align*}
where the last the equality is due to (\ref{commutator}) and the Baker-Campbell-Hausdorff formula. We have
$$
|z_1+z_1'|, |z_2+z_2'|<r_1+r_2, \quad \big|z_3+z_3'+\xi_{t}(z_1z_2'-z_2z_1')\big|< r_1^2+r_2^2 + r_1r_2<(r_1+r_2)^2.
$$
Thus $x^{-1}\cdot w\in B_t(0, r_1+r_2)$ and $w\in B_{t}(x,r_1+r_2)$. Therefore 
\begin{equation}\label{alsotriangle}
B_t(0,r_1)\cdot B_t(0,r_2)\subseteq B_t(0,r_1+r_2),
\end{equation}
and
\begin{equation}\label{triangle}
\rho_l(x,w)\leq \rho_l(x,y)+\rho_l(y,w).
\end{equation}
Hence for each $l\in \mathbb{Z}$, $\rho_l$ is a metric on $\mathbb{H}^1$.
\end{proof}

We now show for any $r\geq 1$, there exists $C_1>0$ such that $B_{t_l}(x,r) \subseteq B_{t_{l+1}}(x,C_1r)$. Let $\{B_{t_{l+1}}(x_i, r)\}_i$ be a maximal disjoint collection such that $x_i\in B_{t_l}(x,r)$. By (\ref{smallinbig}) and (\ref{triangle}), 
$$
B_{t_{l+1}}(x_i,r)\subseteq B_{t_l}(x_i, 2C_0r) \subseteq B_{t_l}\big(x, (2C_0+1)r\big).
$$
Thus by (\ref{ratio}), there are at most $c_0^2 (2C_0+1)^4$ members in the disjoint collection $\{B_{t_{l+1}}(x_i, r)\}_i$. We have
$$
B_{t_l}(x,r)\subseteq \bigcup_i B_{t_{l+1}}(x_i, 2r).
$$
In fact, if there exists $\bar x\in B_{t_l}(x,r)\backslash \bigcup_i B_{t_{l+1}}(x_i, 2r)$, then by (\ref{triangle}),
$$
B_{t_{l+1}}(\bar x, r) \cap B_{t_{l+1}}(x_i, r)=\emptyset, \quad \forall i,
$$
which contradicts with $\{B_{t_{l+1}}(x_i, r)\}_i$ being maximal. Since $ \bigcup_i B_{t_{l+1}}(x_i, 2r)$ is path-connected, we can reorder $ \{B_{t_{l+1}}(x_i, 2r)\}_i$ so that each member intersects at least one of its predecessors. Thus by (\ref{triangle}), for any $y\in B_{t_l}(x,r) \subseteq \bigcup_i B_{t_{l+1}}(x_i, 2r)$,
$$
\rho_{l+1}(y,x) \leq 2r \cdot 2c_0^2(2C_0+1)^4.
$$
Denote $C_1=4c_0^2(2C_0+1)^4>1$. Therefore for every $r\geq 1$,
\begin{equation}\label{biginsmall}
B_{t_l}(x,r) \subseteq B_{t_{l+1}}(x, C_1r).
\end{equation}

\begin{remark}
Since the dilation automorphisms on the Heisenberg group are non-isotropic, (\ref{smallinbig}) and (\ref{biginsmall}) do not hold for $r<1$. This is the reason we can not direclty apply the result of \cite{notion}, which establishes a criterion for the validity of the Calder\'on-Zygmund theory corresponding to a given sequence of metrics. \cite{notion} covers more general cases than the translation-invariant cases in Section 2 of \cite{CARBERYCHRIST}. However, \cite{notion} requires (\ref{smallinbig}) and (\ref{biginsmall}) to hold for all $r>0$.
\end{remark}

\begin{definition}\label{notation}
Recall $C_0 \geq 1$ is defined in (\ref{frobenius}) and (\ref{rivere}), $c_0\geq 1$ is defined in (\ref{doubling}), and
$$
C_1:=4c_0^2(2C_0+1)^4>1
$$
as in (\ref{biginsmall}). Let
$$
A_0:=8C_0>6C_0>1,
$$
and fix a constant $A_1>C_1A_0$.
\end{definition}

We have the following covering lemma:
\begin{lemma}\label{covering}
For every nonempty open proper subset $\Omega$ in $\mathbb{H}^1$,
$$
\Omega = \bigcup_{B\in \mathcal{E}} B,
$$
where
$$
\mathcal{E}=\{B_{t_l}(x,1):x\in \Omega, l\in \mathbb{Z}, A_0<\rho_l(x,\partial \Omega)< A_1\}.
$$
If in addition $|\Omega|<\infty$, then there exists a sequence of disjoint balls $B_{t_{l_i}}(x_i,1)\in \mathcal{E}$ such that 
$$
\Omega =\bigcup_i B_{t_{l_i}}(x_i, 6C_0).    
$$
\end{lemma}

\begin{proof}
To show $\Omega = \cup_\mathcal{E} B$, it suffices to show that for any given $x\in \Omega$, there exists $l\in \mathbb{Z}$ such that
$$
A_0<\rho_l(x,\partial \Omega)< A_1.
$$
For this $x$, applying (\ref{all}) with $r=A_0$, there exists $l\in \mathbb{Z}$ such that
$$
B_{t_l}(x,A_0)\cap \partial \Omega \neq \emptyset,
$$
and thus $\rho_l(x, \partial \Omega )<A_0$. By (\ref{rivere}), fixing an $l_0\in \mathbb{Z}$, we have
$$
\|\widetilde A(t_l)\| \leq \|\widetilde A(t_{l_0})\| \cdot \|\widetilde A(t_{l_0})^{-1} \widetilde A(t_l)\|\to 0, \quad \text{ as } l\to \infty.
$$
Then by the definition of the ball $B_{t_l}(x,r)$,
$$
\rho_l(x, \partial \Omega)\to \infty, \quad \text{as } l\to \infty.
$$
Thus there is a greatest $k$ satisfying $\rho_k(x, \partial \Omega) \leq A_0$. By (\ref{biginsmall}), we have
$$
A_0<\rho_{k+1}(x,\partial \Omega) =\inf_{y\in \partial \Omega} \rho_{k+1}(x,y) \leq \inf_{y\in \partial \Omega} C_1\rho_k(x,y) =C_1\rho_k(x,\partial \Omega)\leq C_1A_0<A_1.
$$

Suppose in addition $|\Omega|<\infty$. Since any member $B_{t_l}(x,1)$ in $\mathcal{E}$ is contained in $\Omega$, by (\ref{area}), the indices $l$ of the members in $\mathcal{E}$ are bounded below. Pick any member $B_{t_{l_1}}(x_1, 1)\in \mathcal{E}$ with the least index $l_1$. Having chosen $B_{t_{l_i}}(x_i,1)$ for $1\leq i \leq m$, choose $B_{t_{l_{m+1}}}(x_{m+1},1)$ to be any member in $\mathcal{E}$ disjoint from $\bigcup_{1\leq i\leq m} B_{t_{l_i}}(x_i,1)$ with the lease index $l_{m+1}$. Hence if $B_{t_{l_0}}(x_0,1)\in \mathcal{E}$, then there exists $l_i\leq l_0$ such that $B_{t_{l_0}}(x_0,1)$ intersects $B_{t_{l_i}}(x_i,1)$. By (\ref{smallinbig}), $B_{t_{l_i}}(x_0,2C_0)$ intersects $B_{t_{l_i}}(x_i,2C_0)$. Then by (\ref{triangle}),
$$
B_{t_{l_0}}(x_0,1)\subseteq B_{t_{l_i}}(x_0, 2C_0)\subseteq B_{t_{l_i}}(x_i, 6C_0).
$$
And since $6C_0<A_0$,
$$
\Omega = \bigcup_i B_{t_{l_i}}(x_i, 6C_0).
$$
\end{proof}

\begin{definition}[Maximal functions]
For each $r\geq 1$, denote
$$
\mathcal{B}_r=\{B_{t_l}(x,r):l\in \mathbb{Z}, x\in \mathbb{H}^1\}.
$$
Define a maximal function for each $r\geq 1$:
$$
\mathcal{M}_{\mathcal{B}_r}f (x) = \sup_{x\in B\in \mathcal{B}_r} \frac{1}{|B|} \int_B |f(y)|\,dy, \quad \forall f\in L^1_{loc}.
$$
\end{definition}

The rest of the section is devoted to the next two lemmas, which follow similarly as in \cite{notion} and Section 2 of \cite{CARBERYCHRIST}.
\begin{lemma}[Maximal inequality]\label{maximalfunction}
For each $r\geq 1$, $\mathcal{M}_{\mathcal{B}_r}$ is of weak type $(1,1)$ and bounded on $L^p(\mathbb{H}^1)$ for $1<p\leq \infty$, with bounds depending only on $p$ and $C_0$. Specifically, for each $f\in L^1(\mathbb{H}^1)$ and each $\alpha>0$,
$$
|\{x: \mathcal{M}_{\mathcal{B}_r}f(x)>\alpha\}|\leq (6C_0)^4 \frac{\|f\|_1}{\alpha}.
$$
\end{lemma}

\begin{lemma}[Calder\'on-Zygmund theory]\label{CZ}
Suppose the convolution operator
$$
Tf(x)=\int f(x\cdot y^{-1}) K(y)\,dy.
$$
is bounded on $L^2(\mathbb{H}^1)$. If there exists $C_K>0$ such that for every $l\in \mathbb{Z}$ and every $y\in B_{t_l}(0,6C_0)$,
\begin{equation}\label{left}
\int_{B_{t_l}(0,A_0)^C} \big|K(y\cdot x)-K(x)\big|\,dx \leq C_K,
\end{equation}
and
\begin{equation}\label{right}
\int_{B_{t_l}(0,A_0)^C} \big|K(x\cdot y)-K(x)\big|\,dx \leq C_K,
\end{equation}
then $T$ is of weak type $(1,1)$ and bounded on $L^p(\mathbb{H}^1)$ for $1<p<\infty$, with bounds depending only on $\|T\|_{2\to 2}, C_K, p$, and $A_1$.
\end{lemma}

\begin{proof}[Proof of Lemma \ref{maximalfunction}]
Since $\|\mathcal{M}_{\mathcal{B}_r}f\|_\infty \leq \|f\|_\infty$, it suffices to show for every $f\in L^1(\mathbb{H}^1)$ and every $\alpha>0$,
$$
|\{x: \mathcal{M}_{\mathcal{B}_r}f(x)>\alpha\}|\leq (6C_0)^4 \frac{\|f\|_1}{\alpha}.
$$
Fix $r\geq 1$. Denote $\Omega =\{x: \mathcal{M}_{\mathcal{B}_r}f(x)>\alpha\}$. If $\Omega$ is empty, then we are done. Otherwise $\Omega$ is a nonempty open set. Fix an arbitrary compact $L\subseteq \Omega$. It suffices to show $|L|\leq (6C_0)^4 \frac{\|f\|_1}{\alpha}$. By the definition of $\mathcal{M}_{\mathcal{B}_r}$, $L$ is covered by balls in a finite subcollection $\mathcal{B}_r'\subseteq \mathcal{B}_r$, where for each $B_{t_l}(x,r)\in \mathcal{B}_r'$,
$$
\frac{1}{|B_{t_l}(x,r)|}\int_{B_{t_l}(x,r)}|f| >\alpha.
$$
Choose $B_{t_{l_1}}(x_1,r)$ to be any member of $\mathcal{B}_r'$ with least index $l_1$. Having chosen $B_{t_{l_i}}(x_i,r)$ for $1\leq i\leq m$, choose $B_{t_{l_{m+1}}}(x_{m+1},r)$ to be any member in $\mathcal{B}_r'$ disjoint from $\bigcup_{1\leq i\leq m} B_{t_{l_i}}(x_i,r)$ with least index $l_{m+1}$. Hence any member $B_{t_{l_0}}(x_0,r)$ in $\mathcal{B}_r'$ intersects a ball $B_{t_{l_i}}(x_i,r)$ with $l_i\leq l_0$. By (\ref{smallinbig}), $B_{t_{l_i}}(x_0, 2C_0r)$ intersects $B_{t_{l_i}}(x_i, 2C_0r)$. Then by (\ref{triangle}),
$$
B_{t_{l_0}}(x_0, r) \subseteq B_{t_{l_i}}(x_0, 2C_0r) \subseteq B_{t_{l_i}}(x_i, 6C_0r).
$$
Therefore
$$
L \subseteq \bigcup_i B_{t_{l_i}}(x_i, 6C_0r).
$$
Hence
\begin{align*}
|L| &\leq \sum_i |B_{t_{l_i}}(x_i, 6C_0r)|= (6C_0)^4 \sum_i |B_{t_{l_i}}(x_i, r)|\\
&<(6C_0)^4 \sum_i\frac{1}{\alpha} \int_{B_{t_{l_i}}(x_i, r)}|f| \leq (6C_0)^4 \frac{\|f\|_1}{\alpha}.
\end{align*}
\end{proof}

\begin{proof}[Proof of Lemma \ref{CZ}]
Denote by $\bar K(y)$ the complex conjugate of $K(y)$, and denote $\widetilde{\bar K}(y):= \bar K(y^{-1})$. The adjoint operator of $T$ is
$$
T^*f(x)= \int f(x\cdot y^{-1}) \widetilde{\bar K}(y)\,dy.
$$
If the kernel $K$ satisfies (\ref{right}), then the kernel $\widetilde{\bar K}$ satisfies (\ref{left}). Once we derive that $T$ is of weak type $(1,1)$ from $K$ satisfying (\ref{left}), we will have that $T^*$ is also of weak type $(1,1)$ since $K$ satisfies (\ref{right}). Then both $T, T^*$ will be bounded on $L^p(\mathbb{H}^1)$ for $1<p\leq 2$, and hence $T$ will be bounded on $L^p(\mathbb{H}^1)$ for $1<p<\infty$. Thus it suffices to derive that $T$ is of weak type $(1,1)$ from (\ref{left}).

Let $f\in L^1(\mathbb{H}^1)$ and $\alpha>0$ be given. Recall from Definition \ref{notation} that $A_1>A_0>6C_0>1$. Applying Lemma \ref{maximalfunction} with $r=A_1$, the open set
$$
\Omega=\{x: \mathcal{M}_{\mathcal{B}_{A_1}}f(x)>\alpha\}
$$
satisfies
\begin{equation}\label{regionbound}
|\Omega|\leq (6C_0)^4 \frac{\|f\|_1}{\alpha} \leq A_1^4 \frac{\|f\|_1}{\alpha}.
\end{equation}
Applying Lemma \ref{covering} to $\Omega$, there exists a sequence of disjoint balls $B_{t_{l_i}}(x_i,1)$, where for each $i$, 
$$
x_i\in \Omega, \quad A_0<\rho_{l_i}(x_i, \partial \Omega)<A_1,
$$
and
$$
\Omega = \bigcup_i B_{t_{l_i}}(x_i, 6C_0).
$$
Let $P_1= B_{t_{l_1}}(x_1, 6C_0) \backslash \bigcup_{k\neq 1} B_{t_{l_k}}(x_k, 1)$. Having chosen $P_i$ for $1\leq i \leq m$, let
$$
P_{m+1}=B_{t_{l_{m+1}}}(x_{m+1}, 6C_0)\backslash \big(P_1\cup \cdots \cup P_m \cup \big(\bigcup_{k\neq m+1} B_{t_{l_k}}(x_k, 1)\big)\big).
$$
Then the $P_i$ are disjoint,
$$
B_{t_{l_i}}(x_i, 1) \subseteq P_i \subseteq B_{t_{l_i}}(x_i, 6C_0), \quad \forall i,
$$
and
$$
\Omega = \bigcup_i P_i.
$$

Hence we can write $f=g+b$, where
$$
g(x)=
\left\{
\begin{aligned}
&f(x), \quad x\not\in \Omega,\\
&\frac{1}{|P_i|}\int_{P_i} f, \quad x\in P_i.
\end{aligned}
\right.
$$
We have $|g(x)|\leq A_1^4 \alpha$ almost everywhere. In fact, if $x\not\in \Omega$, then $\mathcal{M}_{\mathcal{B}_{A_1}}f(x)\leq \alpha$, thus by Lemma \ref{maximalfunction} and by (\ref{singleton}), $|g(x)|\leq \alpha$ almost everywhere for $x\not\in \Omega$. If $x\in P_i$, for this $i$, by Lemma \ref{covering}, there exists $y\in \partial \Omega$ with $\rho_{l_i}(x_i,y)<A_1$. Then
$$
|g(x)|\leq \frac{1}{|P_i|} \int_{P_i} |f| \leq \frac{1}{|B_{t_{l_i}}(x_i, 1)|} \int_{B_{t_{l_i}}(x_i, A_1)} |f| \leq A_1^4 \mathcal{M}_{\mathcal{B}_{A_1}}f(y)\leq A_1^4 \alpha.
$$
Thus 
$$
|\{x: Tg(x)>\frac{\alpha}{2}\}| \leq \frac{4}{\alpha^2} \|Tg\|_2^2 \leq \|T\|_{2\to 2} \frac{4}{\alpha^2} A_1^4 \alpha \|f\|_1 = 4A_1^4\|T\|_{2\to 2} \frac{\|f\|_1}{\alpha}.
$$

By (\ref{regionbound}), it suffices to bound $|\{x\not\in \Omega: |Tb(x)|>\alpha/2\}|$. Denote
$$
b_i(x)= 
\left\{
\begin{aligned}
&f(x)-\frac{1}{|P_i|} \int_{P_i} f, \quad x\in P_i,\\
&0, \quad x\not\in P_i.
\end{aligned}
\right.
$$
Then $b=\sum_i b_i$, $\int b_i =0$, and $\sum_i \|b_i\|_1 \leq 2\|f\|_1$. Since $B_{t_{l_i}}(x_i, A_0) \subseteq \Omega$, by (\ref{left}),
\begin{align*}
&\quad |\{x\not\in \Omega: |Tb(x)|>\alpha/2\}| \leq \frac{2}{\alpha} \int_{\Omega^C} |Tb| \leq \frac{2}{\alpha} \sum_i \int_{B_{t_{l_i}}(x_i, A_0)^C}|Tb_i|\\
&= \frac{2}{\alpha} \sum_i \int_{B_{t_{l_i}}(x_i, A_0)^C}\Big|\int b_i(x\cdot y^{-1})K(y)\,dy\Big|\,dx\\
&= \frac{2}{\alpha} \sum_i \int_{B_{t_{l_i}}(x_i, A_0)^C}\Big|\int b_i(y)\big(K(y^{-1}\cdot x)-K(x_i^{-1}\cdot x)\big)\,dy\Big|\,dx\\
&\leq \frac{2}{\alpha} \sum_i \int_{B_{t_{l_i}}(x_i, 6C_0)} |b_i(y)| \int_{B_{t_{l_i}}(x_i, A_0)^C} \big|K((y^{-1}\cdot x_i)\cdot (x_i^{-1}\cdot x))-K(x_i^{-1}\cdot x)\big)\big|\,dx\,dy\\
&= \frac{2}{\alpha} \sum_i \int_{B_{t_{l_i}}(0, 6C_0)} |b_i(y)| \int_{B_{t_{l_i}}(0, A_0)^C} \big|K(y'\cdot x')-K(x')\big)\big|\,dx'\,dy'\\
&\leq 4C_K \frac{\|f\|_1}{\alpha}.
\end{align*}
\end{proof}

\section{Littlewood-Paley theory}\label{littlewood}
In this section, we will determine a sequence of convolution operators $\{\mathcal{D}_j\}_{j\in \mathbb{N}}$ and apply Lemma \ref{CZ}. This will establish a Littlewood-Paley theory by the argument in Section 11 of Street and Stein \cite{LP}. Such Littlewood-Paley theory will help bound $\mathcal{M}_\Gamma$ and $\mathcal{H}_\Gamma$ on $L^p(\mathbb{H}^1)$ in Sections \ref{section4} and \ref{finalsection}.

\begin{definition}\label{jmatrix} 
Associated with each $j\in \mathbb{N}$, we denote a basis of left-invariant vector fields on $\mathbb{H}^1$:
$$
X_j = 2^{-j}X + \alpha(2^{-j}) Y +\bar \beta(2^{-j})T, \quad Y_j = h_2(2^{-j}) Y + \bar h_\beta(2^{-j})T, \quad T_j = \big(2^{-j}h_2(2^{-j}) + \bar h_3(2^{-j})\big)T.
$$ 
Denote
$$
\widetilde A_j = 
\begin{pmatrix}
2^{-j} & 0 & 0\\
\alpha(2^{-j}) & h_2(2^{-j}) & 0\\
\bar \beta(2^{-j}) & \bar h_\beta (2^{-j}) & 2^{-j}h_2(2^{-j}) + \bar h_3(2^{-j})
\end{pmatrix},
$$
and
$$
\xi_j= \frac{1}{2}\frac{2^{-j}h_2(2^{-j})}{2^{-j}h_2(2^{-j})+\bar h_3(2^{-j})}\leq \frac{1}{2}, \qquad S_j = 2^{-j}h_2(2^{-j})\big(2^{-j}h_2(2^{-j}) + \bar h_3(2^{-j})\big).
$$ 
\end{definition}
We have
\begin{equation}\label{xi}
[X_j, Y_j] =2\xi_jT_j, \qquad \det \widetilde A_j = S_j.
\end{equation}

%Associated to any sequence 
\begin{definition}\label{auxiliary}
Fix a function $\varsigma_0(y)\in C_c^\infty[-1,1]^3$ with $\int \varsigma_0(y)\,dy=1$.% and $\varsigma_0(y)\geq 0$. 
Let
$$
\varsigma_j(y)=
\left\{
\begin{aligned}
& \varsigma_0(y), \quad j=0,\\
& \varsigma_0(y)- \frac{S_j}{S_{j-1}} \varsigma_0\big( \widetilde A_{j-1}^{-1} \widetilde A_j y\big), \quad j>0.
\end{aligned}
\right.
$$ 
Define a sequence of convolution operators
\begin{align*}
\mathcal{D}_jf(x) &=\frac{1}{S_j}\int f(x\cdot y^{-1}) \varsigma_j(\widetilde A_j^{-1}y)\,dy \\
&= \int f\big(x\cdot e^{-y_1X_j - y_2 Y_j - y_3 T_j}0\big) \varsigma_j(y)\,dy, \quad j\in \mathbb{N}.
\end{align*}
\end{definition}

It is clear from the definition that
$$
\int \varsigma_j(y)\,dy=0, \quad \text{for } j>0.
$$
By (\ref{rivere}), 
$$
\|\widetilde A_{j-1}^{-1} \widetilde A_j\| \leq C_0 \big( \frac{2^{-j}}{2^{-j+1}}\big)^\epsilon \leq C_0.
$$
And since $\frac{S_j}{S_{j-1}} \leq 1$, the $\varsigma_j(y)$ have all the derivatives bounded uniformly in $j$. As functions of $y$,
\begin{equation}\label{support}
\begin{aligned}
\text{supp}\,\varsigma_j(y) &\subseteq 
\left\{
\begin{aligned}
&\text{supp}\,\varsigma_0(y), \quad j=0,\\
&\text{supp}\, \varsigma_0(y) \bigcup \text{supp}\,\varsigma_0\big(\widetilde A_{j-1}^{-1} \widetilde A_j y\big), \quad j>0,
\end{aligned}
\right.
\\
&\subseteq
\left\{
\begin{aligned}
&[-1,1]^3, \quad j=0,\\
&[-1,1]^3 \bigcup \big(\widetilde A_j^{-1} \widetilde A_{j-1} [-1,1]^3\big), \quad j>0.
\end{aligned}
\right.
\end{aligned}
\end{equation}

The Fourier transform of $\varsigma_j(y)$ on $\mathbb{R}^3$ satisfies
$$
\hat \varsigma_j\big((\widetilde A_j)^T \xi\big)=
\left\{
\begin{aligned}
& \hat \varsigma_0\big((\widetilde A_0)^T\xi\big), \quad j=0,\\
& \hat \varsigma_0\big((\widetilde A_j)^T\xi\big) - \hat \varsigma_0\big((\widetilde A_{j-1})^T\xi\big), \quad j>0,
\end{aligned}
\right.
$$ 
and $\hat \varsigma_0(0)=\int \varsigma_0=1$.
For every $\xi\in\mathbb{R}^3$, by (\ref{rivere}),
$$
|(\widetilde A_j)^T \xi| \leq \|(\widetilde A_j)^T (\widetilde A_0^T)^{-1} \| \cdot |(\widetilde A_0)^T \xi| \lesssim C_0 2^{-\epsilon j} |\xi|  \to 0, \quad \text{ as } j\to \infty.
$$
Then we have
$$
\sum_{j=0}^\infty \hat \varsigma_j\big((\widetilde A_j)^T \xi \big)= \lim_{j\to \infty} \hat \varsigma_0\big((\widetilde A_j)^T \xi \big) =\hat \varsigma_0(0)=1.
$$
Hence
$$
\sum_{j=0}^\infty \frac{1}{S_j} \varsigma_j(\widetilde A_j^{-1}y) = \delta_0(y),
$$
where $\delta_0(y)$ denotes the Dirac delta function at $0$ in $\mathbb{R}^3$. Therefore
\begin{equation}\label{identity}
\sum_{j=0}^\infty \mathcal{D}_j = \text{Id},
\end{equation}
where Id denotes the identity operator. 

Note $\|\varsigma_0\|_1$ is a fixed finite number. For $j>1$,
$$
\|\varsigma_j\|_1 \leq \|\varsigma_0\|_1 + \frac{S_j}{S_{j-1}} \|\varsigma_0\big(\widetilde A_{j-1}^{-1} \widetilde A_j \cdot \big)\|_1 = 2\|\varsigma_0\|_1.
$$
Hence $\|\mathcal{D}_j\|_{p\to p} \lesssim 1$ for every $1\leq p\leq \infty$, where the implicit constant is independent of $j$. We will use $\{\mathcal{D}_j\}_{j=0}^\infty$ to establish a Littlewood-Paley theory. 

We first need the following uniform ``regularily'' of the $\varsigma_j$:
\begin{lemma}\label{ldeltaone}
For every $j\in \mathbb{N}$ and every $z\in \mathbb{R}^3$ with $|z|\lesssim 1$,
\begin{align*}
\int \big|\varsigma_j\big(y_1+z_1, y_2+z_2, y_3+z_3+\xi_j(y_1z_2-y_2z_1)\big)-\varsigma_j(y)\big|\,dy\lesssim |z|,\\
\int \big|\varsigma_j\big(z_1+y_1, z_2+y_2, z_3+y_3+\xi_j(z_1y_2-z_2y_1)\big)-\varsigma_j(y)\big|\,dy\lesssim |z|,  
\end{align*}
where the implicit constants are independent of $j$.
\end{lemma}

\begin{proof}
We prove only the first inequality; the proof of the second one is similar. We first show 
$$
\int \big|\varsigma_0\big(y_1+z_1, y_2+z_2, y_3+z_3+\xi_j(y_1z_2-y_2z_1)\big)-\varsigma_0(y)\big|\,dy\lesssim |z|.
$$
As a function of $y$,
\begin{align*}
&\quad \text{supp}\,|\varsigma_0\big(y_1+z_1, y_2+z_2, y_3+z_3+\xi_j(y_1z_2-y_2z_1)\big)-\varsigma_0(y)\big|\\
&\subseteq \text{supp}\,\varsigma_0\big(y_1+z_1, y_2+z_2, y_3+z_3+\xi_j(y_1z_2-y_2z_1)\big) \bigcup \text{supp}\,\varsigma_0(y)\\
&\subseteq [-1-|z|, 1+|z|]^2 \times [-(1+|z|)^2, (1+|z|)^2].
\end{align*}
Thus for $|z|\lesssim 1$,
\begin{equation}\label{varsigma0right}
\begin{aligned}
&\quad \int \big|\varsigma_0\big(y_1+z_1, y_2+z_2, y_3+z_3+\xi_j(y_1z_2-y_2z_1)\big)-\varsigma_0(y)\big|\,dy\\
&\lesssim \int_{[-1-|z|, 1+|z|]^2 \times [-(1+|z|)^2, (1+|z|)^2]} \big|\big(y_1+z_1, y_2+z_2, y_3+z_3+\xi_j(y_1z_2-y_2z_1)\big)-y\big|\,dy\\
&\leq \int_{[-1-|z|, 1+|z|]^2 \times [-(1+|z|)^2, (1+|z|)^2]} \big(|z|+\xi_j|y_1z_2-y_2z_1|\big)\,dy\\
&\lesssim |z|,
\end{aligned}
\end{equation}
and similarly,
\begin{equation}\label{varsigma0left}
\int \big|\varsigma_0\big(z_1+y_1, z_2+y_2, z_3+y_3+\xi_j(z_1y_2-z_2y_1)\big)-\varsigma_0(y)\big|\,dy \lesssim |z|.
\end{equation}
By taking $\xi_j=\xi_0$, we obtain the desired inequality for $j=0$.

For $j>0$, by the definition of $\varsigma_j$, 
\begin{align*}
&\quad \int \big|\varsigma_j\big(y_1+z_1, y_2+z_2, y_3+z_3+\xi_j(y_1z_2-y_2z_1)\big)-\varsigma_j(y)\big|\,dy\\
&\leq \int \big|\varsigma_0\big(y_1+z_1, y_2+z_2, y_3+z_3+\xi_j(y_1z_2-y_2z_1)\big)-\varsigma_0(y)\big|\,dy\\
&\quad + \frac{S_j}{S_{j-1}}\int \big|\varsigma_0\big(\widetilde A_{j-1}^{-1} \widetilde A_j\big(y_1+z_1, y_2+z_2, y_3+z_3+\xi_j(y_1z_2-y_2z_1)\big)\big)-\varsigma_0\big(\widetilde A_{j-1}^{-1}\widetilde A_j y\big)\big|\,dy\\
&\lesssim |z|+ \frac{S_j}{S_{j-1}}\int \big|\varsigma_0\big(\widetilde A_{j-1}^{-1} \widetilde A_j\big(y_1+z_1, y_2+z_2, y_3+z_3+\xi_j(y_1z_2-y_2z_1)\big)\big)-\varsigma_0\big(\widetilde A_{j-1}^{-1}\widetilde A_j y\big)\big|\,dy.
\end{align*}
As a function of $y$,
\begin{align*}
\text{supp}\,\varsigma_0\big(\widetilde A_{j-1}^{-1}\widetilde A_j y\big)\subseteq \widetilde A_j^{-1}\widetilde A_{j-1}[-1,1]^3.
\end{align*}
We have
$$
\big|\text{supp}\,\varsigma_0\big(\widetilde A_{j-1}^{-1}\widetilde A_j y\big)\big|\leq \big|\widetilde A_j^{-1}\widetilde A_{j-1}[-1,1]^3\big|=8\frac{\det \widetilde A_{j-1}}{\det \widetilde A_j}=8\frac{S_{j-1}}{S_j}.
$$
And since the Jacobian of the map $y\mapsto (y_1+z_1, y_2+z_2, y_3+z_3+\xi_j(y_1z_2-y_2z_1))$ is $1$, we also have
$$
\big|\text{supp}\,\varsigma_0\big(\widetilde A_{j-1}^{-1} \widetilde A_j\big(y_1+z_1, y_2+z_2, y_3+z_3+\xi_j(y_1z_2-y_2z_1)\big)\big)\big| \leq 8\frac{S_{j-1}}{S_j}.
$$
Thus by denoting $g(y):=\varsigma_0(\widetilde A_{j-1}^{-1} \widetilde A_j y)$,
\begin{align*}
&\quad \frac{S_j}{S_{j-1}}\int \big|\varsigma_0\big(\widetilde A_{j-1}^{-1} \widetilde A_j\big(y_1+z_1, y_2+z_2, y_3+z_3+\xi_j(y_1z_2-y_2z_1)\big)\big)-\varsigma_0\big(\widetilde A_{j-1}^{-1}\widetilde A_j y\big)\big|\,dy\\
&\lesssim \sup_y \big|\varsigma_0\big(\widetilde A_{j-1}^{-1} \widetilde A_j\big(y_1+z_1, y_2+z_2, y_3+z_3+\xi_j(y_1z_2-y_2z_1)\big)\big)-\varsigma_0\big(\widetilde A_{j-1}^{-1}\widetilde A_j y\big)\big|\\
&\leq \sup_y \int_0^1 \Big| \frac{d}{d\theta} \varsigma_0\big(\widetilde A_{j-1}^{-1} \widetilde A_j\big(y_1+\theta z_1, y_2+\theta z_2, y_3+\theta z_3+\theta \xi_j(y_1z_2-y_2z_1)\big)\big)\Big|\,d\theta\\
&\leq \sup_{y, \theta} \Big| z_1\partial_1g \big(y_1+\theta z_1, y_2+\theta z_2, y_3+\theta z_3+\theta \xi_j(y_1z_2-y_2z_1)\big) \\
&\qquad + z_2\partial_2 g\big(y_1+\theta z_1, y_2+\theta z_2, y_3+\theta z_3+\theta \xi_j(y_1z_2-y_2z_1)\big)\\
&\qquad + \big(z_3+\xi_j(y_1z_2-y_2z_1)\big) \partial_3 g\big(y_1+\theta z_1, y_2+\theta z_2, y_3+\theta z_3+\theta \xi_j(y_1z_2-y_2z_1)\big)\Big|\\
&= \sup_{y,\theta} \Big| z_1\big(\partial_1-\xi_j(y_2+\theta z_2)\partial_3\big)g\big(y_1+\theta z_1, y_2+\theta z_2, y_3+\theta z_3+\theta \xi_j(y_1z_2-y_2z_1)\big)\\
&\qquad + z_2\big(\partial_2 +\xi_j(y_1+\theta z_1)\partial_3\big) g\big(y_1+\theta z_1, y_2+\theta z_2, y_3+\theta z_3+\theta \xi_j(y_1z_2-y_2z_1)\big)\\
&\qquad + z_3\partial_3 g\big(y_1+\theta z_1, y_2+\theta z_2, y_3+\theta z_3+\theta \xi_j(y_1z_2-y_2z_1)\big)\Big|\\
&\leq \sup_y|z_1 (\partial_1 -\xi_j y_2\partial_3)g(y)| + \sup_y |z_2(\partial_2 + \xi_j y_1\partial_3)g(y)| + \sup_y |z_3\partial_3 g(y)|.
\end{align*}
By Definition \ref{matrix}, for $y\in \text{supp}\,g \subseteq \widetilde A_j^{-1} \widetilde A_{j-1} [-1,1]^3$,
\begin{align*}
|y_1|&\leq 2,\\
|y_2|&\leq \frac{|\alpha(2^{-j+1})-2\alpha(2^{-j})|}{h_2(2^{-j})} + \frac{h_2(2^{-j+1})}{h_2(2^{-j})}=\frac{2^{-j+1}}{h_2(2^{-j})}\Big|\int_{2^{-j}}^{2^{-j+1}} \big(\frac{\alpha(s)}{s}\big)'\,ds\Big| +  \frac{h_2(2^{-j+1})}{h_2(2^{-j})}\\
&=\frac{2^{-j+1}}{h_2(2^{-j})}\Big|\int_{2^{-j}}^{2^{-j+1}} \frac{h_2(s)}{s^2}\,ds\Big| +  \frac{h_2(2^{-j+1})}{h_2(2^{-j})}\leq 3 \frac{h_2(2^{-j+1})}{h_2(2^{-j})}.
\end{align*}
By Definition \ref{matrix}, 
$$
|\partial_3 g(y)| = \big|(\partial_3\varsigma_0)\big(\widetilde A_{j-1}^{-1} \widetilde A_j y\big)\big|\cdot \frac{2^{-j}h_2(2^{-j}) + \bar h_3(2^{-j})}{2^{-j+1}h_2(2^{-j+1})+\bar h_3(2^{-j+1})} \lesssim \frac{2^{-j}h_2(2^{-j}) + \bar h_3(2^{-j})}{2^{-j+1}h_2(2^{-j+1})+\bar h_3(2^{-j+1})} \leq 1. 
$$
Thus
$$
|\xi_j y_2 \partial_3 g| \lesssim \frac{2^{-j}h_2(2^{-j})}{2^{-j}h_2(2^{-j})+\bar h_3(2^{-j})} \frac{h_2(2^{-j+1})}{h_2(2^{-j})} \frac{2^{-j}h_2(2^{-j}) + \bar h_3(2^{-j})}{2^{-j+1}h_2(2^{-j+1})+\bar h_3(2^{-j+1})} \leq 1.
$$
By (\ref{rivere}),
$$
|\nabla g| \leq \|\widetilde A_{j-1}^{-1} \widetilde A_j\| \cdot |\nabla \varsigma_0| \leq C_0|\nabla \varsigma_0|\lesssim 1.
$$
Therefore
\begin{align*}
&\quad \int \big|\varsigma_j\big(y_1+z_1, y_2+z_2, y_3+z_3+\xi_j(y_1z_2-y_2z_1)\big)-\varsigma_j(y)\big|\,dy\\
&\lesssim |z| + \sup_y|z_1 (\partial_1 -\xi_j y_2\partial_3)g(y)| + \sup_y |z_2(\partial_2 + \xi_j y_1\partial_3)g(y)| + \sup_y |z_3\partial_3 g(y)|\\
&\lesssim |z|.
\end{align*}
%\begin{align*} &\quad \text{supp}\,\big|\varsigma_0\big(\widetilde A_{j-1}^{-1} \widetilde A_j\big(y_1+z_1, y_2+z_2, y_3+z_3+\xi_j(y_1z_2-y_2z_1)\big)\big)-\varsigma_0\big(\widetilde A_{j-1}^{-1}\widetilde A_j y\big)\big|\\ &\subseteq \text{supp}\,\varsigma_0\big(\widetilde A_{j-1}^{-1} \widetilde A_j\big(y_1+z_1, y_2+z_2, y_3+z_3+\xi_j(y_1z_2-y_2z_1)\big)\big) \bigcup\text{supp}\,\varsigma_0\big(\widetilde A_{j-1}^{-1}\widetilde A_j y\big) \end{align*}
\end{proof}

\begin{remark}\label{alsoldeltaone}
By the Baker-Campbell-Hausdorff formula and by (\ref{xi}),
$$
(\widetilde A_j y) \cdot (\widetilde A_j z) = \widetilde A_j \big(y_1+z_1, y_2+z_2, y_3+z_3 + \xi_j(y_1z_2-y_2z_1)\big).
$$
Therefore (\ref{varsigma0right}), (\ref{varsigma0left}), and Lemma \ref{ldeltaone} can be written as: for every $j\in \mathbb{N}$ and every $z\in \mathbb{R}^3$ with $|\widetilde A_j^{-1} z|\lesssim 1$, 
\begin{align*}
\frac{1}{S_j}\int \big| \varsigma_0\big(\widetilde A_j^{-1} (y\cdot z)\big) - \varsigma_0\big(\widetilde A_j^{-1} y\big) \big|\,dy \lesssim \big|\widetilde A_j^{-1} z\big|,\\
\frac{1}{S_j}\int \big| \varsigma_0\big(\widetilde A_j^{-1} (z\cdot y)\big) - \varsigma_0\big(\widetilde A_j^{-1} y\big) \big|\,dy \lesssim \big|\widetilde A_j^{-1} z\big|,\\
\frac{1}{S_j}\int \big| \varsigma_j\big(\widetilde A_j^{-1} (y\cdot z)\big) - \varsigma_j\big(\widetilde A_j^{-1} y\big) \big|\,dy \lesssim \big|\widetilde A_j^{-1} z\big|,\\
\frac{1}{S_j}\int \big| \varsigma_j\big(\widetilde A_j^{-1} (z\cdot y)\big) - \varsigma_j\big(\widetilde A_j^{-1} y\big) \big|\,dy \lesssim \big|\widetilde A_j^{-1} z\big|,
\end{align*}
where the implicit constants are independent of $j$.
\end{remark}

\begin{lemma}\label{cotlarstein}
There exists $\epsilon>0$ (by shrinking the $\epsilon$ in (\ref{rivere})) such that for every $j,k\in \mathbb{N}$,
$$
\|\mathcal{D}_k \mathcal{D}_j^* \|_{2 \to 2}, \quad \|\mathcal{D}_k^* \mathcal{D}_j \|_{2 \to 2}, \quad \|\mathcal{D}_k \mathcal{D}_j \|_{2 \to 2}, \quad \|\mathcal{D}_k^* \mathcal{D}_j^* \|_{2 \to 2} \lesssim 2^{-\epsilon|j-k|},
$$
where the implicit constant is independent of $j,k$.
\end{lemma}

\begin{proof}
\textbf{Case } $\mathbf{k=j}$: due to $\|\mathcal{D}_j\|_{2\to 2}, \|\mathcal{D}_k\|_{2\to 2} \lesssim 1$.

\textbf{Case } $\mathbf{k>j}$: we have $k>0$ and thus $\int \varsigma_k(y)\,dy =0$. If $z\in \text{supp}\,\varsigma_k$, by (\ref{support}),
$$
z\in [-1,1]^3 \bigcup \big(\widetilde A_k^{-1} \widetilde A_{k-1} [-1,1]^3\big).
$$
Denote by $\tilde z$ the point satisfying
$$
z_1 X_k + z_2 Y_k + z_3 T_k = \tilde z_1 X_j + \tilde z_2 Y_j + \tilde z_3 T_j.
$$
Then $\widetilde A_k z = \widetilde A_j \tilde z$. Thus
$$
\tilde z= \widetilde A_j^{-1} \widetilde A_k z \in \big(\widetilde A_j^{-1} \widetilde A_k[-1,1]^3\big) \bigcup \big(\widetilde A_j^{-1} \widetilde A_k \widetilde A_k^{-1} \widetilde A_{k-1} [-1,1]^3\big) =\big(\widetilde A_j^{-1} \widetilde A_k[-1,1]^3\big) \bigcup \big(\widetilde A_j^{-1} \widetilde A_{k-1}[-1,1]^3\big).
$$
Hence by (\ref{rivere}),
$$
|\tilde z|\lesssim 2^{-\epsilon|j-k|}.
$$
Since $\int \varsigma_k=0$, by  the Baker-Campbell-Hausdorff formula and (\ref{xi}), we have
\begin{align*}
&\quad |\mathcal{D}_k^* \mathcal{D}_jf(x)|\\ &=\frac{1}{S_jS_k} \Big|\iint f(x\cdot z \cdot y^{-1})\varsigma_j(\widetilde A_j^{-1}y) \bar \varsigma_k(\widetilde A_k^{-1}z) \,dy\,dz\Big|\\
&= \Big|\iint f(x\cdot e^{z_1 X_k + z_2 Y_k + z_3 T_k} e^{-y_1X_j - y_2 Y_j - y_3 T_j}0) \varsigma_j(y) \bar \varsigma_k(z)\,dy\,dz\Big| \\
&= \Big|\iint f(x\cdot e^{(\tilde z_1-y_1)X_j + (\tilde z_2-y_2)Y_j + \big(\tilde z_3 - y_3 -\xi_j (\tilde z_1 y_2 -\tilde z_2 y_1)\big)T_j}0)\varsigma_j(y) \bar \varsigma_k(z)\,dy\,dz\\
&\quad  - \iint f(x\cdot e^{-y_1X_j -y_2Y_j  - y_3 T_j}0)\varsigma_j(y) \bar \varsigma_k(z)\,dy\,dz\Big|\\
&\lesssim \sup_{z\in \text{supp}\,\varsigma_k}\Big|\int f(x\cdot e^{-y_1X_j -y_2Y_j  - y_3 T_j}0) \big(\varsigma_j\big(y_1 + \tilde z_1, y_2 + \tilde z_2, y_3 + \tilde z_3+ \xi_j(y_1\tilde z_2-y_2\tilde z_2)\big) -\varsigma_j(y)\big)\,dy \Big|\\
&\leq \|f\|_\infty \sup_{z\in \text{supp}\,\varsigma_k} \int \big|\varsigma_j\big(y_1 + \tilde z_1, y_2 + \tilde z_2, y_3 + \tilde z_3+ \xi_j(y_1\tilde z_2-y_2\tilde z_2)\big) -\varsigma_j(y)\big|\,dy \\
&\lesssim 2^{-\epsilon|j-k|} \|f\|_\infty,
\end{align*}
where the last inequality is due to Lemma \ref{ldeltaone}. Thus $\|\mathcal{D}_k^*\mathcal{D}_j\|_{\infty \to \infty} \lesssim 2^{-\epsilon|j-k|}$. Interpolate with  $\|\mathcal{D}_k^*\mathcal{D}_j\|_{1\to 1} \lesssim 1$, we have
$$
\|\mathcal{D}_k^* \mathcal{D}_j \|_{2\to 2} \lesssim 2^{-\epsilon|j-k|},
$$
where the value of $\epsilon$ changes but is still independent of $j,k$. Similarly, 
$$
\|\mathcal{D}_k \mathcal{D}_j^* \|_{2\to 2}, \quad \|\mathcal{D}_k \mathcal{D}_j \|_{2\to 2}, \quad \|\mathcal{D}_k^* \mathcal{D}_j^* \|_{2\to 2} \lesssim 2^{-\epsilon|j-k|}.
$$

\textbf{Case } $\mathbf{k<j}$: this is obtained by taking the adjoint of the operators $\mathcal{D}_k^*\mathcal{D}_j$, $\mathcal{D}_k\mathcal{D}_j^*$, $\mathcal{D}_k\mathcal{D}_j$, $\mathcal{D}_k^*\mathcal{D}_j^*$ in Case $k>j$.
\end{proof}

Define 
$$
T=\sum_{j=0}^N \sigma_j \mathcal{D}_j,
$$
where $N\in \mathbb{N}\cup\{\infty\}$, and the $\sigma_j= \pm 1$ are the Rademacher functions. By the Cotlar-Stein Lemma, $T$ is bounded on $L^2(\mathbb{H}^1)$, with a bound independent of the choice of $N$ and $\{\sigma_j\}$. 

We now apply Lemma \ref{CZ} and obtain the following Littlewood-Paley inequality:

\begin{proposition}\label{littlewoodpaley}
For every $1<p<\infty$,
\begin{align*}
\Big\|\Big(\sum_{j=0}^\infty |\mathcal{D}_jf|^2\Big)^{\frac{1}{2}} \Big\|_p \lesssim \|f\|_p,\quad \forall f\in L^p(\mathbb{H}^1)\\
\Big\|\Big(\sum_{j=0}^\infty |\mathcal{D}_j^*f|^2\Big)^{\frac{1}{2}} \Big\|_p \lesssim \|f\|_p, \quad \forall f\in L^p(\mathbb{H}^1).
\end{align*}
%and $$ \Big\| \sum_{j=0}^\infty \mathcal{D}_jf_j\Big\|_p \lesssim_p \Big\|\Big(\sum_{j=0}^\infty |\mathcal{D}_jf_j|^2\Big)^{\frac{1}{2}}\Big\|_p. $$
\end{proposition}

\begin{proof}
We only prove the first inequality; the second one follows similarly. Once we prove $T$ is bounded on $L^p(\mathbb{H}^1)$ for $1<p<\infty$ with bounds independent of the choice of $N$ and $\{\sigma_j\}$, by Khintchine's inequality, for every $0<N'<\infty$,
\begin{align*}
\Big\| \Big(\sum_{j=0}^{N'} \big|\mathcal{D}_jf\big|^2\Big)^{\frac{1}{2}} \Big\|_p \lesssim \Big\| \Big( \mathbb{E}\Big|\sum_{j=1}^{N'} \sigma_j\mathcal{D}_jf\Big|^p\Big)^{\frac{1}{p}} \Big\|_p = \big(\mathbb{E}\|Tf\|_p^p \big)^{\frac{1}{p}} \lesssim \|f\|_p,
\end{align*}
where the implicit constants are independent of $N'$. Then by the monotone convergence theorem,
$$
\Big\| \Big(\sum_{j=0}^\infty \big|\mathcal{D}_jf\big|^2\Big)^{\frac{1}{2}} \Big\|_p \lesssim \|f\|_p.
$$

We have
$$
Tf(x) = \sum_{j=0}^N \sigma_j \mathcal{D}_jf(x)= \int f(x\cdot y^{-1}) \sum_{j=0}^N \frac{\sigma_j}{S_j} \varsigma_j(\widetilde A_j^{-1} y)\,dy.
$$
By Lemma \ref{CZ}, it suffices to verify the kernel 
$$
K(y):=\sum_{j=0}^N \frac{\sigma_j}{S_j} \varsigma_j\big(\widetilde A_j^{-1} y\big)
$$
satisfies (\ref{left}) and (\ref{right}). Fix arbitrary $l\in \mathbb{Z}$ and arbitrary $y\in B_{t_l}(0, 6C_0)$, we need to bound the following:
\begin{align*}
&\quad \int_{B_{t_l}(0,A_0)^C} |K(y\cdot x)-K(x)|\,dx + \int_{B_{t_l}(0,A_0)^C} |K(x\cdot y)-K(x)|\,dx\\
&\leq \sum_{j=0}^\infty \int_{B_{t_l}(0,A_0)^C} \Big(\frac{1}{S_j} \big|\varsigma_j\big(\widetilde A_j^{-1}(y\cdot x)\big)-\varsigma_j\big(\widetilde A_j^{-1} x\big)\big| + \frac{1}{S_j} \big|\varsigma_j\big(\widetilde A_j^{-1}(x\cdot y)\big)-\varsigma_j\big(\widetilde A_j^{-1} x\big)\big|\Big)\,dx \\
&\leq \int_{B_{t_l}(0,A_0)^C} \Big(\frac{1}{S_0} \big|\varsigma_0\big(\widetilde A_0^{-1}(y\cdot x)\big)-\varsigma_0\big(\widetilde A_0^{-1} x\big)\big| + \frac{1}{S_0} \big|\varsigma_0\big(\widetilde A_0^{-1}(x\cdot y)\big)-\varsigma_0\big(\widetilde A_0^{-1} x\big)\big|\Big)\,dx\\
&\quad +\sum_{j>0} \int_{B_{t_l}(0,A_0)^C} \Big(\frac{1}{S_j} \big|\varsigma_0\big(\widetilde A_j^{-1}(y\cdot x)\big)-\varsigma_0\big(\widetilde A_j^{-1} x\big)\big| + \frac{1}{S_{j-1}} \big|\varsigma_0\big(\widetilde A_{j-1}^{-1}(y\cdot x)\big)-\varsigma_0\big(\widetilde A_{j-1}^{-1} x\big)\big|\\
&\qquad \qquad \qquad \qquad +\frac{1}{S_j} \big|\varsigma_0\big(\widetilde A_j^{-1}(x\cdot y)\big)-\varsigma_0\big(\widetilde A_j^{-1} x\big)\big|+\frac{1}{S_{j-1}} \big|\varsigma_0\big(\widetilde A_{j-1}^{-1}(x\cdot y)\big)-\varsigma_0\big(\widetilde A_{j-1}^{-1} x\big)\big|\Big)\,dx \\
&=2\sum_{j:2^{-j}\geq t_l} \frac{1}{S_j} \int_{B_{t_l}(0,A_0)^C} \Big( \big|\varsigma_0\big(\widetilde A_j^{-1}(y\cdot x)\big)-\varsigma_0\big(\widetilde A_j^{-1} x\big)\big| + \big|\varsigma_0\big(\widetilde A_j^{-1}(x\cdot y)\big)-\varsigma_0\big(\widetilde A_j^{-1} x\big)\big|\Big)\,dx\\
&\quad + 2\sum_{j:2^{-j}< t_l} \frac{1}{S_j} \int_{B_{t_l}(0,A_0)^C} \Big( \big|\varsigma_0\big(\widetilde A_j^{-1}(y\cdot x)\big)-\varsigma_0\big(\widetilde A_j^{-1} x\big)\big| + \big|\varsigma_0\big(\widetilde A_j^{-1}(x\cdot y)\big)-\varsigma_0\big(\widetilde A_j^{-1} x\big)\big|\Big)\,dx.
\end{align*}

Denote 
$$
\tilde y= \widetilde A(t_l)^{-1} y\in \widetilde A(t_l)^{-1}B_{t_l}(0,6C_0)= [-6C_0, 6C_0]^2\times [-36C_0^2, 36C_0^2].
$$
For $2^{-j} \geq t_l$, by (\ref{rivere}),
$$
\big|\widetilde A_j^{-1} y \big|= \big|\widetilde A_j^{-1} \widetilde A(t_l) \tilde y\big| \lesssim \big(\frac{t_l}{2^{-j}}\big)^\epsilon \leq 1.
$$
By Remark \ref{alsoldeltaone},
\begin{align*}
&\quad \sum_{j:2^{-j}\geq t_l} \frac{1}{S_j}\int_{B_{t_l}(0,A_0)^C} \Big(\big|\varsigma_0\big(\widetilde A_j^{-1}(y\cdot x)\big)-\varsigma_0\big(\widetilde A_j^{-1} x\big)\big| +\big|\varsigma_0\big(\widetilde A_j^{-1}(x\cdot y)\big)-\varsigma_0\big(\widetilde A_j^{-1} x\big)\big|\Big)\,dx \\
&\lesssim \sum_{j:2^{-j}\geq t_l} \big|\widetilde A_j^{-1} y \big|\lesssim 1.
\end{align*}

Recall from Definition \ref{notation} that $A_0=8C_0$. By (\ref{invariantmetric}), if $\rho_l(0,x)\geq A_0$ and $\rho_l(0,y)<6C_0$, then 
$$
\rho_l(x,x\cdot y) = \rho_l(0,y)<6C_0, \quad \rho_l(y, y\cdot x)=\rho_l(0,x) \geq A_0,
$$
and thus 
\begin{align*}
\rho_l(0,x\cdot y) &\geq \rho_l(0, x) - \rho_l(x, x\cdot y) > A_0-6C_0=2C_0,\\
\rho_l(0,y\cdot x) &\geq \rho_l(y, y\cdot x)-\rho_l(0,y) > A_0-6C_0 = 2C_0.
\end{align*}
Therefore
\begin{align*}
&\quad \sum_{j:2^{-j}< t_l} \frac{1}{S_j} \int_{B_{t_l}(0,A_0)^C} \Big( \big|\varsigma_0\big(\widetilde A_j^{-1}(y\cdot x)\big)-\varsigma_0\big(\widetilde A_j^{-1} x\big)\big| + \big|\varsigma_0\big(\widetilde A_j^{-1}(x\cdot y)\big)-\varsigma_0\big(\widetilde A_j^{-1} x\big)\big|\Big)\,dx\\
&\leq \sum_{j:2^{-j}< t_l} \int_{B_{t_l}(0,A_0)^C} \frac{\big|\varsigma_0\big(\widetilde A_j^{-1}(y\cdot x)\big)\big|+\big|\varsigma_0\big(\widetilde A_j^{-1} x\big)\big|+ \big|\varsigma_0\big(\widetilde A_j^{-1}(x\cdot y)\big)\big|+\big|\varsigma_0\big(\widetilde A_j^{-1} x\big)\big|}{S_j} \,dx\\
&\leq 4\sum_{j:2^{-j}< t_l} \int_{B_{t_l}(0,2C_0)^C} \frac{\big|\varsigma_0\big(\widetilde A_j^{-1} x\big)\big|}{S_j} \,dx\\
&=0,
\end{align*}
where the last equality follows because for $2^{-j}< t_l$,
$$
\text{supp}\,\varsigma_0\big(\widetilde A_j^{-1} \cdot \big) \subseteq \widetilde A_j[-1,1]^3 = B_{2^{-j}}(0,1)\subseteq B_{t_l}(0,2C_0).
$$
\end{proof}

To obtain the reverse Littlewood-Paley inequality, we need:
\begin{definition}[Street and Stein \cite{LP}]
For notational convenience, we define $\mathcal{D}_j=0$ for $j\in \mathbb{Z}\backslash\mathbb{N}$. Let
$$
\mathcal{U}_M = \sum_{\substack{j,k\in \mathbb{Z}\\ |j-k|\leq M}} \mathcal{D}_j \mathcal{D}_k, \quad \mathcal{R}_M = \sum_{\substack{j,k\in \mathbb{Z}\\ |j-k|>M}} \mathcal{D}_j \mathcal{D}_k.
$$
\end{definition}

By (\ref{identity}),
$$
\mathcal{U}_M+\mathcal{R}_M=\text{Id}.
$$
By Theorem 11.1 in \cite{LP}, Lemma \ref{cotlarstein} and Proposition \ref{littlewoodpaley} imply that for every $1<p<\infty$, there exists $M=M(p)$ such that $\|\mathcal{R}_M\|_{p\to p}<1$, and thus
$$
\mathcal{V}_M := \sum_{m=0}^\infty \mathcal{R}_M^m,
$$
convergent in the uniform operator topology, is the inverse operator of $\mathcal{U}_M$. For this $1<p<\infty$ and $M=M(p)$, denoting by $p'$ the dual of $p$, the adjoint $\mathcal{V}_M^*$ is bounded on $L^{p'}(\mathbb{H}^1)$. Same as in \cite{LP}, Proposition \ref{littlewoodpaley} implies
\begin{align*}
|\langle g,f\rangle| &= |\langle \mathcal{V}_M^* g, \mathcal{U}_Mf\rangle| \leq \sum_{|i|\leq M} \Big|\sum_{\substack{j,k\in \mathbb{Z}\\ |j-k|=i}} \langle \mathcal{D}_j^* \mathcal{V}_M^* g, \mathcal{D}_kf\rangle\Big|\\
&\leq (2M+1) \Big\| \Big( \sum_{j\in \mathbb{Z}} |\mathcal{D}_j^*\mathcal{V}_M^*g|^2\Big)^{\frac{1}{2}}\Big\|_{p'} \Big\| \Big(\sum_{k\in \mathbb{Z}} |\mathcal{D}_kf|^2\Big)^{\frac{1}{2}} \Big\|_p\\
&\lesssim \|g\|_{p'} \Big\| \Big(\sum_{k\in \mathbb{Z}} |\mathcal{D}_kf|^2\Big)^{\frac{1}{2}} \Big\|_p.
\end{align*}
Hence we have
\begin{corollary}[Littlewood-Paley theory]\label{completelittlewood}
For every $1<p<\infty$,
\begin{align*}
&\|f\|_p \lesssim \Big\|\Big(\sum_{j=0}^\infty |\mathcal{D}_jf|^2\Big)^{\frac{1}{2}} \Big\|_p \lesssim \|f\|_p,\quad \forall f\in L^p(\mathbb{H}^1),\\
&\Big\| \Big(\sum_{j=0}^\infty |\mathcal{D}_jg_j|^2\Big)^{\frac{1}{2}}\Big\|_p \lesssim \Big\| \Big(\sum_{j=0}^\infty|g_j|^2\Big)^{\frac{1}{2}}\Big\|_p, \quad \forall \{g_j\}\in L^p(\mathbb{H}^1, l^2(\mathbb{N})).
\end{align*}
\end{corollary}

\section{Key estimates: \texorpdfstring{$TT^*$}{TEXT} method}\label{section4}
The goal of this section is to prove Theorem \ref{main} below. We can decompose the kernel
\begin{equation}\label{kerneldecompose}
\frac{1}{t}=\sum_{j=0}^\infty 2^j\phi_j(2^jt), \quad \text{ for } t\in [-1,1],  
\end{equation}
where $\{\phi_j\}_{j\in\mathbb{N}}$ is bounded in the Fr\'echet space $C_c^\infty([-8,-2]\cup [2,8])$, and 
$$
\int \phi_j(t)\,dt =0, \quad \forall j>0,
$$
(see Remark \ref{decompose}). Fix $\{\phi_j\}_{j\in \mathbb{N}}$.

\begin{definition}
For each $j\in \mathbb{N}$, define
\begin{align*}
\mathcal{H}_j f(x)&= \int_{-1}^1 f\big(x\cdot \Gamma(t)^{-1}\big)\cdot 2^j\phi_j(2^jt)\,dt =\int_{-2^j}^{2^j} f\big(x\cdot \Gamma(2^{-j}t)^{-1}\big) \phi_j(t)\,dt\\
&=\int_{-2^j}^{2^j} f\big(x\cdot e^{-2^{-j}tX-\alpha(2^{-j}t)Y-\beta(2^{-j}t)T}0\big)\,\phi_j(t)\,dt.
\end{align*}
\end{definition}
Thus $\mathcal{H}_\Gamma = \sum_{j\in \mathbb{N}}\mathcal{H}_j$. Note $\|2^j\phi_j(2^j\cdot)\|_1=\|\phi_j\|_1$ is bounded. Thus $\|\mathcal{H}_j\|_{p\to p} \lesssim 1$ for every $1\leq p\leq \infty$, where the implicit constant is independent of $j$. Since $\text{supp}\,\phi_j\subseteq [-8,-2]\cup [2,8]$, for $j\geq 3$, 
$$
\mathcal{H}_jf(x) =\int f\big(x\cdot  e^{-2^{-j}tX-\alpha(2^{-j}t)Y -\beta(2^{-j}t)T}0 \big)\phi_j(t)\,dt.
$$
%For $j\geq 3$, $$ \mathcal{H}_j f(x)= \int_{-2^j}^{2^j} f\big(x\cdot \Gamma(2^{-j}t)^{-1}\big)\phi_j(t)\,dt =\int f\big(x\cdot \Gamma(2^{-j}t)^{-1}\big)\phi_j(t)\,dt. $$

As in Section \ref{finalsection}, to obtain the $L^p(\mathbb{H}^1)$ boundedness of $\mathcal{H}_\Gamma$ and $\mathcal{M}_\Gamma$, it suffices to prove the following theorem:
\begin{theorem}\label{main}
There exists $\epsilon>0$ (by shrinking the $\epsilon$ in (\ref{rivere})) such that for every $j,k\in \mathbb{N}$ with $j\geq 2$,
$$
\|\mathcal{H}_j^* \mathcal{D}_k\|_{2 \to 2}, \quad  \|\mathcal{H}_j \mathcal{D}_k^*\|_{2 \to 2} \lesssim 2^{-\epsilon|j-k|},
$$
where the implicit constant is independent of $j,k$.
\end{theorem}

\begin{remark}\label{decompose}
By Corollary 2.2.2 in \cite{NRS01}, every Calder\'on-Zygmund kernel has a decomposition as in (\ref{kerneldecompose}) satisfying the same properties. For the kernel $\frac{1}{t}$, the proof in this paper uses only the fact that $\frac{1}{t}$ has such decomposition. Therefore Theorem \ref{1} still holds if the kernel $\frac{1}{t}$ is replaced by any other Calder\'on-Zygmund kernel.
\end{remark}

Recall the vector fields $X_j, Y_j, T_j$ and the matrix $\widetilde A_j$ from Definition \ref{jmatrix}, and recall the function $\varsigma_k(y)$ and the operator $\mathcal{D}_k$ from Definition \ref{auxiliary}. Denote by $\bar \varsigma_k(y)$ the complex conjugate of $\varsigma_k(y)$.

\begin{proof}
\textbf{Case }$\mathbf{j\geq k+3}$: we have $j\geq 3$ and 
$$
\int\phi_j =0.
$$
Fixing $j,k$, define functions $c_1(t), c_2(t)$, and $c_3(t)$ on $[-8,8]$ by 
\begin{align*}
\begin{pmatrix}
c_1(t)\\
c_2(t)\\
c_3(t)
\end{pmatrix}
= \widetilde A_{j-3}^{-1}
\begin{pmatrix}
2^{-j}t\\
\alpha(2^{-j}t)\\
\beta(2^{-j}t)
\end{pmatrix}
=
\begin{pmatrix}
2^{-3}t\\
\frac{\alpha(2^{-j}t)-2^{-3}t\alpha(2^{-j+3})}{h_2(2^{-j+3})}\\
\frac{\big(\beta(2^{-j}t)-2^{-3}t\bar \beta(2^{-j+3})\big) - \big(\alpha(2^{-j}t)-2^{-3}t\alpha(2^{-j+3})\big)\bar h_\beta(2^{-j+3})/h_2(2^{-j+3})}{2^{-j+3}h_2(2^{-j+3}) + \bar h_3(2^{-j+3})}
\end{pmatrix},
\end{align*}
and define functions $\tilde c_1(t), \tilde c_2(t)$, and $\tilde c_3(t)$ on $[-8,8]$ by
\begin{align*}
\begin{pmatrix}
\tilde c_1(t)\\
\tilde c_2(t)\\
\tilde c_3(t)
\end{pmatrix}
=\widetilde A_k^{-1}
\begin{pmatrix}
2^{-j}t\\
\alpha(2^{-j}t)\\
\beta(2^{-j}t)
\end{pmatrix}
=\widetilde A_k^{-1} \widetilde A_{j-3}
\begin{pmatrix}
c_1(t)\\
c_2(t)\\
c_3(t)
\end{pmatrix}.
\end{align*}
Thus $c_1(t), c_2(t), c_3(t), \tilde c_1(t), \tilde c_2(t)$, and $\tilde c_3(t)$ are odd functions satisfying
\begin{align*}
2^{-j}tX+\alpha(2^{-j}t)Y + \beta(2^{-j} t) T= c_1(t) X_{j-3} + c_2(t) Y_{j-3} + c_3(t) T_{j-3}=\tilde c_1(t)X_k + \tilde c_2(t) Y_k + \tilde c_3(t) T_k.
\end{align*}
We have $|c_1(t)|\leq 1$ for $|t|\leq 8$. Since
\begin{equation}\label{ratioincrease}
\big(\frac{\bar h_\beta(s)}{h_2(s)}\big)' = \frac{h_2'(s)}{h_2(s)^2} \Big(\frac{\bar \beta''(s)}{\alpha''(s)} h_2(s)-\bar h_\beta(s)\Big) =  \frac{h_2'(s)}{h_2(s)^2} \bar h_3(s)>0, \quad \forall s>0,
\end{equation}
for $t\in (0,8]$,
\begin{align*}
0\geq c_2(t)&= -\frac{2^{-j}t}{h_2(2^{-j+3})}\int_{2^{-j}t}^{2^{-j+3}} \Big(\frac{\alpha(s)}{s}\Big)'\,ds = -\frac{2^{-j}t}{h_2(2^{-j+3})}\int_{2^{-j}t}^{2^{-j+3}} \frac{h_2(s)}{s^2}\,ds \geq -1,\\
0\leq c_3(t) &=\frac{\frac{1}{2}\int_0^{2^{-j}t}h_2 - 2^{-j}t \int_{2^{-j}t}^{2^{-j+3}} \big(\frac{\bar \beta(s)}{s}\big)'\,ds + \frac{\bar h_\beta(2^{-j+3})}{h_2(2^{-j+3})} 2^{-j}t \int_{2^{-j}t}^{2^{-j+3}} \big(\frac{\alpha(s)}{s}\big)'\,ds}{2^{-j+3}h_2(2^{-j+3})+\bar h_3(2^{-j+3})}\\
&=\frac{\frac{1}{2}\int_0^{2^{-j}t}h_2 + 2^{-j}t \int_{2^{-j}t}^{2^{-j+3}} \frac{h_2(s)}{s^2} \big(\frac{\bar h_\beta(2^{-j+3})}{h_2(2^{-j+3})}- \frac{\bar h_\beta(s)}{h_2(s)}\big)\,ds}{2^{-j+3}h_2(2^{-j+3})+\bar h_3(2^{-j+3})}\\
&=\frac{\frac{1}{2}\int_0^{2^{-j}t}h_2 + 2^{-j}t \int_{2^{-j}t}^{2^{-j+3}} \frac{h_2(s)}{s^2} \int_s^{2^{-j+3}} \frac{h_2'(r)}{h_2(r)^2}\bar h_3(r)\,dr\,ds}{2^{-j+3}h_2(2^{-j+3})+\bar h_3(2^{-j+3})}\\
&\leq \frac{\frac{1}{2}(2^{-j}t)h_2(2^{-j}t) + \bar h_3(2^{-j+3})}{2^{-j+3}h_2(2^{-j+3})+\bar h_3(2^{-j+3})} \leq 1.
\end{align*}
By the oddness, $|c_1(t)|, |c_2(t)|, |c_3(t)|\leq 1$ for $t\in [-8,8]$. Thus by (\ref{rivere}),
$$
|(\tilde c_1(t), \tilde c_2(t), \tilde c_3(t))| \lesssim \|\widetilde A_k^{-1} \widetilde A_{j-3}\| \lesssim 2^{-\epsilon|j-k|}, \forall t\in [-8,8].
$$
Therefore by the Baker-Campbell-Hausdorff formula, (\ref{xi}), and Lemma \ref{ldeltaone},
\begin{align*}
&\quad |\mathcal{H}_j \mathcal{D}_k^*f(x) |\\
&= \Big|\iint f(x\cdot e^{-2^{-j}tX- \alpha(2^{-j}t)Y -\beta(2^{-j}t)T} e^{y_1X_k + y_2 Y_k +y_3 T_k}0) \bar \varsigma_k(y) \phi_j(t)\,dy\,dt\Big| \\
&= \Big|\iint f(x\cdot e^{-c_1(t)X_{j-3}- c_2(t)Y_{j-3} -c_3(t)T_{j-3}} e^{y_1X_k + y_2 Y_k +y_3 T_k}0) \bar \varsigma_k(y) \phi_j(t)\,dy\,dt \Big|\\
&= \Big|\iint f(x\cdot e^{-\tilde c_1(t)X_k- \tilde c_2(t)Y_k - \tilde c_3(t)T_k} e^{y_1X_k + y_2 Y_k +y_3 T_k}0) \bar \varsigma_k(y) \phi_j(t)\,dy\,dt\Big| \\
&= \Big|\iint f(x\cdot e^{(y_1-\tilde c_1(t))X_k + (y_2-\tilde c_2(t))Y_k + (y_3-\tilde c_3(t) -\xi_k (y_2 \tilde c_1(t) - y_1 \tilde c_2(t)))T_k}0)\bar \varsigma_k(y) \phi_j(t)\,dy\,dt\\
&\quad - \iint f(x\cdot e^{y_1X_k +y_2Y_k  + y_3 T_k}0)\bar \varsigma_k(y) \phi_j(t)\,dy\,dt\Big|\\
&\lesssim \sup_{|t|\leq 8} \Big|\int f(x\cdot e^{y_1X_k +y_2Y_k  + y_3 T_k}0) \big(\bar \varsigma_k\big(y_1 + \tilde c_1(t), y_2 + \tilde c_2(t), y_3 + \tilde c_3(t)+ \xi_k(y_2 \tilde c_1(t)-y_1 \tilde c_2(t))\big)-\bar \varsigma_k(y)\big)\,dy\Big| \\
&\leq \|f\|_\infty \sup_{|t|\leq 8} \int \big|\varsigma_k\big(y_1 + \tilde c_1(t), y_2 + \tilde c_2(t), y_3 + \tilde c_3(t)+ \xi_k(y_2 \tilde c_1(t)-y_1 \tilde c_2(t))\big) -\varsigma_k(y)\big|\,dy \\
&\lesssim 2^{-\epsilon|j-k|} \|f\|_\infty.
\end{align*}
Thus $\|\mathcal{H}_j \mathcal{D}_k^*\|_{\infty \to \infty} \lesssim 2^{-\epsilon|j-k|}$. Interpolate with  $\|\mathcal{H}_j\mathcal{D}_k^*\|_{1\to 1} \lesssim 1$, we have
$$
\|\mathcal{H}_j \mathcal{D}_k^* \|_{2\to 2} \lesssim 2^{-\epsilon|j-k|},
$$
where the value of $\epsilon$ changes but is still independent of $j,k$. Similarly, $\|\mathcal{H}_j^* \mathcal{D}_k \|_{2\to 2} \lesssim 2^{-\epsilon|j-k|}$.

\textbf{Case }$\mathbf{j=k+2,k+1,k}$: due to $\|\mathcal{H}_j\|_{2\to 2}, \|\mathcal{D}_k\|_{2\to 2} \lesssim 1$, uniformly in $j,k$.

\textbf{Case }$\mathbf{j< k}$: we have $k>0$, and thus
$$
\int \varsigma_k=0.
$$
To show $\|\mathcal{H}_j^* \mathcal{D}_k\|_{2 \to 2}, \|\mathcal{H}_j \mathcal{D}_k^*\|_{2 \to 2} \lesssim 2^{-\epsilon|j-k|}$, it suffices to show 
\begin{equation}\label{cotlarestimate}
\|\mathcal{D}_k^*\mathcal{H}_j\mathcal{H}_j^* \mathcal{H}_j \mathcal{H}_j^* \mathcal{D}_k\|_{\infty \to \infty}, \quad \|\mathcal{D}_k\mathcal{H}_j^*\mathcal{H}_j \mathcal{H}_j^* \mathcal{H}_j \mathcal{D}_k^*\|_{\infty \to \infty} \lesssim 2^{-\epsilon|j-k|},
\end{equation}
where $\epsilon>0$ is allowed to change from line to line finitely many times, but will always be independent of $j,k$. We prove only $\|\mathcal{D}_k^*\mathcal{H}_j\mathcal{H}_j^* \mathcal{H}_j \mathcal{H}_j^* \mathcal{D}_k\|_{\infty \to \infty}\lesssim 2^{-\epsilon|j-k|}$; the other one is similar (see Remark \ref{change}).

Fixing $j$, define functions $C_1(t), C_2(t)$, and $C_3(t)$ on $[-8,8]$ by
\begin{equation}\label{coefficient}
\begin{pmatrix}
C_1(t)\\
C_2(t)\\
C_3(t)
\end{pmatrix}
= \widetilde A_j^{-1}
\begin{pmatrix}
2^{-j}t\\
\alpha(2^{-j}t)\\
\beta(2^{-j}t)
\end{pmatrix}
=
\begin{pmatrix}
t\\
\frac{\alpha(2^{-j}t)-t\alpha(2^{-j})}{h_2(2^{-j})}\\
\frac{\beta(2^{-j}t)-t\bar \beta(2^{-j}) - \big(\alpha(2^{-j}t)-t\alpha(2^{-j})\big)\bar h_\beta(2^{-j})/h_2(2^{-j})}{2^{-j}h_2(2^{-j}) + \bar h_3(2^{-j})}
\end{pmatrix}.
\end{equation}
Thus $C_1(t), C_2(t)$, and $C_3(t)$ are odd functions satisfying
\begin{align*}
2^{-j}tX+\alpha(2^{-j}t)Y + \beta(2^{-j} t) T= C_1(t) X_{j} + C_2(t) Y_{j} + C_3(t) T_j.
\end{align*}
Hence
$$
\mathcal{H}_jf(x)=\int_{-2^j}^{2^j}f\big(x\cdot e^{-C_1(t)X_j-C_2(t)Y_j-C_3(t)T_j}0\big)\phi_j(t)\,dt,
$$
and
\begin{align*}
&\quad |\mathcal{D}_k^*\mathcal{H}_j\mathcal{H}_j^* \mathcal{H}_j \mathcal{H}_j^* \mathcal{D}_kf(x)| \\
&= \Big|\iiint \!\!\! \iiint_{|t|,|s|,|u|,|r|\leq 2^j} f\Big(x\cdot e^{y_1X_k + y_2 Y_k +y_3 T_k} e^{-C_1(t)X_j- C_2(t)Y_j - C_3(t)T_j} e^{C_1(s)X_j+ C_2(s)Y_j + C_3(s)T_j}  \\
&\qquad \qquad \qquad \qquad \qquad \qquad  e^{-C_1(u)X_j- C_2(u)Y_j - C_3(u)T_j} e^{C_1(r)X_j+C_2(r)Y_j+C_3(r)T_j} e^{-v_1X_k-v_2Y_k-v_3T_k} 0\Big)\\
&\qquad \qquad \qquad \qquad \qquad \qquad \qquad \qquad \qquad \qquad \qquad \varsigma_k(v)\bar \phi_j(r)\phi_j(u)\bar \phi_j(s) \phi_j(t)\bar \varsigma_k(y)    \,dv\,dr\,du\,ds\,dt\,dy\Big|.
\end{align*}
We will only bound the part of the above integration on the region $t,s,u,r\geq 
0$, which we denote by $[\mathcal{D}_k^* \mathcal{H}_j \mathcal{H}_j^* \mathcal{H}_j \mathcal{H}_j^* \mathcal{D}_k]_+f(x)$, because the proof for the other parts is similar using the oddness of $C_1(t), C_2(t)$, and $C_3(t)$ (see Remark \ref{change}). 

For $y,v\in \text{supp}\,\varsigma_k\subseteq [-1,1]^3 \cup (\widetilde A_k^{-1} \widetilde A_{k-1}[-1,1]^3)$, let
$$
\tilde y = \widetilde A_j^{-1} \widetilde A_k y, \quad \tilde v =\widetilde A_j^{-1} \widetilde A_k v \in \big(\widetilde A_j^{-1} \widetilde A_k[-1,1]^3\big)\bigcup \big(\widetilde A_j^{-1} \widetilde A_{k-1}[-1,1]^3\big).
$$
By (\ref{rivere}), $|\tilde y|, |\tilde v|\lesssim 2^{-\epsilon|j-k|}$, where the implicit constant is independent of $j,k$. we have
\begin{align*}
y_1X_k+y_2Y_k+y_3T_k= \tilde y_1 X_j + \tilde y_2 Y_j + \tilde y_3 T_j,\\
v_1X_k+v_2Y_k+v_3T_k= \tilde v_1 X_j + \tilde v_2 Y_j + \tilde v_3 T_j.
\end{align*}
Fixing $j$, denote the regions
\begin{equation}\label{regions}
\begin{aligned}
&\bar R_1:=\{(t,s,u,r)\in ([2,8]\cap [0,2^j])^4: t\leq u\}, \quad R_1:=\{(t,s,u)\in ([2,8]\cap [0,2^j])^3: t\leq u\},\\ 
&\bar R_2:=\{(t,s,u,r)\in ([2,8]\cap [0,2^j])^4: r\leq s\}, \quad R_2:=\{(s,u,r)\in ([2,8]\cap [0,2^j])^3: r\leq s\},\\
&\bar R_3:=\{(t,s,u,r)\in ([2,8]\cap [0,2^j])^4: r\geq t\geq u, r\geq s\}, \quad R_3:=\{(t,s,r)\in ([2,8]\cap [0,2^j])^3: r\geq t\geq u, r\geq s\},\\
&\bar R_4:=\{(t,s,u,r)\in ([2,8]\cap [0,2^j])^4: t\geq r\geq s, t\geq u\},\quad R_4:=\{(t,u,r)\in ([2,8]\cap [0,2^j])^3: t\geq r\geq s, t\geq u\}.
\end{aligned}
\end{equation}
Note the region $R_3$ depends on $u\in [2,8]\cap [0,2^j]$, and the region $R_4$ depends on $s\in [2,8]\cap [0,2^j]$. By the Baker-Campbell-Hausdorff formula and (\ref{xi}),
\begin{equation}\label{mainequation}
\begin{aligned}
&\quad |[\mathcal{D}_k^*\mathcal{H}_j\mathcal{H}_j^* \mathcal{H}_j \mathcal{H}_j^* \mathcal{D}_k]_+f(x)| \\
&= \Big|\iiint \!\!\! \iiint_{(t,s,u,r)\in ([2,8]\cap[0,2^j])^4} f\Big(x\cdot e^{y_1X_k + y_2 Y_k +y_3 T_k} e^{-C_1(t)X_j- C_2(t)Y_j - C_3(t)T_j} e^{C_1(s)X_j+ C_2(s)Y_j + C_3(s)T_j}  \\
&\qquad \qquad \qquad \qquad \qquad \qquad  e^{-C_1(u)X_j- C_2(u)Y_j - C_3(u)T_j} e^{C_1(r)X_j+C_2(r)Y_j+C_3(r)T_j} e^{-v_1X_k-v_2Y_k-v_3T_k} 0\Big)\\
&\qquad \qquad \qquad \qquad \qquad \qquad \qquad \qquad \qquad \qquad \qquad \varsigma_k(v)\bar \phi_j(r)\phi_j(u)\bar \phi_j(s) \phi_j(t)\bar \varsigma_k(y)    \,dv\,dr\,du\,ds\,dt\,dy\Big|\\
&\leq \Big|\iiint \!\!\! \iiint_{(t,s,u,r)\in\bar R_1} f\Big(x\cdot e^{\tilde y_1X_j + \tilde y_2 Y_j +\tilde y_3 T_j} \\
&\qquad \qquad \qquad e^{(-t+s-u)X_j +\big(-C_2(t)+C_2(s)-C_2(u)\big)Y_j  +\big(-C_3(t)+C_3(s) -C_3(u)\big)T_j} \\
&\qquad \qquad \qquad e^{-\xi_j\big(tC_2(s) - sC_2(t) - t C_2(u) + s C_2(u)+uC_2(t)-uC_2(s) \big)T_j} \\
&\qquad \qquad \quad e^{C_1(r)X_j+C_2(r)Y_j+C_3(r)T_j} e^{-\tilde v_1X_j-\tilde v_2Y_j-\tilde v_3T_j}0 \Big) \varsigma_k(v)\bar \phi_j(r)\phi_j(u)\bar \phi_j(s) \phi_j(t)\bar \varsigma_k(y)    \,dv\,dr\,du\,ds\,dt\,dy\Big| \\
&\quad + \Big|\iiint \!\!\! \iiint_{(t,s,u,r)\in\bar R_2} f\Big(x\cdot e^{\tilde y_1X_j + \tilde y_2 Y_j +\tilde y_3 T_j} e^{-C_1(t)X_j-C_2(t)Y_j-C_3(t)T_j} \\
&\qquad \qquad \qquad e^{(s-u+r)X_j +\big(C_2(s)-C_2(u)+C_2(r)\big)Y_j  +\big(C_3(s) -C_3(u)+C_3(r)\big)T_j} \\
&\qquad \qquad \qquad e^{-\xi_j\big(s C_2(u)-uC_2(s)-sC_2(r) +uC_2(r) +rC_2(s) -rC_2(u) \big)T_j} \\
&\qquad \qquad \qquad \qquad \qquad \qquad \qquad \qquad e^{-\tilde v_1X_j-\tilde v_2Y_j-\tilde v_3T_j}0\Big) \varsigma_k(v)\bar \phi_j(r)\phi_j(u)\bar \phi_j(s) \phi_j(t)\bar \varsigma_k(y)    \,dv\,dr\,du\,ds\,dt\,dy\Big|\\
&\quad + \Big|\iiint \!\!\! \iiint_{(t,s,u,r)\in\bar R_3} f\Big(x\cdot e^{\tilde y_1X_j + \tilde y_2 Y_j +\tilde y_3 T_j} \\
&\qquad \qquad \qquad e^{(-t+s-u+r)X_j +\big(-C_2(t)+C_2(s)-C_2(u)+C_2(r)\big)Y_j  +\big(-C_3(t)+C_3(s) -C_3(u)+C_3(r)\big)T_j} \\
&\qquad \qquad \qquad e^{-\xi_j\big(s C_2(u)-uC_2(s)-sC_2(r) +uC_2(r) +rC_2(s) -rC_2(u) +tC_2(s)-tC_2(u)+tC_2(r)-sC_2(t)+uC_2(t)-rC_2(t)\big)T_j} \\
&\qquad \qquad \qquad \qquad \qquad \qquad \qquad \qquad e^{-\tilde v_1X_j-\tilde v_2Y_j-\tilde v_3T_j}0\Big) \varsigma_k(v)\bar \phi_j(r)\phi_j(u)\bar \phi_j(s) \phi_j(t)\bar \varsigma_k(y)    \,dv\,dr\,du\,ds\,dt\,dy\Big|\\
&\quad +\Big|\iiint \!\!\! \iiint_{(t,s,u,r)\in\bar R_4} f\Big(x\cdot e^{\tilde y_1X_j + \tilde y_2 Y_j +\tilde y_3 T_j} \\
&\qquad \qquad \qquad e^{(-t+s-u+r)X_j +\big(-C_2(t)+C_2(s)-C_2(u)+C_2(r)\big)Y_j  +\big(-C_3(t)+C_3(s) -C_3(u)+C_3(r)\big)T_j} \\
&\qquad \qquad \qquad e^{-\xi_j\big(s C_2(u)-uC_2(s)-sC_2(r) +uC_2(r) +rC_2(s) -rC_2(u) +tC_2(s)-tC_2(u)+tC_2(r)-sC_2(t)+uC_2(t)-rC_2(t)\big)T_j} \\
&\qquad \qquad \qquad \qquad \qquad \qquad \qquad \qquad  e^{-\tilde v_1X_j-\tilde v_2Y_j-\tilde v_3T_j}0\Big) \varsigma_k(v)\bar \phi_j(r)\phi_j(u)\bar \phi_j(s) \phi_j(t)\bar \varsigma_k(y)    \,dv\,dr\,du\,ds\,dt\,dy\Big| \\
&=: I+I\!I + I\!I\!I + I\!V.
\end{aligned}
\end{equation}
We will have $[2,8]\cap [0,2^j]$ appear repeatedly in the rest of the proof. The interval $[0,2^j]$ slightly modifies the interval only for $j=0,1,2$ and bears no importance. For notational convenience, we write $[2,8]$ instead of $[2,8]\cap [0,2^j]$ below. 

Denote
\begin{equation}\label{zw}
\begin{aligned}
z&=(z_1, z_2, z_3) = \Phi(t,s,u)\\
&:=\Big(-t+s-u, -C_2(t)+C_2(s) -C_2(u), \\
&\qquad -C_3(t)+C_3(s)-C_3(u) - \xi_j\big(tC_2(s)-sC_2(t)-tC_2(u)+sC_2(u)+uC_2(t)-uC_2(s)\big)\Big),
\\ w&=(w_1, w_2, w_3) = \hat \Phi(s,u,r)\\ &:=\Big(s-u+r, C_2(s) -C_2(u)+C_2(r), \\ &\qquad C_3(s)-C_3(u) +C_3(r)- \xi_j\big(sC_2(u)-uC_2(s)-sC_2(r)+uC_2(r)+rC_2(s)-rC_2(u)\big)\Big),\\
\zeta&=(\zeta_1, \zeta_2, \zeta_3) = \Psi(t,s,r) = \hat \Psi(t,u,r)\\
&:=\Big(-t+s-u+r, -C_2(t)+C_2(s)-C_2(u)+C_2(r),  \\
& \qquad -C_3(t)+C_3(s) -C_3(u)+C_3(r)-\xi_j\big(s C_2(u)-uC_2(s)-sC_2(r) +uC_2(r) +rC_2(s) -rC_2(u) \\
&\qquad +tC_2(s)-tC_2(u)+tC_2(r)-sC_2(t)+uC_2(t)-rC_2(t)\big)\Big),
\end{aligned}
\end{equation}
where the function $\Psi$ depends on the parameter $u$, and the function $\hat \Psi$ depends on the parameter $s$. We have
\begin{equation}\label{determinant}
\begin{aligned}
D\Phi(t,s,u)=
\begin{pmatrix}
-1 & 1 & -1\\
&&\\
-C_2'(t) & C_2'(s) & -C_2'(u)\\
&&\\
-C_3'(t)+\xi_j\big( sC_2'(t) & C_3'(s)+\xi_j\big(C_2(t) & -C_3'(u)+\xi_j\big(tC_2'(u)\\
-C_2(s)+ C_2(u)-uC_2'(t) \big) & -tC_2'(s)-C_2(u)+uC_2'(s) \big)  &-sC_2'(u)-C_2(t)+C_2(s) \big)
\end{pmatrix}.
%\\ D\hat \Phi(s,u,r)= \begin{pmatrix} 1 & -1 & 1\\ &&\\ C_2'(s) & -C_2'(u) & C_2'(r)\\ &&\\ C_3'(s)-\xi_j\big( C_2(u) & -C_3'(u)-\xi_j\big(sC_2'(u) & C_3'(r)-\xi_j\big(uC_2'(r)\\ -uC_2'(s)-C_2(r)+rC_2'(s) \big) & -C_2(s)+C_2(r)-rC_2'(u) \big)  &-sC_2'(r)+C_2(s)-C_2(u) \big) \end{pmatrix} 
\end{aligned} \end{equation}
We define the following transported measures
\begin{equation}\label{measure}
\begin{aligned}
d\mu= \kappa(z)\,dz:=\Phi_*\big(\phi_j(t)\bar \phi_j(s)\phi_j(u)1_{R_1}(t,s,u)\,dt\,ds\,du\big),\\
d\hat \mu= \hat \kappa(w)\,dw:=\hat \Phi_*\big(\bar \phi_j(s)\phi_j(u)\bar \phi_j(r)1_{R_2}(s,u,r)\,ds\,du\,dr\big),\\
d\nu=\chi(\zeta)\,d\zeta:= \Psi_*\big(\phi_j(t)\bar \phi_j(s)\bar \phi_j(r)1_{R_3}(t,s,r)\,dt\,ds\,dr\big),\\
d\hat \nu=\hat \chi(\zeta)\,d\zeta:= \hat \Psi_*\big(\phi_j(t)\phi_j(u)\bar \phi_j(r)1_{R_4}(t,u,r)\,dt\,du\,dr\big),
\end{aligned}
\end{equation}
by
\begin{align*}
\int g(z)\,d\mu(z) = \iiint_{R_1} g(\Phi(t,s,u)) \phi_j(t)\bar \phi_j(s)\phi_j(u)\,dt\,ds\,du, \quad \forall g\in L_{loc}^1,\\
\int g(w)\,d\hat \mu(w) = \iiint_{R_2} g(\hat \Phi(s,u,r)) \bar \phi_j(s) \phi_j(u)\bar \phi_j(r)\,ds\,du\,dr, \quad \forall g\in L_{loc}^1,\\
\int g(\zeta)\,d\nu(\zeta) = \iiint_{R_3} g(\Psi(t,s,r)) \phi_j(t)\bar \phi_j(s)\bar \phi_j(r)\,dt\,ds\,dr, \quad \forall g\in L_{loc}^1,\\
\int g(\zeta)\,d\hat \nu(\zeta) = \iiint_{R_4} g(\hat \Psi(t,u,r)) \phi_j(t)\phi_j(u)\bar \phi_j(r)\,dt\,du\,dr, \quad \forall g\in L_{loc}^1.
\end{align*}
Proposition \ref{density} below shows $\mu, \hat \mu, \nu, \hat \nu$ are absolutely continuous. The densities $\kappa, \hat \kappa, \chi, \hat \chi$ depends on $j$, $\chi$ depends on $u$, and $\hat \chi$ depends on $s$. Recall $\int \varsigma_k=0$, and $|\tilde y|, |\tilde v|\lesssim 2^{-\epsilon|j-k|}$, where the implicit constant is independent of $j,k$. If we denote 
$$
\tilde x:=e^{C_1(r)X_j+C_2(r)Y_j+C_3(r)T_j} e^{-\tilde v_1X_j-\tilde v_2Y_j-\tilde v_3T_j}0, \quad g_1(\cdot):=f(x\cdot \cdot),
$$
we have
$$
\begin{aligned}
I&\lesssim \sup_{\substack{r\in [2,8]\\ v\in \text{supp}\,\varsigma_k}} \Big| \iiiint_{(t,s,u)\in R_1} g_1\big(e^{\tilde y_1X_j +\tilde y_2 Y_j + \tilde y_3 T_j} e^{z_1 X_j + z_2 Y_j + z_3 T_j} \tilde x\big) \bar \varsigma_k(y)\phi_j(t)\bar \phi_j(s)\phi_j(u)\,dy\,dt\,ds\,du\Big|\\
&=\sup_{\substack{r\in [2,8]\\ v\in \text{supp}\,\varsigma_k}} \Big| \iint g_1\big(e^{\tilde y_1X_j +\tilde y_2 Y_j + \tilde y_3 T_j} e^{z_1 X_j + z_2 Y_j + z_3 T_j} \tilde x\big)  \bar \varsigma_k(y)\kappa(z)\,dy\,dz\Big|\\
&=\sup_{\substack{r\in [2,8]\\ v\in \text{supp}\,\varsigma_k}} \Big| \iint \Big(g_1\big(e^{\tilde y_1X_j +\tilde y_2 Y_j + \tilde y_3 T_j} e^{z_1 X_j + z_2 Y_j + z_3 T_j} \tilde x\big) -g_1\big(e^{z_1 X_j + z_2 Y_j + z_3 T_j} \tilde x\big)\Big) \bar \varsigma_k(y)\kappa(z)\,dy\,dz\Big|\\
&\lesssim \sup_{\substack{r\in [2,8]\\ y,v\in \text{supp}\,\varsigma_k}} \Big| \int \Big(g_1\big(e^{(\tilde y_1 +z_1)X_j +(\tilde y_2+z_2) Y_j + \big((\tilde y_3+z_3+\xi_j(\tilde y_1 z_2-\tilde y_2 z_1)\big) T_j}  \tilde x\big) -g_1\big(e^{z_1 X_j + z_2 Y_j + z_3 T_j} \tilde x\big)\Big) \kappa(z)\,dz\Big|\\
&=\sup_{\substack{r\in [2,8]\\ y,v\in \text{supp}\,\varsigma_k}} \Big| \int g_1\big(e^{(\tilde y_1 +z_1)X_j +(\tilde y_2+z_2) Y_j + \big((\tilde y_3+z_3+\xi_j(\tilde y_1 z_2-\tilde y_2 z_1)\big) T_j}  \tilde x\big) \\
&\qquad \cdot \big(\kappa(z)-\kappa\big(\tilde y_1 +z_1, \tilde y_2+z_2, \tilde y_3+z_3+\xi_j(\tilde y_1 z_2-\tilde y_2 z_1)\big) \big) \,dz\Big|\\
&\leq \|f\|_\infty \sup_{\substack{r\in [2,8]\\ y\in \text{supp}\,\varsigma_k}} \int \big|\kappa(z)-\kappa\big(\tilde y_1 +z_1, \tilde y_2+z_2, \tilde y_3+z_3+\xi_j(\tilde y_1 z_2-\tilde y_2 z_1)\big) \big|\,dz\\
&\lesssim 2^{-\epsilon|j-k|} \|f\|_\infty,
\end{aligned}
$$
where the last inequality is due to Proposition \ref{density} below, the value of $\epsilon>0$ changes but is independent of $j,k$, and the implicit constants are independent of $j,k$. 

Similarly, by Proposition \ref{density} below, if we denote 
$$
g_2(\cdot) := f\Big(x\cdot e^{\tilde y_1X_j + \tilde y_2 Y_j +\tilde y_3 T_j} e^{-C_1(t)X_j-C_2(t)Y_j-C_3(t)T_j}\cdot \Big), 
$$
we have
$$
\begin{aligned}
I\!I&\lesssim \sup_{\substack{t\in [2,8]\\ y\in \text{supp}\,\varsigma_k}} \Big| \iiiint_{(s,u,r)\in R_2} g_2\big(e^{w_1 X_j + w_2 Y_j + w_3 T_j} e^{-\tilde v_1X_j -\tilde v_2 Y_j -\tilde v_3 T_j}0\big) \bar \phi_j(s)\phi_j(u)\bar \phi_j(r)\varsigma_k(v)\,ds\,du\,dr\,dv\Big|\\
%&=\sup_{\substack{t\in [2,8]\\ y\in \text{supp}\,\varsigma_k}} \Big| \iint g_2\big(e^{w_1 X_j + w_2 Y_j + w_3 T_j} e^{-\tilde v_1X_j -\tilde v_2 Y_j -\tilde v_3 T_j}0\big) \hat \kappa(w) \varsigma_k(v)\,dw\,dv\Big|\\
&\lesssim \sup_{\substack{t\in [2,8]\\ y,v\in \text{supp}\,\varsigma_k}} \Big| \int \Big(g_2\big(e^{w_1 X_j + w_2 Y_j + w_3 T_j} e^{-\tilde v_1X_j -\tilde v_2 Y_j -\tilde v_3 T_j}0\big) -g_2\big(e^{w_1 X_j + w_2 Y_j + w_3 T_j} 0\big)\Big) \hat \kappa(w)\,dw\Big|\\
%&\lesssim \sup_{\substack{t\in [2,8]\\ v,y\in \text{supp}\,\varsigma_k}} \Big| \int \Big(g_2\big(e^{(w_1-\tilde v_1) X_j + (w_2-\tilde v_2) Y_j + \big(w_3-\tilde v_3 -\xi_j(w_1\tilde v_2-w_2 \tilde v_1)\big) T_j}0\big) -g_2\big(e^{w_1 X_j + w_2 Y_j + w_3 T_j} 0\big)\Big) \hat \kappa(w)\,dw\Big|\\
%&=\sup_{\substack{t\in [2,8]\\ v,y\in \text{supp}\,\varsigma_k}} \Big| \int g_2\big(e^{(w_1-\tilde v_1) X_j + (w_2-\tilde v_2) Y_j + \big(w_3-\tilde v_3 -\xi_j(w_1\tilde v_2-w_2 \tilde v_1)\big) T_j}0\big) \\
%&\qquad \cdot \big(\hat \kappa(w)- \hat \kappa \big(w_1-\tilde v_1, w_2-\tilde v_2, w_3-\tilde v_3 -\xi_j(w_1\tilde v_2-w_2 \tilde v_1)\big) \big) \,dw\Big|\\
&\leq \|f\|_\infty \sup_{\substack{t\in [2,8]\\ v\in \text{supp}\,\varsigma_k}} \int \big|\hat \kappa(w)-\hat \kappa \big(w_1-\tilde v_1, w_2-\tilde v_2, w_3-\tilde v_3 -\xi_j(w_1\tilde v_2-w_2 \tilde v_1)\big) \big|\,dw\\
&\lesssim 2^{-\epsilon|j-k|} \|f\|_\infty, 
\end{aligned}
$$
where the value of $\epsilon>0$ changes but is indpendent of $j,k$, and the implicit constants are independent of $j,k$. By Proposition \ref{density} below, if we denote
$$
\vardbtilde x = e^{-\tilde v_1 X_j-\tilde v_2Y_j-\tilde v_3T_j}0, \quad g_3(\cdot )= f(x\cdot \cdot),
$$
we have
\begin{align*}
I\!I\!I & \lesssim \sup_{\substack{u\in [2,8]\\ v\in \text{supp}\,\varsigma_k}} \Big| \iiiint_{(t,s,r)\in R_3} g_3\big(e^{\tilde y_1X_j+\tilde y_2Y_j +\tilde y_3T_j} e^{\zeta_1X_j + \zeta_2Y_j + \zeta_3T_j}\vardbtilde x\big) \bar \varsigma_k(y) \phi_j(t)\bar \phi_j(s)\bar \phi_j(r)\,dy\,dt\,ds\,dr\Big|\\
%&=\sup_{\substack{t\in [2,8]\\ v\in \text{supp}\,\varsigma_k}} \Big|\iint g_3\big(e^{\tilde y_1X_j+\tilde y_2Y_j +\tilde y_3T_j} e^{-C_1(t)X_j-C_2(t)Y_j - C_3(t)T_j} e^{w_1X_j + w_2Y_j + w_3T_j}\vardbtilde x\big)\bar \varsigma_k(y) \hat \chi(w)\,dy\,dw\Big|\\
%&=\sup_{\substack{t\in [2,8]\\ v\in \text{supp}\,\varsigma_k}} \Big|\iint \Big(g_3\big(e^{\tilde y_1X_j+\tilde y_2Y_j +\tilde y_3T_j} e^{-C_1(t)X_j-C_2(t)Y_j - C_3(t)T_j} e^{w_1X_j + w_2Y_j + w_3T_j}\vardbtilde x\big)\\
%&\qquad \qquad \qquad \qquad \qquad - g_3\big(e^{-C_1(t)X_j-C_2(t)Y_j - C_3(t)T_j} e^{w_1X_j + w_2Y_j + w_3T_j}\vardbtilde x\big)\Big)\bar \varsigma_k(y) \hat \chi(w)\,dy\,dw\Big|\\
%&\lesssim\sup_{\substack{t\in [2,8]\\ y,v\in \text{supp}\,\varsigma_k}} \Big|\int \Big(g_3\big(e^{\tilde y_1X_j+\tilde y_2Y_j +\tilde y_3T_j} e^{-C_1(t)X_j-C_2(t)Y_j - C_3(t)T_j} e^{w_1X_j + w_2Y_j + w_3T_j}\vardbtilde x\big)\\
%&\qquad \qquad \qquad \qquad \qquad \qquad \quad  - g_3\big(e^{-C_1(t)X_j-C_2(t)Y_j - C_3(t)T_j} e^{w_1X_j + w_2Y_j + w_3T_j}\vardbtilde x\big)\Big) \hat \chi(w)\,dw\Big|\\
&\lesssim \sup_{\substack{u\in [2,8]\\ y,v\in \text{supp}\,\varsigma_k}} \Big|\int \Big(g_3\big(e^{\tilde y_1X_j+\tilde y_2Y_j +\tilde y_3T_j} e^{\zeta_1X_j + \zeta_2Y_j + \zeta_3T_j}\vardbtilde x\big) - g_3\big(e^{\zeta_1X_j + \zeta_2Y_j + \zeta_3T_j}\vardbtilde x\big)\Big) \chi(\zeta)\,d\zeta\Big|\\
&= \sup_{\substack{u\in [2,8]\\ y,v\in \text{supp}\,\varsigma_k}}\Big|\int g_3\big( e^{\tilde y_1X_j+\tilde y_2Y_j +\tilde y_3T_j} e^{\zeta_1X_j + \zeta_2Y_j + \zeta_3T_j}\vardbtilde x \big) \\
&\qquad \qquad \qquad \cdot \Big(\chi(\zeta)- \chi\big(\tilde y_1+\zeta_1, \tilde y_2 + \zeta_2, \tilde y_3 + \zeta_3+\xi_j(\tilde y_1\zeta_2-\tilde y_2\zeta_1)\big)\Big)\,d\zeta\Big|\\
%&\leq \|f\|_\infty \sup_{\substack{t\in [2,8]\\ y\in \text{supp}\,\varsigma_k}} \int \Big|\hat \chi(w)- \hat \chi\big(\tilde y_1+w_1, \tilde y_2 + w_2, \tilde y_3 + w_3+\xi_j(-2\tilde y_1C_2(t)+2\tilde y_2C_1(t)+\tilde y_1w_2-\tilde y_2w_1)\big)\Big|\,dw\\
&\leq \|f\|_\infty \sup_{\substack{u\in [2,8]\\ y\in \text{supp}\,\varsigma_k}} \int \Big|\chi(\zeta)- \chi\big(\tilde y_1+\zeta_1, \tilde y_2 + \zeta_2, \tilde y_3 + \zeta_3+\xi_j(\tilde y_1\zeta_2-\tilde y_2\zeta_1)\big)\Big|\,d\zeta\\
&\lesssim 2^{-\epsilon|j-k|} \|f\|_\infty, 
\end{align*}
%This is because: Let $w'=(w_1', w_2', w_3')$ be such that \begin{align*} &\quad e^{\tilde y_1X_j+\tilde y_2Y_j +\tilde y_3T_j} e^{-C_1(t)X_j-C_2(t)Y_j - C_3(t)T_j} e^{w_1'X_j + w_2'Y_j + w_3'T_j}\vardbtilde x\\ &=e^{-C_1(t)X_j-C_2(t)Y_j - C_3(t)T_j} e^{w_1X_j + w_2Y_j + w_3T_j}\vardbtilde x. \end{align*} Then \begin{align*} &\quad e^{w_1X_j + w_2Y_j + w_3T_j}\vardbtilde x \\ &= e^{\tilde y_1X_j+\tilde y_2Y_j +\big(\tilde y_3 - 2\xi_j(\tilde y_1C_2(t)-\tilde y_2C_1(t))\big)T_j} e^{w_1'X_j + w_2'Y_j + w_3'T_j}\vardbtilde x\\ &=e^{(\tilde y_1+w_1')X_j +(\tilde y_2+w_2')Y_j +\big(\tilde y_3 +w_3'+\xi_j(-2\tilde y_1C_2(t)+2\tilde y_2C_1(t)+\tilde y_1w_2'-\tilde y_2w_1')\big)T_j}\vardbtilde x. \end{align*} Thus $$ \big|\det \frac{\partial w}{\partial w'}\big|=1. $$ Hence \begin{align*} &\quad \int f\big(e^{-C_1(t)X_j-C_2(t)Y_j - C_3(t)T_j} e^{w_1X_j + w_2Y_j + w_3T_j}\vardbtilde x\big) \hat \chi(w)\,dw\\ &=\int f\big( e^{\tilde y_1X_j+\tilde y_2Y_j +\tilde y_3T_j} e^{-C_1(t)X_j-C_2(t)Y_j - C_3(t)T_j} e^{w_1'X_j + w_2'Y_j + w_3'T_j}\vardbtilde x \big) \\ &\qquad \qquad \qquad \qquad \cdot \hat \chi\big(\tilde y_1+w_1', \tilde y_2 + w_2', \tilde y_3 + w_3'+\xi_j(-2\tilde y_1C_2(t)+2\tilde y_2C_1(t)+\tilde y_1w_2'-\tilde y_2w_1')\big)\,dw' \end{align*}
where the value of $\epsilon>0$ changes but is indpendent of $j,k$, and the implicit constants are independent of $j,k$. By Proposition \ref{density} below, if we denote
$$
g_4(\cdot) = f\Big(x\cdot e^{\tilde y_1X_j + \tilde y_2 Y_j +\tilde y_3 T_j} \cdot \Big), 
$$
we have
\begin{align*}
I\!V &\lesssim \sup_{\substack{s\in [2,8]\\ y\in \text{supp}\,\varsigma_k}} \Big|\iiiint_{(t,u,r)\in R_4}g_4 \big(e^{\zeta_1X_j+\zeta_2Y_j+\zeta_3T_j} e^{-\tilde v_1X_j-\tilde v_2Y_j-\tilde v_3T_j}0\big)\phi_j(t)\phi_j(u)\bar \phi_j(r)\varsigma_k(v)\,dt\,du\,dr\,dv\Big|\\
%&=\sup_{\substack{r\in [2,8]\\ y\in \text{supp}\,\varsigma_k}} \Big|\iint g_4 \big(e^{z_1X_j+z_2Y_j+z_3T_j}e^{C_1(r)X_j+C_2(r)Y_j+C_3(r)T_j} e^{-\tilde v_1X_j-\tilde v_2Y_j-\tilde v_3T_j}0\big)\hat \kappa(z)\varsigma_k(v)\,dz\,dv\Big|\\
%&=\sup_{\substack{r\in [2,8]\\ y\in \text{supp}\,\varsigma_k}} \Big|\iint\Big(g_4 \big(e^{z_1X_j+z_2Y_j+z_3T_j}e^{C_1(r)X_j+C_2(r)Y_j+C_3(r)T_j}e^{-\tilde v_1X_j-\tilde v_2Y_j-\tilde v_3T_j}0\big)\\
%&\qquad \qquad \qquad \qquad \qquad \quad -g_4\big(e^{z_1X_j+z_2Y_j+z_3T_j}e^{C_1(r)X_j+C_2(r)Y_j+C_3(r)T_j}0\big)\Big)\hat \kappa(z)\varsigma_k(v)\,dz\,dv\Big|\\
&\lesssim \sup_{\substack{s\in [2,8]\\ y,v\in \text{supp}\,\varsigma_k}}\Big|\int\Big(g_4 \big(e^{\zeta_1X_j+\zeta_2Y_j+\zeta_3T_j}e^{-\tilde v_1X_j-\tilde v_2Y_j-\tilde v_3T_j}0\big) -g_4\big(e^{\zeta_1X_j+\zeta_2Y_j+\zeta_3T_j}0\big)\Big)\hat \chi(\zeta)\,d\zeta\Big|\\
&=  \sup_{\substack{s\in [2,8]\\y,v\in \text{supp}\,\varsigma_k}}\Big|\int g_4\big(e^{\zeta_1X_j+\zeta_2Y_j+\zeta_3T_j}e^{-\tilde v_1X_j-\tilde v_2Y_j-\tilde v_3T_j}0\big)\\
&\qquad \qquad \qquad \cdot \Big(\hat \chi(\zeta)  -\hat \chi\big(\zeta_1-\tilde v_1, \zeta_2-\tilde v_2, \zeta_3-\tilde v_3-\xi_j(\tilde v_2\zeta_1-\tilde v_1\zeta_2)\big)\Big)\,d\zeta\Big|\\
%&\leq \|f\|_\infty \sup_{\substack{r\in [2,8]\\v\in \text{supp}\,\varsigma_k}} \int \Big|\hat \kappa(z)  -\hat \kappa\big(z_1-\tilde v_1, z_2-\tilde v_2, z_3-\tilde v_3-\xi_j(-2\tilde v_1C_2(r)+2\tilde v_2C_1(r)+\tilde v_2z_1-\tilde v_1z_2)\big)\Big|\,dz\\
&\leq \|f\|_\infty \sup_{\substack{s\in [2,8]\\v\in \text{supp}\,\varsigma_k}} \int \Big|\hat \chi(\zeta)  -\hat \chi\big(\zeta_1-\tilde v_1, \zeta_2-\tilde v_2, \zeta_3-\tilde v_3-\xi_j(\tilde v_2\zeta_1-\tilde v_1\zeta_2)\big)\Big|\,d\zeta\\
&\lesssim 2^{-\epsilon|j-k|}\|f\|_\infty,
\end{align*}
where the value of $\epsilon>0$ changes but is indpendent of $j,k$, and the implicit constants are independent of $j,k$.
%This is because: Let $z'=(z_1',z_2', z_3')$ be such that \begin{align*} &\quad e^{z_1'X_j+z_2'Y_j+z_3'T_j}e^{C_1(r)X_j+C_2(r)Y_j+C_3(r)T_j}e^{-\tilde v_1X_j-\tilde v_2Y_j-\tilde v_3T_j}0\\ &=e^{z_1X_j+z_2Y_j+z_3T_j}e^{C_1(r)X_j+C_2(r)Y_j+C_3(r)T_j}0. \end{align*} Then \begin{align*} &\quad e^{z_1X_j+z_2Y_j+z_3T_j}e^{C_1(r)X_j+C_2(r)Y_j+C_3(r)T_j}0\\ &= e^{z_1'X_j+z_2'Y_j+z_3'T_j}e^{C_1(r)X_j+C_2(r)Y_j+C_3(r)T_j}e^{-\tilde v_1X_j-\tilde v_2Y_j-\tilde v_3T_j} e^{-C_1(r)X_j-C_2(r)Y_j-C_3(r)T_j} \\ &\qquad \qquad \qquad \qquad \qquad \qquad \qquad \qquad \qquad \qquad \qquad \qquad \qquad e^{C_1(r)X_j+C_2(r)Y_j+C_3(r)T_j}0\\ &= e^{z_1'X_j+z_2'Y_j+z_3'T_j}e^{-\tilde v_1X_j-\tilde v_2Y_j-\big(\tilde v_3-2\xi_j(\tilde v_1C_2(r)-\tilde v_2 C_1(r))\big)T_j} e^{C_1(r)X_j+C_2(r)Y_j+C_3(r)T_j}0\\ &= e^{(z_1'-\tilde v_1)X_j+(z_2'-\tilde v_2)Y_j+\big(z_3'-\tilde v_3-\xi_j(-2\tilde v_1C_2(r)+2\tilde v_2C_1(r)+\tilde v_2z_1'-\tilde v_1z_2')\big)T_j} e^{C_1(r)X_j+C_2(r)Y_j+C_3(r)T_j}0. \end{align*} Thus  $$ \big|\det \frac{\partial z}{\partial z'}\big|=1. $$ Hence \begin{align*} &\quad \int g_4\big(e^{z_1X_j+z_2Y_j+z_3T_j}e^{C_1(r)X_j+C_2(r)Y_j+C_3(r)T_j}0\big)\hat \kappa(z)\,dz\\ &=\int g_4\big(e^{z_1'X_j+z_2'Y_j+z_3'T_j}e^{C_1(r)X_j+C_2(r)Y_j+C_3(r)T_j}e^{-\tilde v_1X_j-\tilde v_2Y_j-\tilde v_3T_j}0\big) \\ &\qquad \qquad \qquad \qquad \hat \kappa\big(z_1'-\tilde v_1, z_2'-\tilde v_2, z_3'-\tilde v_3-\xi_j(-2\tilde v_1C_2(r)+2\tilde v_2C_1(r)+\tilde v_2z_1'-\tilde v_1z_2')\big)\,dz' \end{align*}
\end{proof}

The following proposition establishes certain regularity of the densities $\kappa, \hat \kappa, \chi, \hat \chi$.
\begin{proposition}\label{density}
For every $j\in \mathbb{N}$, the measures $\mu, \hat \mu, \nu, \hat \nu$ defined in (\ref{measure}) are absolutely continuous with respect to the Euclidean measure on $\mathbb{R}^3$, and their densities $\kappa, \hat \kappa, \chi, \hat \chi$ sastisfy: 
for $u,s\in [2,8]$ ($\chi$ depends on $u$, $\hat \chi$ depends on $s$),
\begin{align*}
\int |\kappa(z)|\,dz \lesssim 1,\quad \int |\hat \kappa(w)|\,dw\lesssim 1, \quad \int |\chi(\zeta)|\,d\zeta \lesssim 1,  \quad \int |\hat \chi(\zeta)|\,d\zeta \lesssim 1, 
\end{align*}
where the implicit constants are independent of $j$. Moreover, there exists $\sigma>0$ such that for every $j\in \mathbb{N}$, for $0\leq \delta\leq 1$, for $u,s\in [2,8]$, and for $|\tilde y|,|\tilde v|\lesssim \delta$,
\begin{align*}
&\int \big|\kappa(z)-\kappa\big(\tilde y_1 +z_1, \tilde y_2+z_2, \tilde y_3+z_3+\xi_j(\tilde y_1 z_2-\tilde y_2 z_1)\big) \big|\,dz \lesssim \delta^\sigma, \\
&\int \big|\hat \kappa(w)  -\hat \kappa\big(w_1-\tilde v_1, w_2-\tilde v_2, w_3-\tilde v_3-\xi_j(w_1\tilde v_2-w_2\tilde v_1)\big)\big|\,dw \lesssim \delta^\sigma,\\
&\int \big|\chi(\zeta)-\chi\big(\tilde y_1+\zeta_1, \tilde y_2+\zeta_2, \tilde y_3+\zeta_3+\xi_j(\tilde y_1\zeta_2-\tilde y_2\zeta_1) \big)\big|\,d\zeta \lesssim \delta^\sigma,\\
&\int \big|\hat \chi(\zeta)- \hat \chi\big(\zeta_1-\tilde v_1, + \zeta_2-\tilde v_2,  \zeta_3-\tilde v_3-\xi_j(\zeta_1 \tilde v_2-\zeta_2\tilde v_1)\big)\big|\,d\zeta \lesssim \delta^\sigma.
\end{align*}
where the implicit constants are independent of $j$.
\end{proposition}

\subsection{Regularity of the density}\label{4.1}
The goal of this subsection is to prove Proposition \ref{density}. We first need several lemmas for $C_1(t), C_2(t)$, and $C_3(t)$ defined in (\ref{coefficient}).
\begin{definition}
Recall $\bar \beta(t) = \beta(t)-\frac{1}{2}\int_0^t h_2$, $\bar h_\beta(t) = t\bar \beta'(t)-\bar \beta(t)$. Denote
$$
\bar \beta^+(t) = \beta(t) + \frac{1}{2} \int_0^t h_2, \qquad \bar h_\beta^+(t) = t(\bar \beta^+)'(t)-\bar \beta^+(t),
$$
and
\begin{align*}
&C_3^-(t) = \frac{\bar \beta(2^{-j}t)-t\bar \beta(2^{-j}) - (\alpha(2^{-j}t)-t\alpha(2^{-j}))\bar h_\beta(2^{-j})/h_2(2^{-j})}{2^{-j}h_2(2^{-j}) + \bar h_3(2^{-j})},\\
&C_3^+(t) = \frac{\bar \beta^+(2^{-j}t)-t\bar \beta^+(2^{-j}) - (\alpha(2^{-j}t)-t\alpha(2^{-j}))\bar h_\beta^+(2^{-j})/h_2(2^{-j})}{2^{-j}h_2(2^{-j}) + \bar h_3(2^{-j})},\\
&\hat C_3(t)= C_3(t)+\xi_jtC_2(t),\\
&\doublehat C_3(t) = C_3(t)-\xi_j tC_2(t).
\end{align*}
\end{definition}

We have
\begin{equation}\label{C3}
\begin{aligned}
C_3'(t) - \xi_j\big(tC_2'(t)-C_2(t)\big) &= (C_3^-)'(t),\\
C_3'(t) + \xi_j\big(tC_2'(t)-C_2(t)\big) &= (C_3^+)'(t) +\frac{2^{-j}h_2(2^{-j})-\int_0^{2^{-j}}h_2}{2^{-j}h_2(2^{-j})+\bar h_3(2^{-j})}C_2'(t) + \frac{\int_0^{2^{-j}}h_2}{2^{-j}h_2(2^{-j})+\bar h_3(2^{-j})},\\
C_3'(t) + \xi_j\big(tC_2'(t)+C_2(t)\big) &=\hat C_3'(t), \\
C_3'(t)-\xi_j\big(tC_2'(t)+C_2(t)\big)&=\doublehat C_3'(t),
\end{aligned}
\end{equation}
where
\begin{equation}\label{smallsize}
0\leq \frac{2^{-j}h_2(2^{-j})-\int_0^{2^{-j}}h_2}{2^{-j}h_2(2^{-j})+\bar h_3(2^{-j})}, \quad \frac{\int_0^{2^{-j}}h_2}{2^{-j}h_2(2^{-j})+\bar h_3(2^{-j})}\leq 1.
\end{equation}

\begin{lemma}\label{monotone}
For $t\in [2,8]$, 
$$
C_2(t), C_2'(t), C_2''(t), C_3^-(t), (C_3^-)'(t), (C_3^-)''(t)>0,
$$
and $C_2(t) \leq 7C_2'(t)$.
\end{lemma}

\begin{lemma}\label{standardmonotone}
For $t\in [2,8]$,
$$
\Big(\frac{(C_3^-)''(t)}{C_2''(t)}\Big)', \quad \Big(\frac{(C_3^+)''(t)}{C_2''(t)}\Big)', \quad \Big(\frac{\hat C_3''(t)}{C_2''(t)}\Big)', \quad \Big(\frac{\doublehat C_3''(t)}{C_2''(t)}\Big)'>0.
$$
\end{lemma}

The following lemma is from Proposition 3.1 of \cite{VANCE}, of which we need only the $n=3$ case:
\begin{lemma}[\cite{VANCE}]\label{mj}
For a function $f\in C^2(0,\infty)$, denote
$$
\Delta_jf(t):= \Big(\frac{f'(t)}{C_2''(t)}\Big)' \frac{h_2(2^{-j}t)}{h_2(2^{-j})}, \quad j\in \mathbb{N}.
$$
If there exists $j\in \mathbb{N}$ such that $\Delta_jf(t)$ does not change sign on $[2,8]$, and
$$
|f(t)| +|tf'(t)|+|\Delta_jf(t)| \geq 2\lambda_0, \quad \forall t\in [2,8],
$$
for some $\lambda_0>0$, then for every $0\leq \lambda \leq \lambda_0$,
$$
|\{t\in [2,8]: |f(t)|<\lambda\}|\lesssim\lambda^{\frac{1}{2}},
$$
where the implicit constant depends only on $\lambda_0$ and is independent of $j$.
\end{lemma}

\begin{lemma}\label{lowerbound}
%Recall the matrices $A(t), \widetilde A_j$ from Definitions \ref{matrix} and \ref{jmatrix}. We have
%\begin{align*} \frac{d}{dt}\big(\widetilde A_j^{-1} A(2^{-j}t)\big)= \begin{pmatrix} 1 & 0 & \quad 0\\ C_2'(t) & tC_2''(t) & \quad 0\\ (C_3^-)'(t) & t(C_3^-)''(t) & \quad \big(\frac{(C_3^-)''(t)}{C_2''(t)}\big)' \frac{h_2(2^{-j}t)}{h_2(2^{-j})}, \end{pmatrix} \end{align*}
For each $j\in \mathbb{N}$, denote
\begin{align*}
B_j^-(t)=
\begin{pmatrix}
1 & 0 & \quad 0\\
C_2'(t) & tC_2''(t) & \quad 0\\
(C_3^-)'(t) & t(C_3^-)''(t) & \quad \big(\frac{(C_3^-)''(t)}{C_2''(t)}\big)' \frac{h_2(2^{-j}t)}{h_2(2^{-j})},
\end{pmatrix},\\
B_j^+ (t)=
\begin{pmatrix}
1 & 0 & \quad 0\\
C_2'(t) & tC_2''(t) & \quad 0\\
(C_3^+)'(t) & t(C_3^+)''(t) & \quad \big(\frac{(C_3^+)''(t)}{C_2''(t)}\big)' \frac{h_2(2^{-j}t)}{h_2(2^{-j})},
\end{pmatrix}.
\end{align*}
we have
$$
\|(
B_j^-(t))^{-1}\|, \quad \|(
B_j^+(t))^{-1}\|\lesssim 1, \quad t\in [2,8],
$$
where the implicit constant is independent of $j$.
\end{lemma}

\begin{definition}\label{product}
For each $j\in \mathbb{N}$, we denote
$$
x \odot_j y= \big(x_1+y_1, x_2+y_2, x_3+y_3 + \xi_j(x_1y_2-x_2y_1)\big).
$$
\end{definition}

\begin{proof}[Proof of Proposition \ref{density}]
Using the notation in Definition \ref{product}, what we need to prove is that there exists $\sigma>0$ such that for each $j\in \mathbb{N}$,
\begin{align*}
&\int |\kappa(z)|\,dz \lesssim 1, \quad \int \big|\kappa(z)-\kappa\big(\tilde y \odot_j z\big) \big|\,dz \lesssim \delta^\sigma,\\
&\int |\hat \kappa(w)|\,dw \lesssim 1, \quad \int \big|\hat \kappa(w)-\hat \kappa \big(w\odot_j \tilde v^{-1}\big)\big|\,dw\lesssim \delta^\sigma,\\
&\int |\chi(\zeta)|\,d\zeta\lesssim 1, \quad \int \big|\chi(\zeta)-\chi \big(\tilde y \odot_j \zeta\big)\big|\,d\zeta \lesssim \delta^\sigma,\\
&\int |\hat \chi(\zeta)|\,d\zeta\lesssim 1, \quad \int \big| \hat \chi(\zeta)-\hat \chi \big(\zeta\odot_j  \tilde v^{-1}\big)\big|\,d\zeta \lesssim \delta^\sigma,
\end{align*}
for $0\leq |\tilde y|, |\tilde v|\lesssim \delta \leq 1$, where the implicit constant is independent of $j$. We only prove the inequalities for $\kappa$. The proof of the inequalities for $\hat \kappa, \chi, \hat \chi$ is similar (see Remark \ref{modification}). 

Denote $J:=\det D\Phi(t,s,u)$. By (\ref{determinant}),
\begin{align*}
J(t,s,u)
%&=\det \begin{pmatrix} -1 & 1 & -1\\ &&\\ -C_2'(t) & C_2'(s) & -C_2'(u)\\ &&\\ -C_3'(t)+\xi_j\big( sC_2'(t) & C_3'(s)+\xi_j\big(C_2(t) & -C_3'(u)+\xi_j\big(tC_2'(u)\\ -C_2(s)+ C_2(u)-uC_2'(t) \big) & -tC_2'(s)-C_2(u)+uC_2'(s) \big)  &-sC_2'(u)-C_2(t)+C_2(s) \big) \end{pmatrix}   \\
&=C_2'(s)(C_3^-)'(u) - C_2'(u)\doublehat C_3'(s) +C_2'(t)\hat C_3'(s)-C_2'(s)\big(C_3'(t)+\xi_j\big(tC_2'(t)-C_2(t)\big)\big)\\
&\quad -C_2'(t)(C_3^-)'(u) +C_2'(u)\big(C_3'(t)+\xi_j\big(tC_2'(t)-C_2(t)\big)\big) -2\xi_jC_2(s)C_2'(s)-2\xi_jsC_2'(t)C_2'(u).
\end{align*}
Note $J(t,s,u)$ depends on $j$. Denote $W^u:= (W_1^u, W_2^u, W_3^u)$ and
\begin{equation}\label{wu}
\begin{aligned}
W_1^u&:= C_2'(t)\hat C_3'(s)-C_2'(s)\big(C_3'(t)+\xi_j\big(tC_2'(t)-C_2(t)\big)\big)-2\xi_jC_2(s)C_2'(s),\\
W_2^u&:= \big(C_3'(t)+\xi_j\big(tC_2'(t)-C_2(t)\big)\big) -\doublehat C_3'(s) -2\xi_j sC_2'(t)_,\\
W_3^u&:=C_2'(s)-C_2'(t).
\end{aligned}
\end{equation}
We have
\begin{align*}
J=W_1^u+ C_2'(u) W_2^u+ (C_3^-)'(u) W_3^u,
\end{align*}
and $W^u$ is independent of $u$. By Lemma \ref{monotone},
$$
\{(t,s)\in [2,8]^2: |W^u|=0\}\subseteq \{(t,s)\in [2,8]^2: W_3^u=0\}= \{(t,s)\in [2,8]^2: t=s\},
$$
is of measure $0$. Recall functions $C_2(u), C_3^-(u)$ depend on $j$. For this $j$, and for $(t,s)\in [2,8]^2$ with $t\neq s$, we can apply Lemma \ref{mj} to
$$
f(u):=\frac{J}{|W^u|},
$$
as a functions of $u$. In fact, by Lemma \ref{standardmonotone},
$$
\Delta_jf(u) = \frac{W_3^u}{|W^u|} \Big( \frac{(C_3^-)''(u)}{C_2''(u)}\Big)' \frac{h_2(2^{-j}u)}{h_2(2^{-j})}
$$
does not change sign for $u\in[2,8]$. Since 
$$
\frac{1}{|W^u|}(W_1^u, W_2^u, W_3^u) B_j^-(u) = \big(f(u), uf'(u), \Delta_jf(u)\big),
$$
by Lemma \ref{lowerbound},
$$
1\leq \|(B_j^-(u))^{-1}\|\cdot \big|\big(f(u), uf'(u), \Delta_jf(u)\big)\big| \lesssim \big|\big(f(u), uf'(u), \Delta_jf(u)\big)\big|,
$$
where the implicit constant is independent of $j, t,s$. Therefore by Lemma \ref{mj}, there exists $\lambda_0>0$ independent of $j, t,s$, such that for every $0< \lambda \leq \lambda_0$,
\begin{equation}\label{measureu}
|\{u\in [2,8]: \frac{|J|}{|W^u|}<\lambda\}|\lesssim \lambda^{\frac{1}{2}},
\end{equation}
where the implicit constant is independent of $j, t, s (t\neq s)$. In particular, for almost every $(t,s)\in [2,8]^2$, 
$$
|\{u\in [2,8]: J=0\}|=0,
$$
and hence by Fubini's theorem $J=0$ for almost every $(t,s,u)\in [2,8]^3$.

Similarly, by denoting $W^t=(W_1^t, W_2^t, W_3^t)$, $\widetilde W^t = (\widetilde W_1^t, \widetilde W_2^t, \widetilde W_3^t)$, and
\begin{equation}\label{wt}
\begin{aligned}
&W_1^t= C_2'(s)(C_3^-)'(u)-C_2'(u) \doublehat C_3'(s)-2\xi_jC_2(s)C_2'(s), \quad \widetilde W_1^t=\frac{\int_0^{2^{-j}}h_2}{2^{-j}h_2(2^{-j})+\bar h_3(2^{-j})}W_3^t+ W_1^t,\\
&W_2^t= \hat C_3'(s)-(C_3^-)'(u)-2\xi_jsC_2'(u), \quad \widetilde W_2^t=\frac{2^{-j}h_2(2^{-j})-\int_0^{2^{-j}}h_2}{2^{-j}h_2(2^{-j})+\bar h_3(2^{-j})}W_3^t+ W_2^t,\\
&W_3^t= C_2'(u) -C_2'(s), \quad \widetilde W_3^t=W_3^t,
\end{aligned}
\end{equation}
due to (\ref{C3}),
\begin{equation}\label{combo}
J= W_1^t+ C_2'(t) W_2^t+ \Big( C_3'(t) +\xi_j \big(tC_2'(t)-C_2(t)\big)\Big) W_3^t =\widetilde W_1^t + C_2'(t) \widetilde W_2^t+(C_3^+)'(t)\widetilde W_3^t,
\end{equation}
where $W^t, \widetilde W^t$ are independent of $t$. We have $|W^t|\lesssim |\widetilde W^t| \lesssim |W^t|$ by (\ref{smallsize}). For $(s,u)\in [2,8]$ with $s\neq u$, by letting
$$
f(t):=\frac{J}{|W^t|}
$$
in Lemma \ref{mj}, since
$$
\Delta_jf(t) = \frac{\widetilde W_3^t}{|W^t|} \Big( \frac{(C_3^+)''(t)}{C_2''(t)}\Big)' \frac{h_2(2^{-j}t)}{h_2(2^{-j})}
$$
does not change sign for $t\in[2,8]$, and 
$$
1\lesssim \frac{1}{|W^t|}|(\widetilde W_1^t, \widetilde W_2^t, \widetilde W_3^t) | =\big|\big(f(t), tf'(t), \Delta_jf(t)\big)(B_j^+(t))^{-1}\big| \lesssim \big|\big(f(t), tf'(t), \Delta_jf(t)\big)\big|,
$$
we have there exists $\lambda_0>0$ independent of $j, s,u$, such that for every $0< \lambda \leq \lambda_0$,
\begin{equation}\label{measuret}
|\{t\in [2,8]: \frac{|J|}{|W^t|}<\lambda\}|\lesssim \lambda^{\frac{1}{2}},
\end{equation}
where the implicit constant is independent of $j, u, s (u\neq s)$.

Recall $R_1=\{(t,s,u)\in [2,8]^3: t\leq u\}$. let $\partial R_1$ denote the boundary of $R_1$. Define the following hyperplanes in $\mathbb{R}^3$: 
\begin{align*}
&L_1=\{(t,s,u): u=8\}, \quad L_2=\{(t,s,u):t=2\}, \\
&L_3=\{(t,s,u): t=u\}, \quad L_4=\{(t,s,u): s=2\},\\
&L_5=\{(t,s,u):s=8\}. 
\end{align*}
We have
$$
\begin{aligned}
\partial R_1 &= \{(t,s,u)\in L_1: 2\leq t,s\leq 8\}\cup \{(t,s,u)\in L_2: 2\leq u,s\leq 8\}\\
&\quad \cup \{(t,s,u)\in L_3: 2\leq t,s\leq 8\} \cup \{(t,s,u)\in L_4: 2\leq t\leq u\leq 8\}\\
&\quad \cup \{(t,s,u)\in L_5: 2 \leq t\leq u \leq 8\}.
\end{aligned}
$$
\color{black}
Thus $\partial R_1$ is of measure $0$. Denote $\tau = (t,s,u)$.  For every $0<\sigma <1, 1\leq i \leq 5, \tau\in R_1$,
\begin{equation}\label{distance}
\text{dist}\,(\tau, \partial R_1)\geq \min_i \text{dist}\,(\tau,L_i), \qquad |\nabla_\tau\text{dist}\,(\tau,L_i)|\lesssim 1, \qquad \int_{R_1} \frac{1}{\text{dist}\,(\tau,L_i)^\sigma}\,d\tau \lesssim_\sigma 1.
\end{equation}

Since the curve $\Gamma(t) \in C^3(0,\infty)$, we have $D\Phi(\tau), J(\tau)\in C^2$ on a neighborhood of $[2,8]^3$. Thus we can let
$$
M_j = \max\{1, \sup_{[2,8]^3}| \nabla J|, \sup_{[2,8]^3} |2\times 2 \text{ minors of } D\Phi|, \sup_{[2,8]^3} |\text{derivative of each entry of }D\Phi|\}<\infty.
$$
For every $\tau \in R_1$ with $J(\tau) \neq 0$, construct a Euclidean ball $B(\tau, \frac{|J(\tau)|}{80 M_j^2})$. For $\tau' \in B(\tau, \frac{|J(\tau)|}{10 M_j^2})$,
$$
|J(\tau')-J(\tau)| \leq \frac{1}{10}|J(\tau)|.
$$
If $\tau_1, \tau_2\in R_1\backslash\{J=0\}$ with $B(\tau_1, \frac{|J(\tau_1)|}{20 M_j^2})\cap B(\tau_2, \frac{|J(\tau_2)|}{20 M_j^2})\neq \emptyset$, then either $\tau_1\in B(\tau_2, \frac{|J(\tau_2)|}{10 M_j^2})$ or $\tau_2\in B(\tau_1, \frac{|J(\tau_1)|}{10 M_j^2})$. Thus
$$
\frac{9}{11}|J(\tau_1)| \leq |J(\tau_2)| \leq \frac{11}{9}|J(\tau_1)|.
$$

Consider a maximal disjoint collection of balls $\{B(\tau_m, c\frac{|J(\tau_m)|}{80 M_j^2}): \tau_m\in R_1, J(\tau_m)\neq 0\}$, where $0<c<1$ is a small constant to be determined. Denote $B_m := B(\tau_m, c\frac{|J(\tau_m)|}{20 M_j^2})$. We claim that
$$
R_1\backslash \{J=0\} \subseteq \bigcup_m B_m.
$$
In fact, if $\tau\in R_1$ with $J(\tau)\neq 0$, either $\tau$ is one of the center $\tau_m$, or it is not. In the latter case, the ball $B(\tau, c\frac{J(\tau)}{80M_j^2})$ must intersect one of the balls $B(\tau_m, c\frac{J(\tau_m)}{80M_j^2})$, then $J(\tau)\leq \frac{11}{9}J(\tau_m)$, and hence $\tau \in B_m$.

There exists $K<\infty$, independent of $j$, such that no point is covered by more than $K$ of the balls $B_m$. In fact, all balls $B(\tau_m, c\frac{J(\tau_m)}{80M_j^2})$ for which $\tau \in B_m$ have comparable radii, are disjoint, and are contained in a ball of comparable radius. Similarly, each $B_m$ intersects with boundedly many balls in $\{B_m\}$.

Note $\Phi$ has the Jacobian $\geq \frac{9}{10}|J(\tau_m)|>0$ on $B_m$. Restricted on $B_m$, let
$$
G(\tau,z) = \tau + D\Phi(\tau_m)^{-1}(z- \Phi(\tau))
$$
Then 
\begin{align*}
\frac{\partial G}{\partial \tau} = I - D\Phi(\tau_m)^{-1} D\Phi(\tau) = D\Phi(\tau_m)^{-1} \big( D\Phi(\tau_m) - D\Phi(\tau)\big)
\end{align*}
By Cramer's rule, 
\begin{align*}
&\|D\Phi(\tau_m)^{-1}\| \lesssim \frac{M_j}{|J(\tau_m)|},\\
&\|D\Phi(\tau_m) - D\Phi(\tau)\| \lesssim M_j\cdot \frac{c|J(\tau_m)|}{10M_j^2},
\end{align*}
where $\|\cdot\|$ denotes the Frobenius norm of the matrix. Thus 
$$
\Big\|\frac{\partial G}{\partial \tau}\Big\| \lesssim c
$$
Thus by choosing $c$ sufficiently small, the contraction mapping condition on $G$ is satisfied on $B_m$, so that $\Phi$ is invertible on $B_m$. Denote the inverse map of $\Phi|_{B_m}$ to be $\Phi_m^{-1}$.

We then choose a partition of unity $\{\eta_m\}$, consisting of $C^\infty$ functions supported in $B_m$, $0\leq \eta_m \leq 1$, with $\sum \eta_m =1$ on $R_1\backslash\{J=0\}$. Let 
$$
d\mu_m = \Phi_*\big(\phi_j(t)\phi_j(s)\phi_j(u)\eta_m(\tau)1_{R_1}(\tau)\,d\tau\big).
$$ 
Denote $Q(\tau)= |J(\tau)|^{-1} \phi_j(t)\phi_j(s)\phi_j(u) \in C^2([2,8]^3\backslash \{J=0\})$. Then 
$$
d\mu_m(z) = 
\left\{
\begin{aligned}
&Q(\Phi_m^{-1}(z)) \eta_m(\Phi_m^{-1}(z))1_{R_1}(\Phi_m^{-1}(z))\,dz, \quad z\in \Phi(B_m),\\
&0, \quad z\not\in \Phi(B_m).
\end{aligned}
\right.
$$
Denote its density by $\kappa_m(z)$. 
%For notational simplicity, we will write $\kappa_m(z)= Q(\Phi_m^{-1}(z)) \eta_m(\Phi_m^{-1}(z))1_{R_1}(\Phi_m^{-1}(z))$, since $\eta_m(\Phi_m^{-1}(z))=0$ can trivially extend to $z\not\in \Phi(B_m)$. 
We have
$$
\big\|\sum_m \kappa_m\big\|_1 \leq \sum \int |\kappa_m(z)|\,dz \lesssim \sum \int_{B_m}|\eta_m(\tau)| 1_{R_1}(\tau)\,d\tau \lesssim K.
$$
Therefore $\mu = \sum \mu_m$ with density $\kappa = \sum \kappa_m$ is absolutely continuous, and $\kappa\in L^1(\mathbb{R}^3)$ uniformly in $j$.

Now we start to show 
$$
\int \big|\kappa(z)-\kappa \big(\tilde y\odot_j z\big) \big|\,dz \lesssim \delta^\sigma, \quad \forall 0\leq |\tilde y|\lesssim \delta \leq 1,
$$
for some $\sigma >0$. Fix a $C^\infty$ function $0\leq \rho_0\leq 1$ supported on $[1,\infty)$ with $\rho_0\equiv 1$ on $[2,\infty)$. For $l\geq 1$, let $\rho_l(r)= \rho_0(2^l r)-\rho_0(2^{l-1}r)\in C_c^\infty[2^{-l}, 2^{2-l}]$. Then $\sum_{l\geq 0} \rho_l =1_{(0,\infty)}$, and $0\leq \rho_l \leq 1$ for each $l$. For $\tau \in [2,8]^3$, define 
\begin{equation}\label{bigf}
0\leq F(\tau)=\frac{1}{\sqrt{(\frac{1}{J})^2+(\frac{|W^t|}{J})^2+(\frac{|W^u|}{J})^2 +(\frac{2\xi_jC_2'(t)C_2'(u)}{J})^2+(\frac{2\xi_jC_2(s)C_2'(s)}{J})^2 +\sum_{i=1}^5\frac{1}{\text{dist}\,(\tau,L_i)^2}}} <\infty.
\end{equation}
We have $F(\tau)>0$ for almost every $\tau\in [2,8]^3$.

\begin{remark}\label{before}
Although \cite{withoutfourier} proves the $\mathbb{R}^n$ case without using the Fourier transform, it uses the fact that the zero set of the Jacobian $J$ in the $\mathbb{R}^n$ case is the union of $\binom{n}{2}$ many hyperplanes $H_i$. The $B_m$, which decomposes the region $[2,8]^n\backslash\{J=0\}$, is taken to be with the radius comparable to its distance to $\{J=0\}$. An indispensable estimate used in \cite{withoutfourier} is that for such $\{B_m\}$,
\begin{equation}\label{indispensable}
\sum_m |B_m|^{1-\sigma} \lesssim 1, \quad \text{for some } \sigma >0.
\end{equation}
In fact,
\begin{align*}
\sum_m |B_m|^{1-\sigma} &\approx \sum_m \int_{B_m} \frac{1}{\text{dist}\,(\tau, \{J=0\})^{n\sigma}}\,d\tau\\
&\approx \int_{[2,8]^n} \frac{1}{\text{dist}\,(\tau, \{J=0\})^{n\sigma}}\,d\tau \leq \int_{[2,8]^n} \sum_{i=1}^{\binom{n}{2}}\frac{1}{\text{dist}\,(\tau, H_i)^{n\sigma}}\,d\tau \lesssim_{n,\sigma} 1, \quad \text{for every } 0<\sigma<\frac{1}{n},
\end{align*}
where the implicit constant is independent of $j$. Such estimate requires the fact that the $(n-1)$-dimensional Hausdorff measure of $\{J=0\}$ is bounded as $j$ varies. However, in cases that are not translation-invariant, we no longer have $\{J=0\}$ is the union of hyperplanes, and the same fact seems not to hold for the zero set $\{J=0\}$, even though our Jacobian $J$ is a perturbation of the Jacobian in the $\mathbb{R}^n$ case. Since (\ref{indispensable}) is not available for cases that are not translation-invariant, we need to take an alternative route in the rest of the proof, and that is the reason we construct the function $F(\tau)$.
\end{remark}

Let $\psi_l(\tau)=\rho_l(F(\tau))1_{R_1}(\tau)\in C^2$. Then $\sum_{l\geq 0} \psi_l(\tau)=1$ for almost every $\tau\in R_1$, and $0\leq \psi_l \leq 1$ for each $l$. Define
\begin{equation}\label{redecompose}
\bar \kappa_l(z)=\sum_m \kappa_m(z)\psi_l(\Phi_m^{-1}(z)) =
\sum_m Q(\Phi_m^{-1}(z)) \eta_m(\Phi_m^{-1}(z)) \psi_l(\Phi_m^{-1}(z)).
\end{equation}
Since
$$
\int \sum_{l,m} |\kappa_m(z)\psi_l(\Phi_m^{-1}(z))|=\int \sum_{l,m} \big| Q(\Phi_m^{-1}(z)) \eta_m(\Phi_m^{-1}(z)) \psi_l(\Phi_m^{-1}(z)) \big|\,dz \lesssim \sum_{l,m} \int \eta_m(\tau)\psi_l(\tau)\,d\tau \lesssim 1,
$$
we have $\sum_{l,m} \big| \kappa_m(z)\psi_l(\Phi_m^{-1}(z))\big|<\infty$ for almost every $z$.
Thus for almost every $z$,
$$
\kappa(z)=\sum_m\kappa_m(z)= \sum_m\kappa_m(z)\cdot \big(\sum_l\psi_l(\Phi_m^{-1}(z))\big)=\sum_l \bar \kappa_l(z).
$$
Therefore
$$
\int |\kappa(z)-\kappa(\tilde y\odot_j z) \big|\,dz \leq \sum_l \int \big|\bar \kappa_l(z)-\bar \kappa_l(\tilde y \odot_j z)\big|\,dz.
$$

For every $l\geq 0$, define the region 
\begin{equation}\label{el}E_l:=\{\tau\in R_1: F(\tau)\in [2^{-l}, 2^{2-l}]\}.
\end{equation}
We have
$$
\text{supp}\, \psi_l \subseteq E_l \subseteq \{\tau\in R_1: |J|\geq 2^{-l}, \text{dist}\,(\tau, \partial R_1) \geq 2^{-l}\}
$$ 
is a compact subset of the interior of $R_1\backslash\{J=0\}$. Thus $\text{supp}\,\psi_l$ can be covered by finitely many $B_m$. Recall each $B_m$ intersects with boundedly many balls in $\{B_m\}$. Hence $\text{supp}\,\psi_l$ intersects with only finitely many $B_m$. Therefore (\ref{redecompose}) is a finite sum. 
%make it clear by set index set $I_l$ for the sum.
By Proposition \ref{good} below, there exists $0<\sigma <1$ such that
$$
|E_l| \leq 2^{-\sigma(l-2)}\int_{R_1} F(\tau)^{-\sigma}\,d\tau \lesssim 2^{-\sigma l},
$$
where the implicit constant is independent of $j$. Hence
\begin{align*}
&\quad \int \big|\bar \kappa_l(z)- \bar \kappa_l(\tilde y \odot_j z) \big|\,dz \leq 2\int |\bar \kappa_l(z)|\,dz \leq 2\sum_m \int \big|Q(\Phi_m^{-1}(z))\eta_m(\Phi_m^{-1}(z)) \psi_l(\Phi_m^{-1}(z))\big|\,dz  \\
&\lesssim \sum_m \int \eta_m(\tau)\psi_l(\tau)\,d\tau= \int \sum_m \eta_m(\tau)\psi_l(\tau)\,d\tau=\int \psi_l(\tau)\,d\tau \leq |E_l| \lesssim 2^{-\sigma l}.
\end{align*}

We have
\begin{align*}
&\quad \int \big|\bar \kappa_l(z)- \bar \kappa_l(\tilde y \odot_j z) \big|\,dz \\
&=\int\Big| \sum_m \Big( Q(\Phi_m^{-1}(z))\eta_m(\Phi_m^{-1}(z))\psi_l(\Phi_m^{-1}(z)) - Q(\Phi_m^{-1}(\tilde y\odot_j z))\eta_m(\Phi_m^{-1}(\tilde y \odot_j z))\psi_l(\Phi_m^{-1}(\tilde y \odot_j z))\Big)\Big|\,dz\\
&=\int\Big| \sum_m \int_0^1 \frac{d}{d\theta} \Big(Q(\Phi_m^{-1}(\theta \tilde y\odot_j z))\eta_m(\Phi_m^{-1}(\theta \tilde y \odot_j z))\psi_l(\Phi_m^{-1}(\theta \tilde y \odot_j z))\Big)\,d\theta \Big|\,dz\\
&\leq \int\Big| \sum_m \int_0^1  Q(\Phi_m^{-1}(\theta \tilde y\odot_j z))\eta_m(\Phi_m^{-1}(\theta \tilde y \odot_j z))\frac{d}{d\theta}\psi_l(\Phi_m^{-1}(\theta \tilde y \odot_j z))\,d\theta \Big|\,dz\\
&\quad + \int \Big| \sum_m \int_0^1 \eta_m(\Phi_m^{-1}(\theta \tilde y \odot_j z)) \psi_l(\Phi_m^{-1}(\theta \tilde y \odot_j z)) \frac{d}{d\theta}Q(\Phi_m^{-1}(\theta \tilde y \odot_j z))\,d\theta\Big|\,dz \\
&\quad + \int \Big| \sum_m \int_0^1 Q(\Phi_m^{-1}(\theta \tilde y \odot_j z)) \psi_l(\Phi_m^{-1}(\theta \tilde y \odot_j z)) \frac{d}{d\theta}\eta_m(\Phi_m^{-1}(\theta \tilde y \odot_j z))\,d\theta\Big|\,dz\\
&=:P_1 + P_2 + P_3.
\end{align*}

Fix $\tilde y$ with $|\tilde y|\lesssim \delta \leq 1$. Define a vector field $U$ on $R_1\backslash\{J=0\}$ be such that for every $C^1$ function $f$,
\begin{equation}\label{vectorfield}
Uf(\tau)=\nabla f(\tau) \cdot D\Phi(\tau)^{-1} \cdot (\tilde y_1, \tilde y_2, \tilde y_3+\xi_j(\tilde y_1 z_2-\tilde y_2 z_1))^T,
\end{equation}
where the superscript $T$ denotes the transpose of a matrix, and $(z_1, z_2, z_3)=\Phi(\tau)$. Thus when $\theta \tilde y \cdot z =\Phi(\tau)$ for some $\tau \in B_m$,
$$
Uf(\tau)=\frac{d}{d\theta} f(\Phi_m^{-1}(\theta \tilde y \odot_j z)),
$$
which is independent of the choice of $B_m \ni \tau$, due to (\ref{vectorfield}).

Since $\text{supp}\,\psi_l$ intersects with only finitely many $B_m$ and (\ref{redecompose}) is a finite sum, we have
\begin{align*}
P_3&=\int \Big| \sum_m \int_0^1 Q(\Phi_m^{-1}(\theta \tilde y \odot_j z)) \psi_l(\Phi_m^{-1}(\theta \tilde y \odot_j z)) \frac{d}{d\theta}\eta_m(\Phi_m^{-1}(\theta \tilde y \odot_j z))\,d\theta\Big|\,dz\\
&=\int \Big|  \int_0^1 \sum_m Q(\Phi_m^{-1}(\theta \tilde y \odot_j z)) \psi_l(\Phi_m^{-1}(\theta \tilde y \odot_j z)) \frac{d}{d\theta}\eta_m(\Phi_m^{-1}(\theta \tilde y \odot_j z))\,d\theta\Big|\,dz\\
&=\int \Big|  \int_0^1 \sum_{\tau: \theta \tilde y\odot_j z= \Phi(\tau)} \sum_{m: \tau\in B_m} Q(\Phi_m^{-1}(\theta \tilde y \odot_j z)) \psi_l(\Phi_m^{-1}(\theta \tilde y \odot_j z)) \frac{d}{d\theta}\eta_m(\Phi_m^{-1}(\theta \tilde y \odot_j z))\,d\theta\Big|\,dz\\
&=\int \Big|  \int_0^1 \sum_{\tau: \theta \tilde y\odot_j z= \Phi(\tau)} \sum_{m: \tau\in B_m} Q(\tau) \psi_l(\tau) U\eta_m(\tau)\,d\theta\Big|\,dz\\
&=\int \Big|  \int_0^1 \sum_{\tau: \theta \tilde y\odot_j z= \Phi(\tau)} Q(\tau) \psi_l(\tau) U\Big(\sum_{m: \tau\in B_m} \eta_m\Big)(\tau)\,d\theta\Big|\,dz=0.
\end{align*}

By Proposition \ref{estimateofderivative} below, we have
\begin{align*}
P_2 &= \int \Big| \sum_m \int_0^1 \eta_m(\Phi_m^{-1}(\theta \tilde y \odot_j z)) \psi_l(\Phi_m^{-1}(\theta \tilde y \odot_j z)) \frac{d}{d\theta}Q(\Phi_m^{-1}(\theta \tilde y \odot_j z))\,d\theta\Big|\,dz\\
&\leq \sum_m \int_0^1 \int  \Big|\eta_m(\Phi_m^{-1}(\theta \tilde y \odot_j z)) \psi_l(\Phi_m^{-1}(\theta \tilde y \odot_j z)) \frac{d}{d\theta}Q(\Phi_m^{-1}(\theta \tilde y \odot_j z))\Big|\,dz\,d\theta\\
&\lesssim \sum_m \int_0^1 \int \eta_m(\tau) \psi_l(\tau) |UQ(\tau) J(\tau)|\,d\tau\,d\theta\\
&= \sum_m \int \eta_m(\tau) \psi_l(\tau) |UQ(\tau) J(\tau)|\,d\tau\\
&= \int \sum_m \eta_m(\tau) \psi_l(\tau) |UQ(\tau) J(\tau)|\,d\tau\\
&= \int \psi_l(\tau) |UQ(\tau) J(\tau)|\,d\tau\\
&\leq \int_{E_l} |UQ(\tau) J(\tau)|\,d\tau \lesssim 2^{2l} \delta,
\end{align*}
and
\begin{align*}
P_1 &=\int\Big| \sum_m \int_0^1  Q(\Phi_m^{-1}(\theta \tilde y\odot_j z))\eta_m(\Phi_m^{-1}(\theta \tilde y \odot_j z))\frac{d}{d\theta}\psi_l(\Phi_m^{-1}(\theta \tilde y \odot_j z))\,d\theta \Big|\,dz\\
& \leq \sum_m \int_0^1 \int  \Big|Q(\Phi_m^{-1}(\theta \tilde y\odot_j z))\eta_m(\Phi_m^{-1}(\theta \tilde y \odot_j z))\frac{d}{d\theta}\psi_l(\Phi_m^{-1}(\theta \tilde y \odot_j z))\Big|\,dz\,d\theta\\
&\lesssim \sum_m \int_0^1 \int \eta_m(\tau) |U\psi_l(\tau)| \,d\tau\,d\theta\\
&=\sum_m \int \eta_m(\tau)|U\psi_l(\tau)|\,d\tau\\
&=\int \sum_m \eta_m(\tau)|U\psi_l(\tau)|\,d\tau\\
&=\int |U\psi_l(\tau)|\,d\tau \lesssim 2^{2l}\delta,
\end{align*}
where the implicit constants are independent of $j$. Therefore
\begin{align*}
\int \big|\kappa(z)-\kappa\big(\tilde y \odot_j z\big)\big|\,dz \leq \sum_l \int \big|\bar \kappa_l(z)- \bar \kappa_l \big(\tilde y \odot_j z\big)\big|\,dz \lesssim \sum_l \big(2^{-\sigma l}\big)^{1-\frac{\sigma}{3}} \big(2^{2l}\delta)^{\frac{\sigma}{3}} \lesssim_\sigma \delta^{\frac{\sigma}{3}}.
\end{align*}
\end{proof}

%redundent: Let $$ \begin{aligned} L(\tau) &= J(\tau) \cdot \nabla J(\tau) \cdot (D\Phi(\tau))^{-1} \cdot \big(\tilde y_1, \tilde y_2, \tilde y_3 + \xi_j(\tilde y_1 z_2 -\tilde y_2 z_1))\big)^T\\ &=\nabla J(\tau) \cdot D\Phi^*(\tau) \cdot \big(\tilde y_1, \tilde y_2, \tilde y_3 + \xi_j(\tilde y_1 z_2 -\tilde y_2 z_1))\big)^T, \end{aligned} $$ where $z_1=-t+s-u, z_2=-C_2(t)+C_2(s)-C_2(u)$, $\nabla J(\tau)$ is a $1\times 3$ matrix, $(D\phi(\tau))^{-1}$ is a $3\times 3$ matrix, and $\big(\tilde y_1, \tilde y_2, \tilde y_3 + \xi_j(\tilde y_1 z_2 -\tilde y_2 z_1))\big)^T$ is a $3\times 1$ matrix.

Roughly speaking, the following two propositions give an estimation of the Jacobian, and control the regularity of the Jacobian and the partition of unity functions $\psi_l$ on each $E_l$. We postpone their proofs to Subsections \ref{jacobian} and \ref{unity}, respectively. 
\begin{proposition}\label{good}
There exists $0<\sigma <1$ such that for the function $F(\tau)$ defined in (\ref{bigf}),
$$
\int_{R_1} F(\tau)^{-\sigma}\,d\tau \lesssim 1, \quad \forall j\in \mathbb{N},
$$
where the implicit constant is independent of $j$.
\end{proposition}

\begin{proposition}\label{estimateofderivative}
For every $j, l\in \mathbb{N}$, for the vector field $U$ defined in (\ref{vectorfield}),
$$
\int_{E_l} |UQ(\tau)J(\tau)|\,d\tau \lesssim 2^{2l} \delta, \qquad \int |U\psi_l(\tau)|\,d\tau \lesssim 2^{2l} \delta, \quad \forall 0\leq |\tilde y|\lesssim \delta \leq 1,
$$
where the implicit constants are independent of $j, l$.
\end{proposition}

\begin{remark}\label{after}
As mentioned in Remark \ref{before}, we no longer have that $\{J=0\}$ is a union of hyperplanes, and the estimate (\ref{indispensable}) is out of reach in cases that are not translation-invariant. To resolve this issue, we construct the quantity $F(\tau)$ and decompose the complement of $\{J=0\}$ into $E_l$'s based on the value of $F(\tau)$, in contrast to decomposing the complement of $\{J=0\}$ based on the distance to $\{J=0\}$ in the $\mathbb{R}^n$ case. We then finish the proof by obtaining the decay of the measure of the $E_l$, derived from Proposition \ref{good}, and the control of the regularity of the Jacobian and the partition of unity functions on each $E_l$, as in Proposition \ref{estimateofderivative}.

The quantity $F(\tau)$ is selected with deliberation so that the two regularity inequalities in Proposition \ref{estimateofderivative} hold simultaneously. Without confronting the issue of losing hyperplane properties or without carefully choosing $F(\tau)$ would lead to unnecessary conditions to be added in Theorem \ref{1}. We would then largely lose the advantage gained from the techniques of Carbery, Christ, Vance, Wainger, and Watson \cite{CARBERYCHRIST, VANCE}.

Our method works the same for the $\mathbb{R}^n$ case. Compared to the $D\Phi$ in (\ref{determinant}), the Jacobian matrix in the $\mathbb{R}^n$ case has a neat form:
$$
D\Phi(t_1, \ldots, t_n)=
\begin{pmatrix}
-1 & 1 & \cdots & (-1)^n\\
-C_2'(t_1) & C_2'(t_2) & \cdots & (-1)^n C_2'(t_n)\\
\cdots & \cdots & \cdots & \cdots\\
-C_n'(t_1) & C_n'(t_2) & \cdots & (-1)^nC_n'(t_n)
\end{pmatrix}.
$$
For the $\mathbb{R}^n$ case, we simply take 
$$
F(\tau) =\frac{1}{\sqrt{\frac{1+\sum_{i,j=1}^n w_{i,j}^2}{J^2}+ \sum_i \frac{1}{\text{dist}\,(\tau, H_i)^2}}}, \quad \tau=(t_1, \ldots, t_n),
$$
where $J=\det D\Phi$, $(w_{i,j})_{i,j}$ is the adjoint matrix $D\Phi^*$, and the $H_i$ are the finitely many hyperplanes mentioned above. The rest of the proof follows similarly, except for significantly fewer calculations in the $\mathbb{R}^n$ case. 

Our method is likely to be generalized to other stratified nilpotent groups. One needs to find the correct $F(\tau)$ and the corresponding estimates to aim at.  

%The strategy in CNSW lead to even more general $C^\infty$ cases.
\end{remark}

\begin{remark}
The interval $[2,8]$ that we choose for the positive part of the support of $\phi_j$ when we define $\mathcal{H}_j$ can be replaced by any $[a,b]$ with $a>1$ and $\frac{b}{a}>2$. We need $a>1$ for Lemma \ref{lowerbound} to hold, and then we can apply Lemma \ref{mj} to obtain (\ref{measureu}) and (\ref{measuret}), so that the case $j<k$ in the proof of Theorem \ref{main} follows. Because we decompose $\mathcal{H}_\Gamma$ into $\mathcal{H}_j$ dyadically, we require $\frac{b}{a}>2$ so that the kernel $\frac{1}{t}$, as well as other Calder\'on-Zygmund kernels, can be decomposed into $\sum_j 2^j\phi_j(2^j \cdot)$ with every $\phi_j$ supported on $[a,b]$.
\end{remark}

The rest of this subsection is devoted to the proofs of Lemmas \ref{monotone}, \ref{standardmonotone}, and \ref{lowerbound}.
\begin{proof}[Proof of Lemma \ref{monotone}]
We have
$$
C_2''(t)=\frac{2^{-2j}\alpha''(2^{-j}t)}{h_2(2^{-j})}>0, \quad t>0,
$$
and
$$
C_2'(1)=\frac{2^{-j}\alpha'(2^{-j})-\alpha(2^{-j})}{h_2(2^{-j})}=1.
$$
Thus $C_2'(t)>0$ for $t\geq 1$. And since $C_2(1)=0$, we have $C_2(t)>0$ for $t>1$, and for $t\in [2,8]$,
$$
C_2(t)= \int_1^t C_2'(s)\,ds \leq 7C_2'(t).
$$
For $t\geq 1$,
\begin{align*}
(C_3^-)''(t)&=\frac{2^{-2j}\bar \beta''(2^{-j}t) -  2^{-2j}\alpha''(2^{-j}t)\bar h_\beta(2^{-j})/h_2(2^{-j})}{2^{-j}h_2(2^{-j}) + \bar h_3(2^{-j})}\\
&= \frac{2^{-2j}\alpha''(2^{-j}t)}{2^{-j}h_2(2^{-j})+\bar h_3(2^{-j})} \Big(\frac{\bar \beta''(2^{-j}t)}{\alpha''(2^{-j}t)} - \frac{\bar h_\beta(2^{-j})}{h_2(2^{-j})}\Big)\\
&= \frac{2^{-2j}\alpha''(2^{-j}t)}{2^{-j}h_2(2^{-j})+\bar h_3(2^{-j})} \Big(\frac{\bar \beta''(2^{-j}t)}{\alpha''(2^{-j}t)} - \frac{\bar h_\beta(2^{-j}t)}{h_2(2^{-j}t)} +\frac{\bar h_\beta(2^{-j}t)}{h_2(2^{-j}t)} - \frac{\bar h_\beta(2^{-j})}{h_2(2^{-j})}\Big)\\
&= \frac{2^{-2j}\alpha''(2^{-j}t)}{2^{-j}h_2(2^{-j})+\bar h_3(2^{-j})} \Big(\frac{\bar h_3(2^{-j}t)}{h_2(2^{-j}t)}+\frac{\bar h_\beta(2^{-j}t)}{h_2(2^{-j}t)} - \frac{\bar h_\beta(2^{-j})}{h_2(2^{-j})}\Big)\\
&>0,
\end{align*}
where the last inequality is due to (\ref{ratioincrease}). And since $C_3^-(1)= (C_3^-)'(1)=0$, we have $C_3^-(t), (C_3^-)'(t)>0$ for $t>1$.
\end{proof}

\begin{proof}[Proof of Lemma \ref{standardmonotone}]
Recall $\xi_j$ from Definition \ref{jmatrix}. For $t\in [2,8]$,
\begin{align*}
\Big(\frac{(C_3^-)''(t)}{C_2''(t)}\Big)'&=2\xi_j\frac{\bar D_3(2^{-j}t)}{(\alpha''(2^{-j}t))^2}>0,\\
\Big(\frac{(C_3^+)''(t)}{C_2''(t)}\Big)'&=2\xi_j\frac{\bar D_3(2^{-j}t)+(\alpha''(2^{-j}t))^2}{(\alpha''(2^{-j}t))^2}>0,\\
\Big(\frac{\hat C_3''(t)}{C_2''(t)}\Big)'&=2\xi_j\frac{\bar D_3(2^{-j}t)+2(\alpha''(2^{-j}t))^2-\alpha'''(2^{-j}t)\big(\alpha'(2^{-j}t)-2^j\alpha(2^{-j})\big)}{(\alpha''(2^{-j}t))^2}\\
&\geq 2\xi_j \frac{\min\{\bar D_3(2^{-j}t)+2(\alpha''(2^{-j}t))^2, \hat D_3(2^{-j}t)\}}{(\alpha''(2^{-j}t))^2}>0,\\
\Big(\frac{\doublehat C_3''(t)}{C_2''(t)}\Big)'&=2\xi_j \frac{\alpha''(2^{-j}t)\big(\bar \beta'''(2^{-j}t)-\alpha''(2^{-j}t)\big) - \alpha'''(2^{-j}t)\big(\bar \beta''(2^{-j}t)-\alpha'(2^{-j}t) + 2^j\alpha(2^{-j})\big)}{(\alpha''(2^{-j}t))^2}\\
&\geq 2\xi_j \frac{\min \{\doublehat D_3(2^{-j}t), \triplehat D_3(2^{-j}t)\}}{(\alpha''(2^{-j}t))^2}>0.
\end{align*}
\end{proof}

\begin{proof}[Proof of Lemma \ref{lowerbound}]
We have
$$
(B_j^-(t))^{-1}=
\begin{pmatrix}
1 & 0 & 0\\
& &\\
-\frac{C_2'(t)}{tC_2''(t)} & \frac{1}{tC_2''(t)} & 0\\
& &\\
\frac{C_2'(t)(C_3^-)''(t)-C_2''(t)(C_3^-)'(t)}{C_2''(t)\big(\frac{(C_3^-)''(t)}{C_2''(t)}\big)' \frac{h_2(2^{-j}t)}{h_2(2^{-j})}} & -\frac{(C_3^-)''(t)}{C_2''(t)\big(\frac{(C_3^-)''(t)}{C_2''(t)}\big)' \frac{h_2(2^{-j}t)}{h_2(2^{-j})}} & \frac{1}{\big(\frac{(C_3^-)''(t)}{C_2''(t)}\big)' \frac{h_2(2^{-j}t)}{h_2(2^{-j})}}
\end{pmatrix}.
$$
By (\ref{infinitesimal}), for $t\in [2,8]$,
\begin{align*}
&0\leq \frac{1}{tC_2''(t)} = \frac{h_2(2^{-j})}{2^{-j}h_2'(2^{-j}t)} \leq 8\frac{h_2(2^{-j}t)}{(2^{-j}t)h_2'(2^{-j}t)} \lesssim 1,\\
&0\leq \frac{C_2'(t)}{tC_2''(t)} = \frac{h_2(2^{-j}t)}{(2^{-j}t)h_2'(2^{-j}t)} + \frac{\int_{2^{-j}}^{2^{-j}t}\frac{h_2(s)}{s^2}\,ds}{h_2'(2^{-j}t)} \lesssim \frac{h_2(2^{-j}t)}{(2^{-j}t)h_2'(2^{-j}t)} \lesssim 1,\\
&0\leq \frac{1}{\big(\frac{(C_3^-)''(t)}{C_2''(t)}\big)' \frac{h_2(2^{-j}t)}{h_2(2^{-j})}}=\frac{\bar h_3(2^{-j})+ 2^{-j}h_2(2^{-j})}{2^{-j}\bar h_3'(2^{-j}t)} \leq \frac{\bar h_3(2^{-j}t)+\int_0^{2^{-j}t} h_2}{2^{-j}\bar h_3'(2^{-j}t)} \lesssim 1,\\
&0\leq \frac{(C_3^-)''(t)}{C_2''(t)\big(\frac{(C_3^-)''(t)}{C_2''(t)}\big)' \frac{h_2(2^{-j}t)}{h_2(2^{-j})}} = \frac{\frac{\bar h_3(2^{-j}t)}{h_2(2^{-j}t)} + \int_{2^{-j}}^{2^{-j}t} \frac{h_2'(s)}{h_2(s)^2} \bar h_3(s)\,ds}{2^{-j}\frac{\bar h_3'(2^{-j}t)}{h_2(2^{-j})}} \lesssim \frac{\bar h_3(2^{-j}t)}{(2^{-j}t) \bar h_3'(2^{-j}t)}\lesssim 1.
\end{align*}
Since for $t\in [2,8]$,
\begin{align*}
\int_{2^{-j}}^{2^{-j}t} \Big( \frac{\bar \beta''(2^{-j}t)}{\alpha''(2^{-j}t)} h_2(s) - \bar h_\beta (s)\Big) \frac{1}{s^2}\,ds =  \int_{2^{-j}}^{2^{-j}t} \Big( \frac{\bar h_3(2^{-j}t)}{h_2(2^{-j}t)} + \int_s^{2^{-j}t} \frac{h_2'(r)\bar h_3(r)}{h_2(r)^2}\,dr\Big) \frac{h_2(s)}{s^2}\,ds\geq 0,
\end{align*}
we have
\begin{align*}
0&\leq \frac{C_2'(t)(C_3^-)''(t)-C_2''(t)(C_3^-)'(t)}{C_2''(t)\big(\frac{(C_3^-)''(t)}{C_2''(t)}\big)' \frac{h_2(2^{-j}t)}{h_2(2^{-j})}} %&= \frac{\bar \beta''(2^{-j}t)(\alpha'(2^{-j}t)-2^j\alpha(2^{-j}))-\alpha''(2^{-j}t)(\bar \beta'(2^{-j}t)-2^j\bar \beta(2^{-j}))}{\alpha''(2^{-j}t)\bar h_3'(2^{-j}t)}\\
=\frac{\bar h_3(2^{-j}t)}{(2^{-j}t)\bar h_3'(2^{-j}t)} + \frac{\int_{2^{-j}}^{2^{-j}t} \Big( \frac{\bar \beta''(2^{-j}t)}{\alpha''(2^{-j}t)} h_2(s) - \bar h_\beta (s)\Big) \frac{1}{s^2}\,ds}{\bar h_3'(2^{-j}t)}\\
&= \frac{\bar h_3(2^{-j}t)}{(2^{-j}t)\bar h_3'(2^{-j}t)} +\frac{\bar h_3(2^{-j}t)}{2^{-j}\bar h_3'(2^{-j}t)} -\frac{\int_{2^{-j}}^{2^{-j}t}\int_s^{2^{-j}t} \int_r^{2^{-j}t} \frac{\bar h_3'(u)}{h_2(u)}\,du h_2'(r)\,dr \frac{1}{s^2}\,ds}{\bar h_3'(2^{-j}t)} \lesssim 1.
\end{align*}
Hence $\|(B_j^-(t))^{-1}\|\lesssim 1$. For $B_j^+(t)$, by denoting 
$$
\bar h_3^+(t) := \frac{(\bar \beta^+)''(t)}{\alpha''(t)} h_2(t)- \bar h_\beta^+(t) = \bar h_3(t) +\int_0^t h_2 \geq \bar h_3(t), \quad \text{ for } t>0,
$$
we still have $(\bar h_3^+)'(t)\gtrsim \frac{\bar h_3^+(t)}{t}$ for $t>0$, and thus the above argument works similarly for $B_j^+(t)$.
\end{proof}

\subsection{An estimation for the Jacobian}\label{jacobian}
The goal of this subsection is to prove Proposition \ref{good}. Recall that $W^u$ and $W^t$ are defined in (\ref{wu}) and (\ref{wt}), respectively. We first use Proposition 3.1 of \cite{VANCE} (see Lemma \ref{mj}) or its ideas to prove the following three lemmas:
\begin{lemma}\label{lowdim}
For every $0<\sigma<1$, we have
$$
\iint_{[2,8]^2} |W_3^u|^{-\sigma}\,dt\,ds \lesssim 1, \quad \forall j\in \mathbb{N},
$$
where the implicit constant is independent of $j$.
\end{lemma}

\begin{proof}
We have
\begin{align*}
\iint_{[2,8]^2} |W_3^u|^{-\sigma}\,dt\,ds &= \iint_{s< t} \big(C_2'(t)-C_2'(s)\big)^{-\sigma}\,dt\,ds + \iint_{t< s} \big(C_2'(s)-C_2'(t)\big)^{-\sigma}\,dt\,ds.
\end{align*}
For $2\leq s< t\leq 8$, by (\ref{infinitesimal}),
\begin{align*}
C_2'(t)-C_2'(s) =\int_s^t C_2''(r)\,dr =\int_s^t \frac{2^{-j}h'(2^{-j}r)}{rh_2(2^{-j})}\,dr \gtrsim \int_s^t \frac{h(2^{-j}r)}{r^2h_2(2^{-j})}\,dr \geq \frac{1}{s}-\frac{1}{t} \geq \frac{t-s}{64},
\end{align*}
and we write $C_2'(t)-C_2'(s) \geq K_0(t-s)$ for some constant $K_0>0$ independent of $j$. Thus for $\lambda>0$,
\begin{align*}
|\{s\in [2,t): \big(C_2'(t)-C_2'(s)\big)^{-\sigma}>\lambda\}| \leq |\{s\in [2,t): K_0^{-\sigma} (t-s)^{-\sigma} >\lambda\}| \leq \frac{1}{K_0} \lambda^{-\frac{1}{\sigma}}.
\end{align*}
Hence
\begin{align*}
\iint_{s\leq t} \big(C_2'(t)-C_2'(s)\big)^{-\sigma}\,ds\,dt &= \int \int_0^\infty|\{s\in [2,t):  \big(C_2'(t)-C_2'(s)\big)^{-\sigma}>\lambda\}| \,d\lambda \,dt\\
&\lesssim \int \int_0^\infty\min\{6, \lambda^{-\frac{1}{\sigma}}\}\,d\lambda \,dt \lesssim 1.
\end{align*}
Similarly, $\iint_{t< s} \big(C_2'(s)-C_2'(t)\big)^{-\sigma}\,dt\,ds\lesssim 1$.
\end{proof}

\begin{lemma}\label{ss}
For every $0<\sigma <\frac{1}{2}$, we have
$$
\iint_{[2,8]^2} \Big(\frac{2\xi_jC_2(s)C_2'(s)}{|W_1^u|}\Big)^\sigma\,dt\,ds \lesssim 1, \quad \forall j\in \mathbb{N},
$$
where the implicit constant is independent of $j$.
\end{lemma}

\begin{proof}
By letting 
\begin{align*}
f(t)&:=\frac{W_1^u}{2\xi_jC_2(s)C_2'(s)} = \frac{C_2'(t)\hat C_3'(s)-C_2'(s)\big(C_3'(t)+\xi_j(tC_2'(t)-C_2(t))\big) -2\xi_j C_2(s)C_2'(s)}{2\xi_jC_2(s)C_2'(s)} \\
&= \frac{C_2'(t)\hat C_3'(s)-C_2'(s)\big((C_3^+)'(t)+\frac{2^{-j}h_2(2^{-j})-\int_0^{2^{-j}}h_2}{2^{-j}h_2(2^{-j})+\bar h_3(2^{-j})}C_2'(t) +\frac{\int_0^{2^{-j}}h_2}{2^{-j}h_2(2^{-j})+\bar h_3(2^{-j})}\big) -2\xi_j C_2(s)C_2'(s)}{2\xi_jC_2(s)C_2'(s)},
\end{align*}
we can apply Lemma \ref{mj}. In fact, by (\ref{C3}) and Lemma \ref{standardmonotone},
$$
\Delta_jf(t) 
%= \frac{-C_2'(s)}{2\xi_jC_2(s)C_2'(s)}\Big(\frac{\big(C_3'(t)+\xi_j(tC_2'(t)-C_2(t))\big)'}{C_2''(t)}\Big)' \frac{h_2(2^{-j}t)}{h_2(2^{-j})} 
= \frac{-C_2'(s)}{2\xi_jC_2(s)C_2'(s)}\Big(\frac{(C_3^+)''(t)}{C_2''(t)}\Big)' \frac{h_2(2^{-j}t)}{h_2(2^{-j})}
$$
does not change sign for $t\in [2,8]$, and by Lemma \ref{lowerbound},
\begin{align*}
1&\lesssim \Big|\Big(\frac{-2\xi_jC_2(s)C_2'(s)-\frac{\int_0^{2^{-j}}h_2}{2^{-j}h_2(2^{-j})+\bar h_3(2^{-j})}C_2'(s)}{2\xi_jC_2(s)C_2'(s)}, \frac{\hat C_3'(s) -\frac{2^{-j}h_2(2^{-j})-\int_0^{2^{-j}}h_2}{2^{-j}h_2(2^{-j})+\bar h_3(2^{-j})} C_2'(s)}{2\xi_jC_2(s)C_2'(s)}, \frac{-C_2'(s)}{2\xi_jC_2(s)C_2'(s)}\Big) \Big| \\
&=\big|\big(f(t), tf'(t), \Delta_jf(t)\big)(B_j^+(t))^{-1}\big| \lesssim \big|\big(f(t), tf'(t), \Delta_jf(t)\big)\big|.
\end{align*}
Thus by Lemma \ref{mj}, there exists $\lambda_0>0$ independent of $j, s,u$, such that for every $0< \lambda \leq \lambda_0$,
\begin{equation}
|\{t\in [2,8]: \frac{|W_1^u|}{2\xi_jC_2(s)C_2'(s)}<\lambda\}|\lesssim \lambda^{\frac{1}{2}},
\end{equation}
where the implicit constant is independent of $j, u, s$. Hence for $\sigma<\frac{1}{2}$,
\begin{align*}
\iint_{[2,8]^2} \Big(\frac{2\xi_jC_2(s)C_2'(s)}{|W_1^u|}\Big)^\sigma\,dt\,ds & = \int \int_0^\infty |\{t\in [2,8]: \Big(\frac{2\xi_jC_2(s)C_2'(s)}{|W_1^u|}\Big)^\sigma>\lambda \}| \,d\lambda \,ds\\
&\lesssim \int \Big(\int_0^{\lambda_0^{-\sigma}} 6\,d\lambda + \int_{\lambda_0^{-\sigma}}^\infty \lambda^{-\frac{1}{2\sigma}}\,d\lambda\Big)\,ds \lesssim_\sigma 1.
\end{align*}
\end{proof}

\begin{lemma}\label{ut}
For every $0<\sigma<1$, we have
$$
\iiint_{[2,8]^3} \Big(\frac{2\xi_jC_2'(u)C_2'(t)}{|J-W_1^u-W_1^t-2\xi_jC_2(s)C_2'(s)|}\Big)^\sigma \,dt\,ds\,du \lesssim 1, \quad \forall j\in \mathbb{N},
$$
where the implicit constant is independent of $j$.
\end{lemma}

\begin{proof}
We have
\begin{align*}
&\quad \iiint_{[2,8]^3} \Big(\frac{2\xi_jC_2'(u)C_2'(t)}{|J-W_1^u-W_1^t-2\xi_jC_2(s)C_2'(s)|}\Big)^\sigma \,dt\,ds\,du \\
&= \iint \Big(\int_2^8 \frac{1}{\Big|s-\frac{C_2'(u)\big(C_3'(t)+\xi_j(tC_2'(t)-C_2(t))\big) -C_2'(t)(C_3^-)'(u)}{2\xi_jC_2'(t)C_2'(u)}\Big|^\sigma} \,ds\Big)\,dt\,du\\
&\leq \iint \int_{-3}^3 \frac{1}{|s|^\sigma}\,ds \,dt\,du \lesssim_\sigma 1.
\end{align*}
\end{proof}

\begin{proof}[Proof of Proposition \ref{good}]
Take $0< \sigma<1$ to be determined. We have
\begin{align*}
&\quad \int_{R_1} F(\tau)^{-\sigma}\,d\tau \\
& =\int_{R_1} \Big( \frac{1+|W^t|^2+|W^u|^2 +(2\xi_jC_2'(t)C_2'(u))^2 +(2\xi_jC_2(s)C_2'(s))^2}{J^2}+\sum_i \frac{1}{\text{dist}\,(\tau, L_i)^2} \Big)^{\frac{\sigma}{2}}\,d\tau\\
&\lesssim_\sigma \int_{R_1} \big(\frac{1}{|J|}\big)^\sigma + \big( \frac{|W^t|}{|J|}\big)^\sigma + \big( \frac{|W^u|}{|J|}\big)^\sigma + \big(\frac{2\xi_jC_2'(t)C_2'(u)}{|J|}\big)^\sigma +\big(\frac{2\xi_jC_2(s)C_2'(s)}{|J|}\big)^\sigma + \sum_i \frac{1}{\text{dist}\,(\tau, L_i)^\sigma} \,d\tau.
\end{align*}
By (\ref{distance}), for each $i$,
$$
\int_{R_1} \frac{1}{\text{dist}\,(\tau, L_i)^\sigma}\,d\tau \lesssim_\sigma 1, \quad \forall 0<\sigma<1.
$$

By (\ref{measureu}) and (\ref{measuret}),
$$
|\{u\in [2,8]: \big(\frac{|W^u|}{|J|}\big)^\sigma > \lambda\}| \lesssim \lambda^{-\frac{1}{2\sigma}}, \quad \forall \lambda \geq \lambda_0^{-\sigma}, \quad \forall a.e. \,(t,s)\in [2,8]^2,
$$
where the implicit constant is independent of $j, t,s$, and
$$
|\{t\in [2,8]: \big(\frac{|W^u|}{|J|}\big)^\sigma >\lambda \}| \lesssim \lambda^{-\frac{1}{2\sigma}}, \quad \forall \lambda \geq \lambda_0^{-\sigma}, \quad \forall a.e. \,(s,u)\in [2,8]^2,
$$
where the implicit constant is independent of $j, s,u$. Thus for $\sigma<\frac{1}{2}$,
\begin{align*}
\int_{[2,8]^3} \big(\frac{|W^u|}{|J|}\big)^\sigma\,d\tau & = \iint \int_0^\infty |\{u\in [2,8]: \big(\frac{|W^u|}{|J|}\big)^\sigma>\lambda \}| \,d\lambda \,dt\,ds\\
&\lesssim \iint \Big(\int_0^{\lambda_0^{-\sigma}} 6\,d\lambda + \int_{\lambda_0^{-\sigma}}^\infty \lambda^{-\frac{1}{2\sigma}}\,d\lambda\Big)\,dt\,ds \lesssim_\sigma 1.
\end{align*}
Similarly, for $\sigma<\frac{1}{2}$, $\int_{[2,8]^3} \big(\frac{|W^t|}{|J|}\big)^\sigma\,d\tau\lesssim_\sigma 1$.

By Lemma \ref{lowdim}, for $\sigma<\frac{1}{4}$,
\begin{align*}
\int_{[2,8]^3} \big(\frac{1}{|J|}\big)^\sigma \,d\tau \leq \Big(\int \Big(\frac{1}{|W_3^u|}\Big)^{2\sigma}\,d\tau\Big)^{\frac{1}{2}} \Big(\int \Big(\frac{|W_3^u|}{|J|}\Big)^{2\sigma}\,d\tau\Big)^{\frac{1}{2}} \lesssim 1.
\end{align*}

By Lemma \ref{ss}, for $\sigma <\frac{1}{4}$,
\begin{align*}
\int_{[2,8]^3} \Big(\frac{2\xi_j C_2(s)C_2'(s)}{|J|}\Big)^\sigma \,d\tau \leq \Big(\int \Big(\frac{2\xi_j C_2(s)C_2'(s)}{|W_1^u|}\Big)^{2\sigma} \,d\tau\Big)^{\frac{1}{2}} \Big(\int \Big(\frac{|W_1^u|}{|J|}\Big)^{2\sigma} \,d\tau\Big)^{\frac{1}{2}} \lesssim 1.
\end{align*}
Thus by Lemma \ref{ut}, for $\sigma<\frac{1}{8}$,
\begin{align*}
&\quad \int_{[2,8]^3} \Big(\frac{2\xi_jC_2'(t)C_2'(u)}{|J|}\Big)^\sigma\,d\tau \\
&\leq \Big(\int \Big(\frac{2\xi_jC_2'(t)C_2'(u)}{|J-W_1^u -W_1^t-2\xi_jC_2(s)C_2'(s)|}\Big)^{2\sigma}\,d\tau\Big)^{\frac{1}{2}} \Big(\int \Big(\frac{|J-W_1^u -W_1^t-2\xi_jC_2(s)C_2'(s)|}{|J|}\Big)^{2\sigma}\,d\tau\Big)^{\frac{1}{2}} \lesssim 1.
\end{align*}
We then obtain the desired inequality by taking $\sigma<\frac{1}{8}$.
\end{proof}

\subsection{Regularity of the Jacobian and the partition of unity functions}\label{unity}
The goal of this subsection is to prove Proposition \ref{estimateofderivative}. The regularity inequality of the Jacobian in the proposition is inspired by Lemma 5.9 of Carbery, Wainger, and Wright \cite{CWW}.

We first need several lemmas, whose proof will be postponed to the end of this subsection.

\begin{lemma}\label{sign}
For each fixed $(s,u)\in [2,8]^2$, $J$ changes monotonicity at most once and changes sign at most twice on $[2,8]$ as a function of $t$. For each fixed $(t,s)\in [2,8]^2$, $J$ changes monotonicity at most once and changes sign at most twice on $[2,8]$ as a function of $u$. For each fixed $(s,u)\in [2,8]^2$, $\tilde J:=J-W_1^t + \xi_jC_2'(u)C_2'(t)$ changes monotonicity at most once and changes sign at most twice on $[2,8]$ as a function of $t$.
\end{lemma}

\begin{lemma}\label{sign2}
For each fixed $(s,u)\in [2,8]^2$, $\frac{J-W_1^t}{C_2'(t)}$ changes monotonicity at most once on $[2,8]$ as a function of $t$.
\end{lemma}

\begin{definition}\label{decomposej}
Define 
\begin{align*}
J_1&:=C_2'(s)(C_3^-)'(u)-C_2'(u)\doublehat C_3'(s),\\
J_2&:=C_2'(t)\hat C_3'(s) -C_2'(s)\big(C_3'(t)+\xi_j(tC_2'(t)-C_2(t))\big),\\
J_3&:=C_2'(u)\big(C_3'(t)+\xi_j(tC_2'(t)-C_2(t))\big)-C_2'(t)(C_3^-)'(u) -2\xi_jsC_2'(t)C_2'(u),\\
J_4&:=-2\xi_jC_2(s)C_2'(s).
\end{align*}
\end{definition}

\begin{lemma}\label{monotonic}
We have
$$
J= J_1+J_2 + J_3  +J_4.
$$
Denote
\begin{align*}
\tilde J_1&:= J_1+\xi_jC_2'(t)C_2'(s),\\
\tilde J_2&:= J_2+\xi_jC_2'(t)C_2'(s).
\end{align*}
For each fixed $(t,u)$, $J_1, \ldots, J_4$ and $\tilde J_1, \tilde J_2$ change monotonicity at most once and changes sign at most twice on $[2,8]$ as functions of $s$. Moreover, for each fixed $(t,u)$, $\frac{J_1}{C_2'(s)}$ and $\frac{J_2}{C_2'(s)}$ change monotonicity at most once and change sign at most twice on $[2,8]$ as functions of $s$.

For each $l\in \mathbb{N}$, on $E_l$,
$$
|J_i|\lesssim 2^l|J|, \quad \forall 1\leq i \leq 4.
$$
\end{lemma}

\begin{lemma}\label{tmonotone}
Denote
$$
\widehat W_i^u = W_i^u + \xi_jC_2'(t)C_2'(u), \quad 1\leq i\leq 3.
$$
For every fixed $(s,u)\in [2,8]^2$, $W_1^u, W_2^u, W_3^u, \widehat W_1^u, \widehat W_2^u, \widehat W_3^u, \frac{W_1^u}{C_2'(t)}, \frac{W_2^u}{C_2'(t)}, \frac{W_3^u}{C_2'(t)}$ change monotonicity at most once and change sign at most twice on $[2,8]$ as functions of $t$. We have $|\widehat W_i^u|\lesssim 2^l|J|$ on $E_l$ for $1\leq i \leq 3$.
\end{lemma}

\begin{lemma}\label{smonotonic}
Denote
\begin{align*}
&\doublehat W_i^u :=W_i^u +2\xi_jC_2(s)C_2'(s) + \xi_jC_2'(s)C_2'(t),\\
&\doublehat W_i^t :=W_i^t +2\xi_jC_2(s)C_2'(s) + \xi_jC_2'(s)C_2'(t),\\
&\big(2\xi_jC_2(s)C_2'(s)\doublehat{\big)}:=2\xi_jC_2(s)C_2'(s)+\xi_jC_2'(s)C_2'(t).
\end{align*}
For each fixed $(t,u)\in [2,8]^2$, as functions of $s$, $W_1^u+2\xi_jC_2(s)C_2'(s)$, $\frac{W_1^u+2\xi_jC_2(s)C_2'(s)}{C_2'(s)}$, $W_1^t+2\xi_jC_2(s)C_2'(s)$, $\frac{W_1^t+2\xi_jC_2(s)C_2'(s)}{C_2'(s)}$, $W_2^u+2\xi_jsC_2'(t)$, $\frac{W_2^u+2\xi_jsC_2'(t)}{C_2'(s)}$,$W_2^t+2\xi_jsC_2'(u)$, $W_3^u$, $\frac{W_3^u}{C_2'(s)}$, $W_3^t$, $\frac{W_3^t}{C_2'(s)}$, $2\xi_jC_2(s)C_2'(s)$, $\frac{2\xi_jC_2(s)C_2'(s)}{C_2'(s)}$, $\doublehat W_1^u, \doublehat W_1^t, \doublehat W_2^u, \doublehat W_2^t, \doublehat W_3^u, \doublehat W_3^t$, $(2\xi_jC_2(s)C_2'(s)\doublehat{)}$ change monotonicity at most once and change sign at most twice on $[2,8]$.
\end{lemma}

\begin{proof}[Proof of Proposition \ref{estimateofderivative}]
We need to show that for the vector field $U$ defined in (\ref{vectorfield}),
$$
\int_{E_l} |UQ(\tau)J(\tau)|\,d\tau \lesssim 2^{2l} \delta, \qquad \int |U\psi_l(\tau)|\,d\tau \lesssim 2^{2l} \delta, \quad \forall 0\leq |\tilde y|\lesssim \delta \leq 1, \quad \forall j,l\in \mathbb{N}.
$$

Recall from (\ref{zw}) that for $z=\Phi(\tau)$, $z_1=-t+s-u, z_2=-C_2(t)+C_2(s)-C_2(u)$. Recall $Q(\tau) = |J(\tau)|^{-1} \phi_j(t)\phi_j(s)\phi_j(u)$ is $C^2$ on $E_l \subseteq R_1\backslash\{J=0\}$. Denote the adjoint matrix $D\Phi(\tau)^*$ by
\begin{align*}
\begin{pmatrix}
w_{11} & -w_{21} & w_{31}\\
-w_{12} & w_{22} & -w_{32}\\
w_{13} & -w_{23} &w_{33}
\end{pmatrix}.
\end{align*}

\textbf{Proof of } $\bm{\int_{E_l} |UQ(\tau)J(\tau)|\,d\tau \lesssim 2^{2l} \delta}$:

We have
\begin{equation}\label{gij}
\begin{aligned}
UQ(\tau) &= \nabla Q(\tau) \cdot \frac{D\Phi(\tau)^*}{J(\tau)} \cdot \big(\tilde y_1, \tilde y_2, \tilde y_3 + \xi_j(\tilde y_1 z_2 -\tilde y_2 z_1)\big)^T\\
&=\frac{\tilde y_1\big(\partial_tQ w_{11}-\partial_sQ w_{12} +\partial_uQ w_{13}\big) +\tilde y_2\big(-\partial_tQ w_{21} +\partial_sQ w_{22} -\partial_uQ w_{23}\big)}{J}\\
&\quad +\frac{\big(\tilde y_3 +\xi_j(\tilde y_1 z_2-\tilde y_2z_1)\big) \big(\partial_tQ w_{31} -\partial_sQ w_{32} + \partial_uQ w_{33}\big)
}{J}\\
&=\tilde y_1\frac{\partial_tQ (w_{11}+\xi_j z_2w_{31})-\partial_sQ(w_{12}+\xi_j z_2w_{32})+\partial_uQ(w_{13}+\xi_jz_2w_{33})}{J}\\
&\quad -\tilde y_2 \frac{\partial_tQ (w_{21}+\xi_jz_1w_{31})-\partial_sQ(w_{22}+\xi_jz_1w_{32}) +\partial_uQ(w_{23}+\xi_jz_1w_{33})}{J}\\
&\quad +\tilde y_3 \frac{\partial_tQ w_{31} -\partial_sQ w_{32} +\partial_uQ w_{33}}{J}\\
&=: \tilde y_1 \frac{\partial_tQ g_{11} - \partial_sQ g_{12} + \partial_u Q g_{13}}{J} - \tilde y_2 \frac{\partial_tQ g_{21} - \partial_sQ g_{22} + \partial_uQ g_{23}}{J} \\
&\quad + \tilde y_3 \frac{\partial_tQ g_{31}-\partial_sQ g_{32} + \partial_uQg_{33}}{J},
\end{aligned}
\end{equation}
and
\begin{align*}
\int_{E_l} |UQ(\tau)J(\tau)|\,d\tau \lesssim \delta \int_{E_l}\big(|\partial_tQ| \sum_{i=1}^3|g_{i1}| + |\partial_sQ|\sum_{i=1}^3 |g_{i2}| + |\partial_uQ|\sum_{i=1}^3|g_{i3}|\big)\,d\tau.
\end{align*}

\textbf{Estimation of }$\bm{\int_{E_l} |g_{11}\partial_tQ|\,d\tau}$:

Recall the matrix $D\Phi(t,s,u)$ from (\ref{determinant}). We have
\begin{align*}
g_{11}&=w_{11}+\xi_jz_2w_{31}\\
&=C_2'(u)C_3'(s)-C_2'(s)C_3'(u)+\xi_j\big(uC_2'(u)C_2'(s)-sC_2'(s)C_2'(u)\\
&\quad -2C_2(t)C_2'(s)+2C_2(s)C_2'(s)+2C_2(t)C_2'(u)-C_2(u)C_2'(s)-C_2(s)C_2'(u)\big)\\
&= -W_1^t + 2\xi_jC_2(t)\big(C_2'(u)-C_2'(s)\big).
\end{align*}
Recall the region $E_l$ from (\ref{el}). On $E_l$, 
\begin{equation}\label{smallquantities}
\frac{1}{|J|},\quad \frac{|W^t|}{|J|},\quad \frac{|W^u|}{|J|},\quad \frac{2\xi_jC_2'(t)C_2'(u)}{|J|},\quad \frac{2\xi_jC_2(s)C_2'(s)}{|J|}\leq 2^l.
\end{equation}
By (\ref{regions}), $t\leq u$ on $E_l\subseteq R_1$. Then by the monotonicity in Lemma \ref{monotone},
\begin{align*}
|g_{11}| &\leq |W_1^t| +2\xi_j C_2(t)C_2'(u)+ \max\{2\xi_jC_2(t)C_2'(u), 2\xi_jC_2(s)C_2'(s)\}\leq 3\cdot 2^l|J|, \quad \text{on } E_l.
\end{align*}
Since
$$
|\partial Q| \lesssim \frac{|\partial J|}{J^2} + \frac{1}{|J|}, \quad \text{on } R_1\backslash\{J=0\},
$$
we have
\begin{align*}
\int_{E_l} |g_{11}\partial_tQ|\,d\tau &\lesssim \int_{E_l} \Big(\frac{|g_{11}\partial_tJ|}{J^2} + \Big|\frac{g_{11}}{J}\Big|\Big)\,d\tau \\
&\lesssim \int_{E_l} \Big(\frac{\big(|W_1^t|+2\xi_jC_2(s)C_2'(s)\big)|\partial_tJ|}{J^2} +\frac{2\xi_jC_2(t)C_2'(u)|\partial_tJ|}{J^2} + \Big|\frac{g_{11}}{J}\Big|\Big)\,d\tau.
\end{align*}
We have
$$
\int_{E_l} \Big|\frac{g_{11}}{J}\Big|\,d\tau \lesssim 2^l.
$$

By Lemma \ref{sign}, for each fixed $(s,u)\in [2,8]^2$, $J$ changes monotonicity at most once on $[2,8]$ as a function of $t$. And since $|W_1^t|+2\xi_jC_2(s)C_2'(s)$ is independent of $t$, for each fixed $(s,u)$, $\{t\in [2,8]: |W_1^t|+2\xi_jC_2(s)C_2'(s)\leq 2\cdot 2^l|J|\}$ is the union of at most three intervals, on each of which $J$ does not change sign. Thus we can write
$$
\{t\in [2,8]: |W_1^t|+2\xi_jC_2(s)C_2'(s)\leq 2\cdot 2^l|J|\}=[b_1, b_2]\cup [b_3, b_4] \cup [b_5, b_6] \cup [b_7, b_8],
$$
where $J$ is monotone and does not change sign on each $[b_{2i-1}, b_{2i}]$, and $2\leq b_1\leq \cdots \leq b_8 \leq 8$ depends on $(s,u)$.
Therefore we have
\begin{align*}
&\quad \int_{E_l}\frac{\big(|W_1^t|+2\xi_jC_2(s)C_2'(s)\big)|\partial_tJ|}{J^2}\,d\tau \\
& \leq \int_{\substack{\tau\in [2,8]^3\\|W_1^t|+2\xi_jC_2(s)C_2'(s)\leq 2\cdot 2^l|J|}} \frac{\big(|W_1^t|+2\xi_jC_2(s)C_2'(s)\big)|\partial_tJ|}{J^2}\,d\tau\\
&=\iint_{[2,8]^2} \Big(\int_{\substack{t\in [2,8]\\ |W_1^t|+2\xi_jC_2(s)C_2'(s)\leq 2\cdot 2^l|J|}} \frac{|\partial_tJ|}{J^2}\,dt\Big)\big(|W_1^t|+2\xi_jC_2(s)C_2'(s)\big)\,ds\,du\\
&=\iint_{[2,8]^2} \sum_{i=1}^4\Big| \int_{b_{2i-1}}^{b_{2i}} \frac{\partial_tJ}{J^2}\,dt \Big|\cdot \big(|W_1^t|+2\xi_jC_2(s)C_2'(s)\big)\,ds\,du\\
&\leq \iint_{[2,8]^2} \sum_{i=1}^8 \frac{|W_1^t|+2\xi_jC_2(s)C_2'(s)}{|J|}\Big|_{t=b_i}\,ds\,du\leq \iint_{[2,8]^2} \sum_{i=1}^8 2\cdot 2^l\,ds\,du \lesssim 2^l.
\end{align*}

Since $W_1^t$ is independent of $t$, by Lemma \ref{sign}, for each fixed $(s,u)\in [2,8]^2$, $J - W_1^t$ changes monotonicity at most once on $[2,8]$ as a function of $t$. By Lemma \ref{sign2}, for each fixed $(s,u)\in [2,8]^2$, $\frac{J-W_1^t}{C_2'(t)}$ changes monotonicity at most once on $[2,8]$ as a function of $t$. Thus $\{t\in [2,8]:\frac{1}{2}\xi_j C_2'(u)C_2'(t) \leq |J- W_1^t|\}$ is the union of at most three intervals, on each of which $J-W_1^t$ does not change sign. And $\{t\in [2,8]:\frac{1}{2}\xi_j C_2'(u)C_2'(t) \geq |J- W_1^t|\}$ is the union of at most two intervals. Thus we can write
\begin{align*}
\{t\in [2,8]:\frac{1}{2}\xi_j C_2'(u)C_2'(t) \leq |J- W_1^t|\}=[b_1, b_2]\cup [b_3, b_4] \cup [b_5, b_6] \cup [b_7, b_8],
\end{align*}
where $J-W_1^t$ is monotone and does not change sign on each $[b_{2i-1}, b_{2i}]$, and $2\leq b_1\leq \cdots \leq b_8 \leq 8$ depends on $(s,u)$. We have
\begin{align*}
\int_{E_l} \frac{2\xi_j C_2(t)C_2'(u) |\partial_tJ|}{J^2}\,d\tau &= \int_{E_l} \frac{2\xi_j C_2(t)C_2'(u) |\partial_t(J-W_1^t)|}{J^2}\,d\tau\\
&= \int_{\substack{E_l\\ \frac{1}{2}\xi_jC_2'(u)C_2'(t)\leq |J-W_1^t|}} \frac{2\xi_j C_2(t)C_2'(u) |\partial_t(J-W_1^t)|}{J^2}\,d\tau \\
&\quad + \int_{\substack{E_l\\ \frac{1}{2}\xi_jC_2'(u)C_2'(t)\geq |J-W_1^t|}} \frac{2\xi_j C_2(t)C_2'(u) |\partial_t(J-W_1^t)|}{J^2}\,d\tau.
\end{align*}
Recall $C_2(t)\lesssim C_2'(t)$. Thus
\begin{align*}
&\quad \int_{\substack{E_l\\ \frac{1}{2}\xi_jC_2'(u)C_2'(t)\leq |J-W_1^t|}} \frac{2\xi_j C_2(t)C_2'(u) |\partial_t(J-W_1^t)|}{J^2}\,d\tau \\
&\lesssim 2^{2l} \int_{\substack{\tau\in [2,8]^3\\\frac{1}{2}\xi_jC_2'(u)C_2'(t)\leq |J-W_1^t|}} \frac{2\xi_j C_2'(u)C_2(t) |\partial_t(J-W_1^t)|}{(J-W_1^t)^2} \,d\tau \\
&=2^{2l}\iint_{[2,8]^2} \sum_{i=1}^4 \Big|\int_{b_{2i-1}}^{b_{2i}} 2\xi_j C_2'(u)C_2(t) \partial_t\big(\frac{1}{J-W_1^t}\big)\,dt\Big|\,ds\,du\\
&= 2^{2l} \iint_{[2,8]^2} \sum_{i=1}^4\Big|\int_{b_{2i-1}}^{b_{2i}} \Big(\frac{2\xi_j C_2'(u)C_2'(t)}{J-W_1^t}- \partial_t\big(\frac{2\xi_jC_2(t)C_2'(u)}{J-W_1^t}\big)\Big)\,dt\Big|\,ds\,du\\
&\lesssim 2^{2l}.
\end{align*}
We have
\begin{align*}
&\quad \int_{E_l} \frac{2\xi_jC_2(t)C_2'(u)\partial_t\big(\xi_jC_2'(u)C_2'(t)\big)}{J^2}\,d\tau\\
&\leq 2^{2l} \int_{[2,8]^3} \frac{2\xi_jC_2(t)C_2'(u)\partial_t\big(\xi_jC_2'(u)C_2'(t)\big)}{\big(\xi_jC_2'(u)C_2'(t)\big)^2}\,d\tau\\
&=2^{2l}\int_{[2,8]^3} \Big(\frac{2\xi_jC_2'(u)C_2'(t)}{\xi_jC_2'(u)C_2'(t)} -\partial_t\big(\frac{2\xi_jC_2(t)C_2'(u)}{\xi_jC_2'(u)C_2'(t)}\big)\Big)\,d\tau\\
&\lesssim 2^{2l}.
\end{align*}
By Lemma \ref{sign}, for each fixed $(s,u)\in [2,8]^2$,
$$
\tilde J:=J-W_1^t + \xi_jC_2'(u)C_2'(t)
$$
changes monotonicity at most once and changes sign at most twice on $[2,8]$ as a function of $t$. Since $\{t\in [2,8]:\frac{1}{2}\xi_j C_2'(u)C_2'(t) \geq |J- W_1^t|\}$ is the union of at most two intervals, we can write
\begin{align*}
\{t\in [2,8]:\frac{1}{2}\xi_j C_2'(u)C_2'(t) \geq |J- W_1^t|\} = [a_1, a_2]\cup [a_3, a_4]\cup \cdots \cup [a_9, a_{10}],
\end{align*}
where $\tilde J$ is monotone and does not change sign on each $[a_{2i-1}, a_{2i}]$, and $2\leq a_1\leq \cdots \leq a_{10} \leq 8$ depends on $(s,u)$. Since $\frac{1}{2}\xi_jC_2'(u)C_2'(t)\geq |J-W_1^t|$ implies $\tilde J \geq \frac{1}{2}\xi_jC_2'(u)C_2'(t)$, we have
\begin{align*}
&\quad \int_{\substack{E_l\\ \frac{1}{2}\xi_jC_2'(u)C_2'(t)\geq |J-W_1^t|}} \frac{2\xi_jC_2(t)C_2'(u)|\partial_t\tilde J|}{J^2}\,d\tau\\
&\lesssim 2^{2l} \int_{\substack{\tau\in [2,8]^3\\\frac{1}{2}\xi_jC_2'(u)C_2'(t)\geq |J-W_1^t|}} \frac{2\xi_jC_2(t)C_2'(u)|\partial_t\tilde J|}{\tilde J^2}\,d\tau\\
&=2^{2l} \iint_{[2,8]^2} \sum_{i=1}^5\Big|\int_{a_{2i-1}}^{a_{2i}} 2\xi_jC_2(t)C_2'(u)\partial_t\big(\frac{1}{\tilde J}\big)\,dt\Big|\,ds\,du\\
&=2^{2l} \iint_{[2,8]^2} \sum_{i=1}^5\Big|\int_{a_{2i-1}}^{a_{2i}} \frac{2\xi_jC_2'(u)C_2'(t)}{\tilde J} - \partial_t\big(\frac{2\xi_jC_2(t)C_2'(u)}{\tilde J}\big)\,dt\Big|\,ds\,du\\
&\lesssim 2^{2l}.
\end{align*}
Thus 
\begin{align*}
&\quad \int_{\substack{E_l\\ \frac{1}{2}\xi_jC_2'(u)C_2'(t)\geq |J-W_1^t|}} \frac{2\xi_jC_2(t)C_2'(u)|\partial_t(J-W_1^t)|}{J^2}\,d\tau\\
&\leq \int_{\substack{E_l\\ \frac{1}{2}\xi_jC_2'(u)C_2'(t)\geq |J-W_1^t|}} \frac{2\xi_jC_2(t)C_2'(u)|\partial_t\tilde J|}{J^2}\,d\tau + \int_{\substack{E_l\\ \frac{1}{2}\xi_jC_2'(u)C_2'(t)\geq |J-W_1^t|}} \frac{2\xi_jC_2(t)C_2'(u)|\partial_t\big(\xi_jC_2'(u)C_2'(t)\big)|}{J^2}\,d\tau\\
&\lesssim 2^{2l}.
\end{align*}
Hence
\begin{align*}
&\quad \int_{E_l} \frac{2\xi_j C_2(t)C_2'(u) |\partial_t(J-W_1^t)|}{J^2}\,d\tau\\
&\leq \int_{\substack{E_l\\ \frac{1}{2}\xi_jC_2'(u)C_2'(t)\leq |J-W_1^t|}} \frac{2\xi_j C_2(t)C_2'(u) |\partial_t(J-W_1^t)|}{J^2}\,d\tau + \int_{\substack{E_l\\ \frac{1}{2}\xi_jC_2'(u)C_2'(t)\geq |J-W_1^t|}} \frac{2\xi_j C_2(t)C_2'(u) |\partial_t(J-W_1^t)|}{J^2}\,d\tau\\
&\lesssim 2^{2l}.
\end{align*}
Hence 
\begin{equation}\label{g11bound}
\int_{E_l} \frac{|g_{11}\partial_tJ|}{J^2}\,d\tau \lesssim 2^{2l}.
\end{equation}

\textbf{Estimation of } $\bm{\int_{E_l} |g_{12}\partial_sQ|\,d\tau}$:

We have
\begin{align*}
g_{12}&=w_{12} + \xi_j z_2 w_{32} = C_2'(t)C_3'(u)-C_2'(u)C_3'(t)+\xi_j\big(2sC_2'(t)C_2'(u)-tC_2'(t)C_2'(u)\\
&\quad -uC_2'(u)C_2'(t)-C_2(t)C_2'(u)+C_2(u)C_2'(t)-2C_2(s)C_2'(t)+2C_2(t)C_2'(t)\big)\\
&=W_1^t +W_1^u -J + 2\xi_j \big(C_2(t)C_2'(t)-C_2(t)C_2'(u)-C_2(s)C_2'(t) +C_2(s)C_2'(s)\big).
\end{align*}
On $E_l\subseteq R_1$, $t\leq u$. Thus either 
$$
C_2(t)C_2'(t)\leq C_2(s)C_2'(s), \quad C_2(s) C_2'(t) \leq C_2(s) C_2'(s),
$$
or 
$$
C_2(t)C_2'(t) \leq C_2(t)C_2'(u)\leq C_2'(u)C_2'(t), \quad C_2(s)C_2'(t) \leq C_2(u)C_2'(t) \leq C_2'(t)C_2'(u).
$$
Recall (\ref{smallquantities}). Therefore $|g_{12}|\leq 7\cdot 2^l|J|$ on $E_l$. Denote 
$$
\tilde g_{12} :=W_1^t +W_1^u-J+2\xi_j\big(C_2(s)C_2'(s)-sC_2'(u)C_2'(t)+C_2(t)C_2'(t)-C_2(t)C_2'(u)\big).
$$
It is straightforward to check that $\tilde g_{12}$ is independent of $s$. Since $s\leq 8$, we have $|\tilde g_{12}|\leq 14\cdot 2^l|J|$. Thus
\begin{align*}
\int_{E_l} |g_{12}\partial_sQ |\,d\tau &\lesssim \int_{E_l} \Big(\frac{|g_{12}\partial_sJ|}{J^2} + \Big|\frac{g_{12}}{J}\Big|\Big)\,d\tau\\
&\lesssim \int_{E_l} \Big(\frac{\big(|\tilde g_{12}|+2\xi_jC_2'(t)C_2'(u) \big)|\partial_sJ|}{J^2} +\frac{2\xi_jC_2'(t)C_2(s)|\partial_sJ|}{J^2} + \Big|\frac{g_{12}}{J}\Big|\Big)\,d\tau.
\end{align*}
We have
$$
\int_{E_l} \Big|\frac{g_{12}}{J}\Big|\,d\tau \lesssim \int_{E_l} 2^l\,d\tau \lesssim 2^l.
$$

By Lemma \ref{monotonic}, we have 
\begin{align*}
&\quad \int_{E_l}\frac{\big(|\tilde g_{12}| + 2\xi_jC_2'(t)C_2'(u)\big)|\partial_sJ|}{J^2}\,d\tau \leq \sum_{i=1}^4 \int_{E_l}\frac{\big(|\tilde g_{12}| + 2\xi_jC_2'(t)C_2'(u)\big)|\partial_sJ_i|}{J^2}\,d\tau.
\end{align*}
By Lemma \ref{monotonic}, for each fixed $(t,u)\in [2,8]^2$, $J_1$ changes monotonicity at most once on $[2,8]$ as a function of $s$. Thus $\{s\in [2,8]:|J_1|\leq |\tilde g_{12}| + 2\xi_jC_2'(t)C_2'(u)\}$ is the union of at most two intervals. And $\{s\in [2,8]:|J_1|\geq |\tilde g_{12}| + 2\xi_jC_2'(t)C_2'(u)\}$ is the union of at most three intervals, on each of which $J_1$ does not change sign. Hence we can write
\begin{align*}
&\{s\in [2,8]:|J_1|\leq |\tilde g_{12}| + 2\xi_jC_2'(t)C_2'(u) \} = [a_1, a_2]\cup [a_3, a_4]\cup \cdots \cup [a_9, a_{10}],\\
&\{s\in [2,8]:|J_1|\geq |\tilde g_{12}| + 2\xi_jC_2'(t)C_2'(u) \} = [b_1, b_2]\cup [b_3, b_4]\cup[b_5, b_6]\cup [b_7, b_8],
\end{align*}
where $J_1$ is monotone and does not change sign on each $[a_{2i-1}, a_i]$ and each $[b_{2i-1}, b_{2i}]$, and $2\leq a_1\leq \cdots \leq a_{10}\leq 8$ and $2\leq b_1\leq \cdots \leq b_8\leq 8$ depends on $(t,u)$. By Lemma \ref{monotonic},
\begin{align*}
&\quad \int_{E_l}\frac{\big(|\tilde g_{12}| + 2\xi_jC_2'(t)C_2'(u)\big)|\partial_sJ_1|}{J^2}\,d\tau\\
&= \int_{\substack{E_l\\|J_1|\leq |\tilde g_{12}| + 2\xi_jC_2'(t)C_2'(u)}}\frac{\big(|\tilde g_{12}| + 2\xi_jC_2'(t)C_2'(u)\big)|\partial_sJ_1|}{J^2}\,d\tau \\
&\quad + \int_{\substack{E_l\\|J_1|\geq |\tilde g_{12}| + 2\xi_jC_2'(t)C_2'(u)}}\frac{\big(|\tilde g_{12}| + 2\xi_jC_2'(t)C_2'(u)\big)|\partial_sJ_1|}{J^2}\,d\tau\\
&\lesssim 2^{2l}\Big(\iint_{[2,8]^2}\sum_{i=1}^5 \Big|\int_{a_{2i-1}}^{a_{2i}} \frac{\partial_sJ_1}{|\tilde g_{12}| + 2\xi_jC_2'(t)C_2'(u)}\,ds\Big|\,dt\,du\\
&\qquad + \iint_{[2,8]^2}\sum_{i=1}^4 \Big|\int_{b_{2i-1}}^{b_{2i}} \frac{(|\tilde g_{12}| + 2\xi_jC_2'(t)C_2'(u)\big)\partial_sJ_1}{J_1^2}\,ds\Big|\,dt\,du\Big) \\
&\lesssim 2^{2l}.
\end{align*}
The estimation of the other four terms is similar, and thus
$$
\int_{E_l}\frac{\big(|\tilde g_{12}| + 2\xi_jC_2'(t)C_2'(u)\big)|\partial_sJ|}{J^2}\,d\tau\lesssim 2^{2l}.
$$
In particular,
$$
\int_{E_l}\frac{ 2\xi_jC_2'(t)C_2'(u)|\partial_sJ|}{J^2}\,d\tau\lesssim 2^{2l}.
$$

We have
$$
\int_{E_l} \frac{2\xi_jC_2'(t)C_2(s)|\partial_sJ|}{J^2}\,d\tau = \int_{\substack{E_l\\ C_2(s)\leq C_2'(u)}} \frac{2\xi_jC_2'(t)C_2(s)|\partial_sJ|}{J^2}\,d\tau + \int_{\substack{E_l\\ C_2'(u)\leq C_2(s)}} \frac{2\xi_jC_2'(t)C_2(s)|\partial_sJ|}{J^2}\,d\tau,
$$
and
$$
\int_{\substack{E_l\\ C_2(s)\leq C_2'(u)}} \frac{2\xi_jC_2'(t)C_2(s)|\partial_sJ|}{J^2}\,d\tau \leq \int_{E_l} \frac{2\xi_jC_2'(u)C_2'(t)|\partial_sJ|}{J^2}\,d\tau \lesssim 2^{2l}.
$$
For the other term,
\begin{align*}
&\quad \int_{\substack{E_l\\ C_2'(u)\leq C_2(s)}} \frac{2\xi_jC_2'(t)C_2(s)|\partial_sJ|}{J^2} \leq \sum_{i=1}^4\int_{\substack{E_l\\ C_2'(u)\leq C_2(s)}} \frac{2\xi_jC_2'(t)C_2(s)|\partial_sJ_i|}{J^2}.
\end{align*}
Recall $t\leq u$ on $E_l\subseteq R_1$. Thus
\begin{align*}
&\quad \int_{\substack{E_l\\ C_2'(u)\leq C_2(s)}} \frac{2\xi_jC_2'(t)C_2(s)|\partial_sJ_3|}{J^2}= \int_{\substack{E_l\\ C_2'(u)\leq C_2(s)}} \frac{2\xi_jC_2'(t)C_2(s)\cdot 2\xi_jC_2'(t)C_2'(u)}{J^2} \\
&\leq \int_{E_l} \frac{2\xi_j\max\{C_2'(t)C_2'(u), C_2(s)C_2'(s)\}\cdot 2\xi_jC_2'(t)C_2'(u)}{J^2} \lesssim 2^{2l}.
\end{align*}
Since $\{s\in [2,8]: C_2'(u)\leq C_2(s)\}\subseteq \{s\in [2,8]: C_2'(t)\leq C_2(s)\}$, which is an interval, we have
\begin{align*}
\int_{\substack{E_l\\ C_2'(u)\leq C_2(s)}} \frac{2\xi_jC_2'(t)C_2(s)|\partial_sJ_4|}{J^2}&= \int_{\substack{E_l\\ C_2'(u)\leq C_2(s)}} \frac{2\xi_jC_2'(t)C_2(s)\partial_s\big(2\xi_jC_2(s)C_2'(s)\big)}{J^2}\\
&\leq 2^{2l}\int_{C_2'(t)\leq C_2(s)} \frac{2\xi_jC_2'(t)C_2(s)\partial_s\big(2\xi_jC_2(s)C_2'(s)\big)}{\big(2\xi_jC_2(s)C_2'(s)\big)^2}\\
&=2^{2l}\int_{C_2'(t)\leq C_2(s)} \Big(\frac{2\xi_jC_2'(t)C_2'(s)}{2\xi_jC_2(s)C_2'(s)}-\partial_s\big(\frac{2\xi_jC_2'(t)C_2(s)}{2\xi_jC_2(s)C_2'(s)}\big)\Big) \\
&\lesssim 2^{2l}.
\end{align*}
By Lemma \ref{monotonic}, for each fixed $(t,u)$, $J_1$ and $\frac{J_1}{C_2'(s)}$ change monotonicity at most once and change sign at most twice on $[2,8]$ as functions of $s$. Thus the set $\{s\in [2,8]: |J_1|\geq \frac{1}{2}\xi_jC_2'(t)C_2'(s)\}$ is the union of at most three intervals, on each of which $J_1$ does not change sign. And $\{s\in [2,8]: |J_1|\leq \frac{1}{2}\xi_jC_2'(t)C_2'(s)\}$ is the union of at most two intervals. We can write
\begin{align*}
\{s\in [2,8]: |J_1| \geq \frac{1}{2}\xi_j C_2'(t)C_2'(s)\} = [b_1, b_2] \cup [b_3, b_4]\cup [b_5, b_6] \cup [b_7, b_8],
\end{align*}
where $J_1$ is monotone and does not change sign on each interval $[b_{2i-1}, b_{2i}]$, and $2\leq b_1\leq \cdots \leq b_8\leq 8$ depends on $(t,u)$. Since $C_2(s)\lesssim C_2'(s)$, we have
\begin{align*}
&\quad \int_{\substack{E_l\\C_2'(u)\leq C_2(s)\\|J_1|\geq \frac{1}{2}\xi_jC_2'(t)C_2'(s)}} \frac{2\xi_jC_2'(t)C_2(s)|\partial_sJ_1|}{J^2}\,d\tau\\
&\lesssim 2^{2l} \int_{|J_1|\geq \frac{1}{2}\xi_jC_2'(t)C_2'(s)} \frac{2\xi_jC_2'(t)C_2(s)|\partial_sJ_1|}{J_1^2}\,d\tau\\
&=2^{2l} \iint_{[2,8]^2} \sum_{i=1}^4 \Big|\int_{b_{2i-1}}^{b_{2i}} \frac{2\xi_jC_2'(t)C_2(s)\partial_sJ_1}{J_1^2}\,ds\Big|\,dt\,du\\
&=2^{2l} \iint_{[2,8]^2} \sum_{i=1}^4 \Big|\int_{b_{2i-1}}^{b_{2i}}\Big(\frac{2\xi_jC_2'(t)C_2'(s)}{J_1}-\partial_s\big(\frac{2\xi_jC_2'(t)C_2(s)}{J_1}\big)\Big)\,ds\Big|\,dt\,du\\
&\lesssim 2^{2l}.
\end{align*}
Denote 
$$
\tilde J_1=J_1 + \xi_jC_2'(t)C_2'(s).
$$
By Lemma \ref{monotonic}, $|J_1|\lesssim 2^l|J|$. Since $t\leq u$ on $E_l\subseteq R_1$, when $C_2'(u)\leq C_2(s)$, 
\begin{equation}\label{ts}
2\xi_jC_2'(t)C_2'(s)\leq 2\xi_jC_2(s)C_2'(s)\leq 2^l|J|,
\end{equation}
and thus $|\tilde J|\lesssim 2^l|J|$. By Lemma \ref{monotonic}, for each fixed $(t,u)$, $\tilde J_1$ changes monotonicity at most once as a function of $s$. Since $\{s\in [2,8]: |J_1|\leq \frac{1}{2}\xi_jC_2'(t)C_2'(s)\}$ is the union of at most two intervals, we can write
$$
\{s\in [2,8]:|J_1|\leq \frac{1}{2}\xi_jC_2'(t)C_2'(s)\} = [a_1, a_2]\cup \cdots [a_9, a_{10}],
$$
where $\tilde J_1$ is monotone and does not change sign on each $[a_{2i-1}, a_{2i}]$, and $2\leq a_1\leq \cdots \leq a_{10}\leq 8$ depends on $(t,u)$. Note $|J_1|\leq \frac{1}{2}\xi_jC_2'(t)C_2'(s)$ implies $\tilde J_1 \geq \frac{1}{2}\xi_jC_2'(t)C_2'(s)$. We have
\begin{align*}
&\quad \int_{\substack{E_l\\C_2'(u)\leq C_2(s)\\|J_1|\leq \frac{1}{2}\xi_jC_2'(t)C_2'(s)}} \frac{2\xi_jC_2'(t)C_2(s)|\partial_s\tilde J_1|}{J^2}\,d\tau\\
&\lesssim 2^{2l} \int_{|J_1|\leq \frac{1}{2}\xi_jC_2'(t)C_2'(s)} \frac{2\xi_jC_2'(t)C_2(s)|\partial_s\tilde J_1|}{\tilde J_1^2}\,d\tau\\
&=2^{2l} \iint_{[2,8]^2} \sum_{i=1}^5 \Big|\int_{a_{2i-1}}^{a_{21}} \frac{2\xi_jC_2'(t)C_2(s)\partial_s\tilde J_1}{\tilde J_1^2}\,ds\Big|\,dt\,du\\
&=2^{2l} \iint_{[2,8]^2} \sum_{i=1}^5 \Big|\int_{a_{2i-1}}^{a_{21}} \Big(\frac{2\xi_jC_2'(t)C_2'(s)}{\tilde J_1}-\partial_s\big(\frac{2\xi_jC_2'(t)C_2(s)}{\tilde J_1}\big)\Big)\,ds\Big|\,dt\,du\\
&\lesssim 2^{2l}.
\end{align*}
By (\ref{ts}), and since $C_2(s)\lesssim C_2'(s)$,
\begin{align*}
&\quad \int_{\substack{E_l\\ C_2'(u)\leq C_2(s)}} \frac{2\xi_jC_2'(t)C_2(s)\partial_s\big(2\xi_jC_2'(s)C_2'(t)\big)}{J^2}\,d\tau\\
&\leq 2^{2l} \int_{[2,8]^3} \frac{2\xi_jC_2'(t)C_2(s)\partial_s\big(2\xi_jC_2'(s)C_2'(t)\big)}{\big(2\xi_jC_2'(s)C_2'(t)\big)^2}\,d\tau\\
&= 2^{2l} \int_{[2,8]^3} \Big(\frac{2\xi_jC_2'(t)C_2'(s)}{2\xi_jC_2'(s)C_2'(t)} - \partial_s\big(\frac{2\xi_jC_2'(t)C_2(s)}{2\xi_jC_2'(s)C_2'(t)}\big)\Big)\,d\tau \\
&\lesssim 2^{2l}.
\end{align*}
Therefore
\begin{align*}
&\quad \int_{\substack{E_l\\C_2'(u)\leq C_2(s)\\|J_1|\leq \frac{1}{2}\xi_jC_2'(t)C_2'(s)}} \frac{2\xi_jC_2'(t)C_2(s)|\partial_sJ_1|}{J^2}\,d\tau\\
&\leq \int_{\substack{E_l\\C_2'(u)\leq C_2(s)\\|J_1|\leq \frac{1}{2}\xi_jC_2'(t)C_2'(s)}} \frac{2\xi_jC_2'(t)C_2(s)|\partial_s\tilde J_1|}{J^2}\,d\tau + \int_{\substack{E_l\\C_2'(u)\leq C_2(s)\\|J_1|\leq \frac{1}{2}\xi_jC_2'(t)C_2'(s)}} \frac{2\xi_jC_2'(t)C_2(s)\partial_s\big(\xi_jC_2'(t)C_2'(s)\big)}{J^2}\,d\tau\\
&\lesssim 2^{2l}.
\end{align*}
Hence
$$
\int_{\substack{E_l\\C_2'(u)\leq C_2(s)}} \frac{2\xi_jC_2'(t)C_2(s)|\partial_sJ_1|}{J^2}\,d\tau \lesssim 2^{2l}.
$$
Similarly,
$$
\int_{\substack{E_l\\C_2'(u)\leq C_2(s)}} \frac{2\xi_jC_2'(t)C_2(s)|\partial_sJ_2|}{J^2}\,d\tau \lesssim 2^{2l}.
$$
Hence 
\begin{equation}\label{g12bound}
\int_{E_l} \frac{|g_{12}\partial_sJ|}{J^2}\,d\tau \lesssim 2^{2l}.
\end{equation}

\textbf{Estimation of } $\bm{\int_{E_l} |g_{13}\partial_uQ|\,d\tau}$:

We have 
\begin{align*}
g_{13}&=w_{13} + \xi_j z_2w_{33}\\
&=C_2'(s)C_3'(t)-C_2'(t)C_3'(s)+\xi_j\big(C_2(t)C_2'(s)+C_2(s)C_2'(t) -sC_2'(s)C_2'(t) + tC_2'(s)C_2'(t)-2C_2(t)C_2'(t)\big)\\
&=-W_1^u + 2\xi_j\big(C_2(s)C_2'(t)+C_2(t)C_2'(s)-C_2(t)C_2'(t)-C_2(s)C_2'(s)\big).
\end{align*}
Thus $g_{13}$ is independent of $u$. On $E_l$, since $t\leq u$,
$$
|g_{13}|\leq |W_1^u| + 2\xi_j\big(2\max\{C_2'(t)C_2'(u), C_2(s)C_2'(s)\} +C_2'(u)C_2'(t)+C_2(s)C_2'(s)\big)\leq 5\cdot 2^l|J|.
$$
We have 
\begin{align*}
\int_{E_l} |g_{13}\partial_uQ|\,d\tau &\lesssim \int_{E_l} \Big(\frac{|g_{13}\partial_uJ|}{J^2} + \Big|\frac{g_{13}}{J}\Big|\Big)\,d\tau \lesssim \int_{E_l} \frac{|g_{13}\partial_uJ|}{J^2}\,d\tau + 2^l.
\end{align*}
By Lemma \ref{sign}, for each fixed $(t,s)\in [2,8]^2$, $J$ changes monotonicity at most once on $[2,8]$ as a function of $u$. Thus $\{u\in [2,8]: |g_{13}|\leq 5\cdot 2^l|J|\}$ is the union of at most three intervals, on each of which $J$ does not change sign. We can write
$$
\{u\in [2,8]: |g_{13}|\leq 5\cdot 2^l|J|\} =[b_1, b_2]\cup [b_3, b_4]\cup [b_5, b_6]\cup [b_7, b_8],
$$
where $J$ is monotone and does not change sign on each $[b_{2i-1}, b_{2i}]$, and $2\leq b_1 \leq \cdots \leq b_8\leq 8$ depends on $(t,s)$. Therefore we have
\begin{align*}
\int_{E_l} \frac{|g_{13}\partial_uJ|}{J^2}\,d\tau & \leq \int_{\substack{\tau\in [2,8]^3\\ |g_{13}|\leq 5\cdot 2^l|J|}} |g_{13}|\big|\partial_u\big(\frac{1}{J}\big)\big|\,d\tau\\
&=\iint_{[2,8]^2} \sum_{i=1}^4 \Big|\int_{b_{2i-1}}^{b_{2i}} \partial_u\big(\frac{|g_{13}|}{J}\big)\,du\Big|\,dt\,ds \lesssim 2^l.
\end{align*}
Hence 
\begin{equation}\label{g13bound}
\int_{E_l} \frac{|g_{13}\partial_uJ|}{J^2}\,d\tau \lesssim 2^{2l}.
\end{equation}

The estimations of 
$$
\int_{E_l} |g_{21}\partial_tQ|\,d\tau, \int_{E_l} |g_{22}\partial_sQ|\,d\tau, \int_{E_l} |g_{23}\partial_uQ|\,d\tau, \int_{E_l} |g_{31}\partial_tQ|\,d\tau, \int_{E_l} |g_{32}\partial_sQ|\,d\tau, \int_{E_l} |g_{33}\partial_uQ|\,d\tau
$$  
follow in a similar way and are simpler. We therefore omit their calculations and finish the proof of $\int_{E_l}|UQ(\tau)J(\tau)|\,d\tau\lesssim 2^{2l}\delta$.

\textbf{Proof of } $\bm{\int |U\psi_l(\tau)|\,d\tau \lesssim 2^{2l} \delta}$:

Recall the $g_{ij}$ from (\ref{gij}). We have
\begin{align*}
\int |U\psi_l(\tau)|\,d\tau &=\int_{E_l} |U\psi_l(\tau)|\,d\tau =\int_{E_l} \Big| \nabla \psi_l(\tau)\cdot \frac{D\Phi(\tau)^*}{J(\tau)} \cdot \big(\tilde y_1, \tilde y_2, \tilde y_3 + \xi_j(\tilde y_1z_2-\tilde y_2z_1)\big)^T \Big|\,d\tau\\
&=\int_{E_l} \Big| \tilde y_1 \frac{g_{11}\partial_t\psi_l -g_{12}\partial_s\psi_l +g_{13}\partial_u\psi_l}{J} -\tilde y_2 \frac{g_{21}\partial_t\psi_l-g_{22}\partial_s\psi_l + g_{23}\partial_u\psi_l}{J}\\
&\qquad \qquad \qquad \qquad \qquad \qquad \qquad \qquad \quad +\tilde y_3 \frac{g_{31}\partial_t\psi_l -g_{32}\partial_s\psi_l + g_{33}\partial_u\psi_l}{J}\Big|\,d\tau\\
&\lesssim \delta \int_{E_l} \frac{|\partial_t\psi_l|\sum_{i=1}^3 |g_{i1}| + |\partial_s\psi_l| \sum_{i=1}^3|g_{i2}| + |\partial_u\psi_l|\sum_{i=1}^3 |g_{i3}|}{|J|}\,d\tau.
\end{align*}
On $E_l\subseteq R_1$, $\psi_l(\tau)= \rho_l(F(\tau)) \cdot 1_{R_1}(\tau) =\rho_l(F(\tau))$. Thus on $E_l$,
\begin{align*}
|\partial_t\psi_l(\tau)|&= \big|\partial_t \big(\rho_l(F(\tau))\big)\big| = |\rho_l'(F(\tau))|\cdot |\partial_tF(\tau)| \lesssim 2^l \cdot |\partial_tF(\tau)|\\
&= 2^l \cdot |F(\tau)|^3\\
&\quad \cdot \Big|\frac{1}{J} \partial_t\frac{1}{J} + \sum_{i=1}^3\frac{W_i^t}{J}\partial_t \frac{W_i^t}{J} + \sum_{i=1}^3\frac{W_i^u}{J} \partial_t \frac{W_i^u}{J} + \frac{2\xi_jC_2'(t)C_2'(u)}{J} \partial_t \frac{2\xi_jC_2'(t)C_2'(u)}{J} \\
& \qquad + \frac{2\xi_jC_2(s)C_2'(s)}{J} \partial_t \frac{2\xi_jC_2(s)C_2'(s)}{J} + \sum_{i=1}^5 \frac{1}{\text{dist}\,(\tau, L_i)} \partial_t \frac{1}{\text{dist}\,(\tau, L_i)}\Big|\\
&\leq 2^l \cdot |F(\tau)|^3 \cdot |F(\tau)|^{-1} \cdot \Big(\Big| \partial_t \frac{1}{J}\Big| + \sum_{i=1}^3\Big|\partial_t \frac{W_i^t}{J}\Big| + \sum_{i=1}^3\Big|\partial_t \frac{W_i^u}{J}\Big|  \\
&\qquad + \Big|\partial_t \frac{2\xi_jC_2'(t)C_2'(u)}{J}\Big| + \Big|\partial_t \frac{2\xi_j C_2(s)C_2'(s)}{J}\Big| + \sum_{i=1}^5\Big|\partial_t \frac{1}{\text{dist}\,(\tau, L_i)}\Big|\Big)\\
&\leq  2^{4-l} \cdot \Big(\Big| \partial_t \frac{1}{J}\Big| + \sum_{i=1}^3\Big|\partial_t \frac{W_i^t}{J}\Big| + \sum_{i=1}^3\Big|\partial_t \frac{W_i^u}{J}\Big| + \Big|\partial_t \frac{2\xi_jC_2'(t)C_2'(u)}{J}\Big| + \Big|\partial_t \frac{2\xi_j C_2(s)C_2'(s)}{J} \Big|\Big) \\
&\quad+ 2^{4-l}\cdot \sum_{i=1}^5\Big| \frac{\partial_t \text{dist}\,(\tau, L_i)}{\text{dist}\,(\tau, L_i)^2}\Big|\\
&\lesssim 2^{-l}\cdot \frac{|\partial_tJ|}{|J|} \cdot\Big(\frac{1}{|J|}+ \sum_{i=1}^3 \frac{|W_i^t|}{|J|} + \sum_{i=1}^3\frac{|W_i^u|}{|J|} + \frac{2\xi_jC_2'(t)C_2'(u)}{|J|} + \frac{2\xi_j C_2(s)C_2'(s)}{|J|}\Big) \\
&\quad + 2^{-l} \cdot \frac{\sum_{i=1}^3|\partial_tW_i^t| +\sum_{i=1}^3|\partial_tW_i^u| + \partial_t\big(2\xi_jC_2'(t)C_2'(u)\big) + \partial_t\big(2\xi_jC_2(s)C_2'(s)\big)}{|J|} \\
&\quad + 2^l\\
&\lesssim \frac{|\partial_tJ|}{|J|} + 2^{-l} \cdot \frac{\sum_{i=1}^3|\partial_tW_i^t| +\sum_{i=1}^3|\partial_tW_i^u| + \partial_t\big(2\xi_jC_2'(t)C_2'(u)\big) + \partial_t\big(2\xi_jC_2(s)C_2'(s)\big)}{|J|} + 2^l\\
&=\frac{|\partial_tJ|}{|J|} + 2^{-l} \cdot \frac{\sum_{i=1}^3|\partial_tW_i^u| + \partial_t\big(2\xi_jC_2'(t)C_2'(u)\big)}{|J|} + 2^l.
\end{align*}
Similarly, on $E_l$ we have
\begin{align*}
|\partial_s \psi_l(\tau)|&\lesssim \frac{|\partial_sJ|}{|J|} + 2^{-l} \cdot \frac{\sum_{i=1}^3|\partial_sW_i^t| +\sum_{i=1}^3|\partial_sW_i^u| + \partial_s\big(2\xi_jC_2(s)C_2'(s)\big)}{|J|} + 2^l,
\end{align*}
and
\begin{align*}
|\partial_u \psi_l(\tau)|&\lesssim \frac{|\partial_uJ|}{|J|} + 2^{-l} \cdot \frac{\sum_{i=1}^3|\partial_uW_i^t| + \partial_u\big(2\xi_jC_2'(t)C_2'(u)\big)}{|J|} + 2^l.
\end{align*}

\textbf{Estimation of } $\bm{\int_{E_l}|g_{11}\partial_t\psi_l|/|J|\,d\tau}$:

By (\ref{g11bound}), we have
\begin{align*}
\int_{E_l} \frac{|g_{11}\partial_t\psi_l|}{|J|}\,d\tau & \lesssim \int_{E_l} \frac{|g_{11}\partial_tJ|}{J^2}\,d\tau +2^{-l} \sum_{i=1}^3 \int_{E_l} \frac{|g_{11}\partial_tW_i^u|}{J^2}\,d\tau\\
&\quad +2^{-l}\int_{E_l} \frac{|g_{11}|\partial_t\big(2\xi_jC_2'(t)C_2'(u)\big)}{J^2}\,d\tau + 2^l\int_{E_l}\frac{|g_{11}|}{|J|}\,d\tau\\
&\lesssim 2^{2l} +2^{-l} \sum_{i=1}^3 \int_{E_l} \frac{|g_{11}\partial_tW_i^u|}{J^2}\,d\tau+2^{-l}\int_{E_l} \frac{|g_{11}|\partial_t\big(2\xi_jC_2'(t)C_2'(u)\big)}{J^2}\,d\tau.
\end{align*}

Recall $|g_{11}|\lesssim |W_1^t| + 2\xi_j C_2(s)C_2'(s) + 2\xi_jC_2(t)C_2'(u)$, and $|W_1^t| + 2\xi_j C_2(s)C_2'(s)$ is independent of $t$. By Lemma \ref{tmonotone}, for each fixed $(s,u)\in [2,8]^2$, $W_1^u$ changes monotonicity at most once and changes sign at most twice on $[2,8]$ as a functions of $t$. Thus $\{t\in [2,8]: |W_1^u|\leq |W_1^t| + 2\xi_j C_2(s)C_2'(s)\}$ is the union of at most two intervals. And $\{t\in [2,8]: |W_1^u|\geq |W_1^t| + 2\xi_j C_2(s)C_2'(s)\}$ is the union of at most three intervals, on each of which $W_1^u$ does not change sign. We can write
\begin{align*}
&\{t\in [2,8]: |W_1^u|\leq |W_1^t| + 2\xi_j C_2(s)C_2'(s)\}=[a_1, a_2]\cup [a_3, a_4]\cup \cdots \cup [a_9, a_{10}],\\
&\{t\in [2,8]: |W_1^u|\geq |W_1^t| + 2\xi_j C_2(s)C_2'(s)\} =[b_1, b_2]\cup [b_3, b_4]\cup [b_5, b_6]\cup [b_7, b_8],
\end{align*}
where $W_1^u$ is monotone and does not change sign on each $[a_{2i-1}, a_{2i}]$ and each $[b_{2i-1}, b_{2i}]$, and $2\leq a_1\leq \cdots \leq a_{10}\leq 8$ and $2\leq b_1\leq \cdots \leq b_8\leq 8$ depends on $(s,u)$. We have
\begin{align*}
&\quad \int_{E_l} \frac{\big(|W_1^t|+2\xi_jC_2(s)C_2'(s)\big)|\partial_tW_1^u|}{J^2}\,d\tau\\
&\lesssim 2^{2l} \int_{|W_1^u|\leq |W_1^t|+2\xi_jC_2(s)C_2'(s)} \frac{|\partial_tW_1^u|}{|W_1^t|+2\xi_jC_2(s)C_2'(s)}\,d\tau\\
&\quad +2^{2l} \int_{|W_1^u|\geq |W_1^t|+2\xi_jC_2(s)C_2'(s)} \frac{\big(|W_1^t|+2\xi_jC_2(s)C_2'(s)\big)|\partial_tW_1^u|}{|W_1^u|^2}\,d\tau\\
&= 2^{2l}\iint_{[2,8]^2}\sum_{i=1}^5\Big|\int_{a_{2i-1}}^{a_{2i}} \frac{\partial_tW_1^u}{|W_1^t|+2\xi_jC_2(s)C_2'(s)}\,dt\Big|\,ds\,du\\
&\quad + 2^{2l}\iint_{[2,8]^2}\sum_{i=1}^4\Big|\int_{b_{2i-1}}^{b_{2i}} \partial_t\Big(\frac{|W_1^t|+2\xi_jC_2(s)C_2'(s)}{W_1^u}\Big)\,dt\Big|\,ds\,du\\
&\lesssim 2^{2l}.
\end{align*}
Similarly,
$$
\int_{E_l} \frac{\big(|W_1^t|+2\xi_jC_2(s)C_2'(s)\big)\partial_t\big(2\xi_jC_2'(t)C_2'(u)\big)}{J^2}\,d\tau \lesssim 2^{2l}.
$$

By Lemma \ref{tmonotone}, for each fixed $(s,u)\in [2,8]^2$, $\frac{W_1^u}{C_2'(t)}$ changes monotonicity at most once and changes sign at most twice on $[2,8]$ as a function of $t$. Thus $\{t\in [2,8]: |W_1^u|\geq \frac{1}{2}\xi_jC_2'(t)C_2'(u)\}$ is the union of at most three intervals, on each of which $W_1^u$ does not change sign. And $\{t\in [2,8]: |W_1^u|\leq \frac{1}{2}\xi_jC_2'(t)C_2'(u)\}$ is the union of at most two intervals. We can write
$$
\{t\in [2,8]: |W_1^u|\geq \frac{1}{2}\xi_jC_2'(t)C_2'(u)\} =[b_1, b_2]\cup [b_3, b_4]\cup [b_5, b_6]\cup [b_7, b_8],
$$
where $W_1^u$ is monotone and does not change sign on each $[b_{2i-1}, b_{2i}]$, and $2\leq b_1\leq \cdots \leq b_8\leq 8$ depends on $(s,u)$. Since $C_2(t)\lesssim C_2'(t)$,
\begin{align*}
&\quad \int_{\substack{E_l\\ |W_1^u|\geq \frac{1}{2}\xi_jC_2'(t)C_2'(u)}} \frac{2\xi_jC_2(t)C_2'(u)|\partial_tW_1^u|}{J^2}\,d\tau\\
&\leq 2^{2l} \int_{|W_1^u|\geq \frac{1}{2}\xi_jC_2'(t)C_2'(u)} \frac{2\xi_jC_2(t)C_2'(u)|\partial_tW_1^u|}{|W_1^u|^2}\,d\tau\\
&= 2^{2l}\iint_{[2,8]^2} \sum_{i=1}^4 \Big| \int_{b_{2i-1}}^{b_{2i}} \Big(\frac{2\xi_jC_2'(u)C_2'(t)}{W_1^u} - \partial_t\big(\frac{2\xi_jC_2'(u)C_2(t)}{W_1^u}\big)\Big)\,dt\Big|\,ds\,du\\
&\lesssim 2^{2l}.
\end{align*}
We have 
\begin{align*}
&\quad \int_{E_l} \frac{2\xi_jC_2(t)C_2'(u)\partial_t\big(2\xi_jC_2'(t)C_2'(u)\big)}{J^2}\,d\tau\\
&\leq 2^{2l} \int_{[2,8]^3} \frac{2\xi_jC_2(t)C_2'(u)\partial_t\big(2\xi_jC_2'(t)C_2'(u)\big)}{\big(2\xi_jC_2'(t)C_2'(u)\big)^2}\,d\tau\\
&=2^{2l}\int_{[2,8]^3} \Big(\frac{2\xi_jC_2'(u)C_2'(t)}{2\xi_jC_2'(t)C_2'(u)} - \partial_t\big(\frac{2\xi_jC_2(t)C_2'(u)}{2\xi_jC_2'(t)C_2'(u)}\big)\Big)\,d\tau\\
&\lesssim 2^{2l}.
\end{align*}

Denote
$$
\widehat W_1^u = W_1^u+\xi_jC_2'(t)C_2'(u).
$$
By Lemma \ref{tmonotone}, $\widehat W_1^u\lesssim 2^l|J|$ on $E_l$. By Lemma \ref{tmonotone}, for each fixed $(s,u)\in [2,8]^2$, $\widehat W_1^u$ changes monotonicity at most once and changes sign at most twice on $[2,8]$ as a function of $t$. Since $\{t\in [2,8]: |W_1^u|\leq \frac{1}{2}\xi_jC_2'(t)C_2'(u)\}$ is the union of at most two intervals. We can write
$$
\{t\in [2,8]: |W_1^u|\leq \frac{1}{2}\xi_jC_2'(t)C_2'(u)\} = [a_1, a_2]\cup \cdots \cup [a_9, a_{10}],
$$
where $\widehat W_1^u$ is monotone and does not change sign on each $[a_{2i-1}, a_{2i}]$, and $2\leq a_1\leq \cdots \leq a_{10}\leq 8$ depends on $(s,u)$. Since $|W_1^u|\leq \frac{1}{2}\xi_jC_2'(t)C_2'(u)$ implies $\widehat W_1^u \geq  \frac{1}{2}\xi_jC_2'(t)C_2'(u)$, we have
\begin{align*}
&\quad \int_{\substack{E_l\\ |W_1^u|\leq \frac{1}{2}\xi_jC_2'(t)C_2'(u)}} \frac{2\xi_jC_2(t)C_2'(u)\big|\partial_t \widehat W_1^u \big|}{J^2}\,d\tau\\
&\lesssim 2^{2l} \int_{|W_1^u|\leq \frac{1}{2}\xi_jC_2'(t)C_2'(u)} \frac{2\xi_jC_2(t)C_2'(u)\big|\partial_t\widehat W_1^u\big|}{|\widehat W_1^u|^2}\,d\tau\\
&=2^{2l}\iint_{[2,8]^2} \sum_{i=1}^5 \Big|\int_{a_{2i-1}}^{a_{2i}} \Big(\frac{2\xi_jC_2'(u)C_2'(t)}{\widehat W_1^u} -\partial_t\big( \frac{2\xi_jC_2'(u)C_2(t)}{\widehat W_1^u}\big)\Big)\,dt\Big|\,ds\,du\\
&\lesssim 2^{2l}.
\end{align*}
Thus
\begin{align*}
&\quad \int_{\substack{E_l\\ |W_1^u|\leq \frac{1}{2}\xi_jC_2'(t)C_2'(u)}} \frac{2\xi_jC_2(t)C_2'(u)|\partial_tW_1^u|}{J^2}\,d\tau\\
&\leq \int_{\substack{E_l\\ |W_1^u|\leq \frac{1}{2}\xi_jC_2'(t)C_2'(u)}} \frac{2\xi_jC_2(t)C_2'(u)\big|\partial_t \widehat W_1^u \big|}{J^2}\,d\tau + \int_{E_l} \frac{2\xi_jC_2(t)C_2'(u)\partial_t\big(2\xi_jC_2'(t)C_2'(u)\big)}{J^2}\,d\tau\\
&\lesssim 2^{2l}.
\end{align*}
Hence
$$
\int_{E_l} \frac{2\xi_jC_2(t)C_2'(u)|\partial_tW_1^u|}{J^2}\,d\tau \lesssim 2^{2l}.
$$
Similarly,
$$
\int_{E_l} \frac{2\xi_jC_2(t)C_2'(u)|\partial_tW_2^u|}{J^2}\,d\tau, \quad \int_{E_l} \frac{2\xi_jC_2(t)C_2'(u)|\partial_tW_3^u|}{J^2}\,d\tau \lesssim 2^{2l}.
$$
Therefore
$$
\sum_{i=1}^3 \int_{E_l} \frac{|g_{11}\partial_tW_i^u|}{J^2}\,d\tau+\int_{E_l} \frac{|g_{11}|\partial_t\big(2\xi_jC_2'(t)C_2'(u)\big)}{J^2}\,d\tau \lesssim 2^{2l}.
$$

\textbf{Estimation of } $\bm{\int_{E_l} |g_{12}\partial_s\psi_l|/|J|\,d\tau}$:

By (\ref{g12bound}), we have
\begin{align*}
\int_{E_l} \frac{|g_{12}\partial_s\psi_l|}{|J|}\,d\tau &\lesssim \int_{E_l} \frac{|g_{12}\partial_sJ|}{J^2}\,d\tau + 2^{-l}\sum_{i=1}^3\int_{E_l} \frac{|g_{12}\partial_sW_i^t|}{J^2}\,d\tau + 2^{-l}\sum_{i=1}^3 \int_{E_l} \frac{|g_{12}\partial_sW_i^u|}{J^2}\,d\tau \\
&\quad +2^{-l}\int_{E_l} \frac{|g_{12}|\partial_s\big(2\xi_jC_2(s)C_2'(s)\big)}{J^2}\,d\tau + 2^l\int_{E_l} \frac{|g_{12}|}{|J|}\,d\tau\\
&\lesssim 2^{2l} +2^{-l}\sum_{i=1}^3\int_{E_l} \frac{|g_{12}\partial_sW_i^t|}{J^2}\,d\tau + 2^{-l}\sum_{i=1}^3 \int_{E_l} \frac{|g_{12}\partial_sW_i^u|}{J^2}\,d\tau \\
&\quad +2^{-l}\int_{E_l} \frac{|g_{12}|\partial_s\big(2\xi_jC_2(s)C_2'(s)\big)}{J^2}\,d\tau.
\end{align*}
Recall $|g_{12}|\lesssim |\tilde g_{12}| + 2\xi_jC_2'(t)C_2'(u) + 2\xi_jC_2'(t)C_2(s)$, and $|\tilde g_{12}| + 2\xi_jC_2'(t)C_2'(u)$ is independent of $s$. By Lemma \ref{smonotonic}, for each fixed $(t,u)\in [2,8]^2$, $W_1^u+2\xi_jC_2(s)C_2'(s)$ changes monotonicity at most once and changes sign at most twice on $[2,8]$ as a function of $s$. Thus $\{s\in [2,8]: |W_1^u+2\xi_jC_2(s)C_2'(s)|\geq |\tilde g_{12}|+2\xi_jC_2'(t)C_2'(u)\}$ is the union of at most three intervals, on each of which $W_1^u+2\xi_jC_2(s)C_2'(s)$ does not change sign. And $\{s\in [2,8]: |W_1^u+2\xi_jC_2(s)C_2'(s)|\leq |\tilde g_{12}|+2\xi_jC_2'(t)C_2'(u)\}$ in the union of at most two intervals. We can write
\begin{align*}
\{s\in [2,8]: |W_1^u+2\xi_jC_2(s)C_2'(s)|\leq |\tilde g_{12}|+2\xi_jC_2'(t)C_2'(u)\}=[a_1, a_2]\cup [a_3, a_4]\cup \cdots \cup [a_9, a_{10}],\\
\{s\in [2,8]: |W_1^u+2\xi_jC_2(s)C_2'(s)|\geq |\tilde g_{12}|+2\xi_jC_2'(t)C_2'(u)\}=[b_1, b_2]\cup [b_3, b_4]\cup [b_5, b_6]\cup [b_7, b_8],
\end{align*}
where $W_1^u+2\xi_jC_2(s)C_2'(s)$ is monotone and does not change sign on each $[a_{2i-1}, a_{2i}]$ and each $[b_{2i-1}, b_{2i}]$, and $2\leq a_1 \leq \cdots \leq a_{10}\leq 8$ and $2\leq b_1\leq \cdots \leq b_8\leq 8$ depends on $(t,u)$. We have
\begin{align*}
&\quad \int_{E_l} \frac{\big(|\tilde g_{12}| + 2\xi_jC_2'(t)C_2'(u)\big)|\partial_s\big(W_1^u +2\xi_jC_2(s)C_2'(s)\big)|}{J^2}\,d\tau\\
&\lesssim 2^{2l} \int_{|W_1^u+2\xi_jC_2(s)C_2'(s)|\leq |\tilde g_{12}|+2\xi_jC_2'(t)C_2'(u)} \frac{|\partial_s\big(W_1^u +2\xi_jC_2(s)C_2'(s)\big)|}{|\tilde g_{12}|+2\xi_jC_2'(t)C_2'(u)}\,d\tau\\
&\quad + 2^{2l} \int_{|W_1^u+2\xi_jC_2(s)C_2'(s)|\geq |\tilde g_{12}|+2\xi_jC_2'(t)C_2'(u)} \frac{\big(|\tilde g_{12}| + 2\xi_jC_2'(t)C_2'(u)\big)|\partial_s\big(W_1^u +2\xi_jC_2(s)C_2'(s)\big)|}{\big|W_1^u +2\xi_jC_2(s)C_2'(s)\big|^2}\,d\tau\\
&=2^{2l}\iint_{[2,8]^2}\sum_{i=1}^5\Big|\int_{a_{2i-1}}^{a_{2i}} \frac{\partial_s\big(W_1^u +2\xi_jC_2(s)C_2'(s)\big)}{|\tilde g_{12}|+2\xi_jC_2'(t)C_2'(u)}\,ds\Big|\,dt\,du\\
&\quad +2^{2l} \iint_{[2,8]^2}\sum_{i=1}^4\Big|\int_{b_{2i-1}}^{b_{2i}}\partial_s\Big( \frac{|\tilde g_{12}| + 2\xi_jC_2'(t)C_2'(u)}{W_1^u +2\xi_jC_2(s)C_2'(s)}\Big)\,ds\Big|\,ds\,du\\
&\lesssim 2^{2l}.
\end{align*}
Similarly,
\begin{align*}
&\int_{E_l} \frac{\big(|\tilde g_{12}| + 2\xi_jC_2'(t)C_2'(u)\big)|\partial_s\big(W_1^t +2\xi_jC_2(s)C_2'(s)\big)|}{J^2}\,d\tau \lesssim 2^{2l},\\
&\int_{E_l} \frac{\big(|\tilde g_{12}| + 2\xi_jC_2'(t)C_2'(u)\big)|\partial_s\big(W_2^u +2\xi_jsC_2'(t)\big)|}{J^2}\,d\tau\lesssim 2^{2l},\\
&\int_{E_l} \frac{\big(|\tilde g_{12}| + 2\xi_jC_2'(t)C_2'(u)\big)|\partial_s\big(W_2^t +2\xi_jsC_2'(u)\big)|}{J^2}\,d\tau\lesssim 2^{2l},\\
&\int_{E_l} \frac{\big(|\tilde g_{12}| + 2\xi_jC_2'(t)C_2'(u)\big)|\partial_sW_3^u|}{J^2}\,d\tau \lesssim 2^{2l},\\
&\int_{E_l} \frac{\big(|\tilde g_{12}| + 2\xi_jC_2'(t)C_2'(u)\big)|\partial_sW_3^t|}{J^2}\,d\tau \lesssim 2^{2l},\\
&\int_{E_l} \frac{\big(|\tilde g_{12}| + 2\xi_jC_2'(t)C_2'(u)\big)|\partial_s\big(2\xi_jC_2(s)C_2'(s)\big)|}{J^2}\,d\tau\lesssim 2^{2l}.
\end{align*}
And since
\begin{align*}
&\int_{E_l} \frac{\big(|\tilde g_{12}| + 2\xi_jC_2'(t)C_2'(u)\big)|\partial_s\big(2\xi_jsC_2'(t)\big)|}{J^2}\,d\tau=\int_{E_l} \frac{\big(|\tilde g_{12}| + 2\xi_jC_2'(t)C_2'(u)\big)2\xi_jC_2'(t)}{J^2}\,d\tau\lesssim 2^{2l},\\
&\int_{E_l} \frac{\big(|\tilde g_{12}| + 2\xi_jC_2'(t)C_2'(u)\big)|\partial_s\big(2\xi_jsC_2'(u)\big)|}{J^2}\,d\tau=\int_{E_l} \frac{\big(|\tilde g_{12}| + 2\xi_jC_2'(t)C_2'(u)\big)2\xi_jC_2'(u)}{J^2}\,d\tau\lesssim 2^{2l},
\end{align*}
for $i=1,2,3$,
\begin{align*}
&\int_{E_l} \frac{\big(|\tilde g_{12}| + 2\xi_jC_2'(t)C_2'(u)\big)|\partial_sW_i^u|}{J^2}\,d\tau \lesssim 2^{2l},\\
&\int_{E_l} \frac{\big(|\tilde g_{12}| + 2\xi_jC_2'(t)C_2'(u)\big)|\partial_sW_i^t|}{J^2}\,d\tau \lesssim 2^{2l}.
\end{align*}

Similarly, for $i=1,2,3$,
\begin{align*}
&\int_{\substack{E_l\\ C_2(s)\leq C_2'(u)}} \frac{2\xi_jC_2(s)C_2'(t)|\partial_sW_i^u|}{J^2}\,d\tau \leq \int_{E_l} \frac{2\xi_jC_2'(t)C_2'(u)|\partial_sW_i^u|}{J^2}\,d\tau \lesssim 2^{2l},\\
&\int_{\substack{E_l\\ C_2(s)\leq C_2'(u)}} \frac{2\xi_jC_2(s)C_2'(t)|\partial_sW_i^t|}{J^2}\,d\tau \leq \int_{E_l} \frac{2\xi_jC_2'(t)C_2'(u)|\partial_sW_i^t|}{J^2}\,d\tau \lesssim 2^{2l},\\
&\int_{\substack{E_l\\C_2(s)\leq C_2'(u)}} \frac{2\xi_jC_2(s)C_2'(t)\partial_s\big(2\xi_jC_2(s)C_2'(s)\big)}{J^2}\,d\tau \leq \int_{E_l} \frac{2\xi_jC_2'(t)C_2'(u)\partial_s\big(2\xi_jC_2(s)C_2'(s)\big)}{J^2}\,d\tau \lesssim 2^{2l}.
\end{align*}
By Lemma \ref{smonotonic}, for each fixed $(t,u)\in [2,8]^2$, $\frac{W_1^u+2\xi_jC_2(s)C_2'(s)}{C_2'(s)}$ changes monotonicity at most once and changes sign at most twice on $[2,8]$ as a functions of $s$. Thus $\{s\in [2,8]: |W_1^u+2\xi_jC_2(s)C_2'(s)|\leq \frac{1}{2}\xi_jC_2'(s)C_2'(t)\}$ is the union of at most two intervals. And $\{s\in [2,8]: |W_1^u+2\xi_jC_2(s)C_2'(s)|\geq \frac{1}{2}\xi_jC_2'(s)C_2'(t)\}$ is the union of at most three intervals, on each of which $W_1^u+2\xi_jC_2(s)C_2'(s)$ does not change sign. We can write
$$
\{s\in [2,8]: |W_1^u+2\xi_jC_2(s)C_2'(s)|\geq \frac{1}{2}\xi_jC_2'(s)C_2'(t)\} =[b_1, b_2]\cup [b_3, b_4]\cup [b_5, b_6]\cup [b_7, b_8],
$$
where $W_1^u +2\xi_jC_2(s)C_2'(s)$ is monotone and does not change sign on each $[b_{2i-1}, b_{2i}]$, and $2\leq b_1\leq \cdots \leq b_8\leq 8$ depends on $(t,u)$. We have
\begin{align*}
&\quad \int_{\substack{E_l\\ C_2(s)\geq C_2'(u)\\|W_1^u+2\xi_jC_2(s)C_2'(s)|\geq \frac{1}{2}\xi_jC_2'(s)C_2'(t)}} \frac{2\xi_jC_2(s)C_2'(t) |\partial_s\big(W_1^u +2\xi_jC_2(s)C_2'(s)\big)|}{J^2}\,d\tau\\
&\lesssim 2^{2l} \int_{|W_1^u+2\xi_jC_2(s)C_2'(s)|\geq \frac{1}{2}\xi_jC_2'(s)C_2'(t)}\frac{2\xi_jC_2(s)C_2'(t) |\partial_s\big(W_1^u +2\xi_jC_2(s)C_2'(s)\big)|}{\big|W_1^u +2\xi_jC_2(s)C_2'(s)\big|^2}\,d\tau\\
&=2^{2l} \iint_{[2,8]^2} \sum_{i=1}^4\Big|\int_{b_{2i-1}}^{b_{2i}} \Big(\frac{2\xi_jC_2'(s)C_2'(t)}{W_1^u +2\xi_jC_2(s)C_2'(s)}-\partial_s \big(\frac{2\xi_jC_2(s)C_2'(t)}{W_1^u +2\xi_jC_2(s)C_2'(s)}\big) \Big)\,ds\Big|\,dt\,du \\
&\lesssim 2^{2l}.
\end{align*}

Denote
$$
\doublehat W_1^u:=W_1^u +2\xi_jC_2(s)C_2'(s)+\xi_jC_2'(s)C_2'(t).
$$
On $E_l$, since $t\leq u$, we have when $C_2(s)\geq C_2'(u)$,
$$
2\xi_jC_2'(s)C_2'(t)\leq 2\xi_j C_2'(s)C_2(s) \leq 2^l|J|,
$$
and thus $|\doublehat W_1^u|\lesssim 2^l|J|$. By Lemma \ref{smonotonic}, for each fixed $(t,u)\in [2,8]^2$, $\doublehat W_1^u$ changes monotonicity at most once and changes sign at most twice on $[2,8]$ as a functions of $s$. Since $\{s\in [2,8]: |W_1^u+2\xi_jC_2(s)C_2'(s)|\leq \frac{1}{2}\xi_jC_2'(s)C_2'(t)\}$ is the union of at most two intervals, we can write
$$
\{s\in [2,8]: |W_1^u+2\xi_jC_2(s)C_2'(s)|\leq \frac{1}{2}\xi_jC_2'(s)C_2'(t)\}=[a_1, a_2]\cup \cdots \cup [a_9,a_{10}],
$$
where $\doublehat W_1^u$ is monotone and does not change sign on each $[b_{2i-1}, b_{2i}]$, and $2\leq b_1\leq \cdots \leq b_8\leq 8$ depends on $(t,u)$. And since $|W_1^u+2\xi_jC_2(s)C_2'(s)|\leq \frac{1}{2}\xi_jC_2'(s)C_2'(t)$ implies $\doublehat W_1^u \geq \frac{1}{2}\xi_jC_2'(s)C_2'(t)$, we have
\begin{align*}
&\quad \int_{\substack{E_l\\ C_2(s)\geq C_2'(u)\\|W_1^u+2\xi_jC_2(s)C_2'(s)|\leq \frac{1}{2}\xi_jC_2'(s)C_2'(t)}} \frac{2\xi_jC_2(s)C_2'(t) |\partial_s\doublehat W_1^u|}{J^2}\,d\tau\\
& \lesssim 2^{2l} \int_{|W_1^u+2\xi_jC_2(s)C_2'(s)|\leq \frac{1}{2}\xi_jC_2'(s)C_2'(t)} \frac{2\xi_jC_2(s)C_2'(t) |\partial_s\doublehat W_1^u|}{|\doublehat W_1^u|^2}\,d\tau\\
&=2^{2l} \iint_{[2,8]^2} \sum_{i=1}^5 \Big|\int_{a_{2i-1}}^{a_{2i}} \Big(\frac{2\xi_jC_2'(s)C_2'(t)}{\doublehat W_1^u} - \partial_s\big(\frac{2\xi_jC_2(s)C_2'(t)}{\doublehat W_1^u}\big)\Big)\,ds\Big|\,dt\,du\\
&\lesssim 2^{2l}.
\end{align*}
Since
\begin{align*}
&\quad \int_{\substack{E_l\\ C_2(s)\geq C_2'(u)}} \frac{2\xi_jC_2(s)C_2'(t)\partial_s\big(2\xi_jC_2'(s)C_2'(t)\big)}{J^2}\,d\tau\\
&\leq 2^{2l} \int_{[2,8]^3} \frac{2\xi_jC_2(s)C_2'(t)\partial_s \big(2\xi_jC_2'(s)C_2'(t)\big)}{\big(2\xi_jC_2'(s)C_2'(t)\big)^2}\,d\tau\\
&=2^{2l} \int_{[2,8]^3} \Big(\frac{2\xi_jC_2'(s)C_2'(t)}{2\xi_jC_2'(s)C_2'(t)} - \partial_s\big(\frac{2\xi_jC_2(s)C_2'(t)}{2\xi_jC_2'(s)C_2'(t)}\big)\Big)\,d\tau\\
&\lesssim 2^{2l},
\end{align*}
we have
\begin{align*}
&\quad \int_{\substack{E_l\\ C_2(s)\geq C_2'(u)\\|W_1^u+2\xi_jC_2(s)C_2'(s)|\leq \frac{1}{2}\xi_jC_2'(s)C_2'(t)}} \frac{2\xi_jC_2(s)C_2'(t) |\partial_s\big(W_1^u +2\xi_jC_2(s)C_2'(s)\big)|}{J^2}\,d\tau\\
&\leq \int_{\substack{E_l\\ C_2(s)\geq C_2'(u)\\|W_1^u+2\xi_jC_2(s)C_2'(s)|\leq \frac{1}{2}\xi_jC_2'(s)C_2'(t)}} \frac{2\xi_jC_2(s)C_2'(t) |\partial_s\doublehat W_1^u|}{J^2}\,d\tau\\
&\quad + \int_{\substack{E_l\\ C_2(s)\geq C_2'(u)\\|W_1^u+2\xi_jC_2(s)C_2'(s)|\leq \frac{1}{2}\xi_jC_2'(s)C_2'(t)}} \frac{2\xi_jC_2(s)C_2'(t)\partial_s\big(\xi_jC_2'(s)C_2'(t)\big)}{J^2}\,d\tau\\
&\lesssim 2^{2l}.
\end{align*}
Hence
\begin{align*}
\int_{\substack{E_l\\ C_2(s)\geq C_2'(u)}} \frac{2\xi_jC_2(s)C_2'(t) |\partial_s\big(W_1^u +2\xi_jC_2(s)C_2'(s)\big)|}{J^2}\,d\tau\lesssim 2^{2l}.
\end{align*}
Similarly,
\begin{align*}
&\int_{\substack{E_l\\ C_2(s)\geq C_2'(u)}} \frac{2\xi_jC_2(s)C_2'(t) |\partial_s\big(W_1^t +2\xi_jC_2(s)C_2'(s)\big)|}{J^2}\,d\tau\lesssim 2^{2l},\\
&\int_{\substack{E_l\\ C_2(s)\geq C_2'(u)}} \frac{2\xi_jC_2(s)C_2'(t) |\partial_s\big(W_2^u+2\xi_jsC_2'(t)\big)|}{J^2}\,d\tau\lesssim 2^{2l},\\
&\int_{\substack{E_l\\ C_2(s)\geq C_2'(u)}} \frac{2\xi_jC_2(s)C_2'(t) |\partial_s\big(W_2^t+2\xi_jsC_2'(u)\big)|}{J^2}\,d\tau\lesssim 2^{2l},\\
&\int_{\substack{E_l\\ C_2(s)\geq C_2'(u)}} \frac{2\xi_jC_2(s)C_2'(t) |\partial_sW_3^u|}{J^2}\,d\tau\lesssim 2^{2l},\\
&\int_{\substack{E_l\\ C_2(s)\geq C_2'(u)}} \frac{2\xi_jC_2(s)C_2'(t) |\partial_sW_3^t|}{J^2}\,d\tau\lesssim 2^{2l},\\
&\int_{\substack{E_l\\ C_2(s)\geq C_2'(u)}} \frac{2\xi_jC_2(s)C_2'(t) |\partial_s\big(2\xi_jC_2(s)C_2'(s)\big)|}{J^2}\,d\tau\lesssim 2^{2l}.
\end{align*}
And since
\begin{align*}
&\int_{E_l} \frac{2\xi_jC_2(s)C_2'(t) |\partial_s\big(2\xi_jsC_2'(t)\big)|}{J^2}\,d\tau =\int_{E_l} \frac{2\xi_jC_2(s)C_2'(t) 2\xi_jC_2'(t)}{J^2}\,d\tau \lesssim 2^{2l},\\
&\int_{E_l} \frac{2\xi_jC_2(s)C_2'(t) |\partial_s\big(2\xi_jsC_2'(u)\big)|}{J^2}\,d\tau =\int_{E_l} \frac{2\xi_jC_2(s)C_2'(t) 2\xi_jC_2'(u)}{J^2}\,d\tau \lesssim 2^{2l},
\end{align*}
for $i=1,2,3$,
\begin{align*}
\int_{\substack{E_l\\ C_2(s)\geq C_2'(u)}} \frac{2\xi_jC_2(s)C_2'(t) |\partial_sW_i^u|}{J^2}\,d\tau\lesssim 2^{2l},\quad \int_{\substack{E_l\\ C_2(s)\geq C_2'(u)}} \frac{2\xi_jC_2(s)C_2'(t) |\partial_sW_i^t|}{J^2}\,d\tau\lesssim 2^{2l}.
\end{align*}
Hence for $i=1,2,3$,
\begin{align*}
\int_{E_l} \frac{2\xi_jC_2(s)C_2'(t) |\partial_sW_i^u|}{J^2}\,d\tau,\quad \int_{E_l} \frac{2\xi_jC_2(s)C_2'(t) |\partial_sW_i^t|}{J^2}\,d\tau, \quad \int_{E_l} \frac{2\xi_jC_2(s)C_2'(t) \partial_s\big(2\xi_jC_2(s)C_2'(s)\big)}{J^2}\,d\tau \lesssim 2^{2l}.
\end{align*}
Therefore $\int_{E_l} \frac{|g_{12}\partial_s\psi_l|}{|J|}\,d\tau \lesssim 2^{2l}$.

The estimations of
$$
\int_{E_l} \frac{|g_{13}\partial_u \psi_l|}{|J|}\,d\tau, \int_{E_l} \frac{|g_{21}\partial_t \psi_l|}{|J|}\,d\tau, \int_{E_l} \frac{|g_{22}\partial_s \psi_l|}{|J|}\,d\tau, \int_{E_l} \frac{|g_{23}\partial_u \psi_l|}{|J|}\,d\tau, \int_{E_l} \frac{|g_{31}\partial_t \psi_l|}{|J|}\,d\tau, \int_{E_l} \frac{|g_{32}\partial_s \psi_l|}{|J|}\,d\tau, \int_{E_l} \frac{|g_{33}\partial_u \psi_l|}{|J|}\,d\tau
$$
follow in a similar way and are simpler. We therefore omit their calculations and finish the proof.
\end{proof}

The rest of this subsection is devoted to Lemmas \ref{sign}, \ref{sign2}, \ref{monotonic}, \ref{tmonotone}, \ref{smonotonic}.

\begin{proof}[Proof of Lemma \ref{sign}]
As functions of $t$, by (\ref{combo}), (\ref{C3}), and Lemma \ref{standardmonotone}, we have
\begin{align*}
\Big(\frac{\partial_tJ}{C_2''(t)}\Big)'=\Big(W_2^t+\frac{(C_3^+)''(t)+\frac{2^{-j}h_2(2^{-j})-\int_0^{2^{-j}}h_2}{2^{-j}h_2(2^{-j})+\bar h_3(2^{-j})}C_2''(t)}{C_2''(t)}W_3^t\Big)' = \Big(\frac{(C_3^+)''(t)}{C_2''(t)}\Big)'W_3^t
\end{align*}
does not change sign since $W_3^t$ is independent of $t$. Thus $\partial_tJ/C_2''(t)$ is monotone and changes sign at most once. Hence $\partial_tJ$ changes sign at most once. Therefore $J$ changes monotonicity at most once and changes sign at most twice. It is similar for $J$ as a function of $u$, and for $\tilde J$ as a functions of $t$.
\end{proof}

\begin{proof}[Proof of Lemma \ref{sign2}]
By (\ref{C3}), we have
\begin{align*}
\Big(\frac{C_3'(t)+\xi_j(tC_2'(t)-C_2(t))}{C_2'(t)}\Big)' &= \Big(\frac{(C_3^+)'(t) + \frac{2^{-j}h_2(2^{-j})-\int_0^{2^{-j}}h_2}{2^{-j}h_2(2^{-j})+\bar h_3(2^{-j})}C_2'(t)+\frac{\int_0^{2^{-j}}h_2}{2^{-j}h_2(2^{-j})+\bar h_3(2^{-j})}}{C_2'(t)}\Big)'\\
&=\frac{C_2''(t)}{C_2'(t)^2} \Big(\frac{(C_3^+)''(t)}{C_2''(t)}C_2'(t)-(C_3^+)'(t)-\frac{\int_0^{2^{-j}}h_2}{2^{-j}h_2(2^{-j})+\bar h_3(2^{-j})}\Big).
\end{align*}
By Lemma \ref{standardmonotone},
\begin{align*}
\Big(\frac{(C_3^+)''(t)}{C_2''(t)}C_2'(t)-(C_3^+)'(t)-\frac{\int_0^{2^{-j}}h_2}{2^{-j}h_2(2^{-j})+\bar h_3(2^{-j})}\Big)'= \Big(\frac{(C_3^+)''(t)}{C_2''(t)}\Big)'C_2'(t)>0,
\end{align*}
and thus $\big((C_3'(t)+\xi_j(tC_2'(t)-C_2(t)))/C_2'(t)\big)'$ changes sign at most once. Since $W_3^t$ is independent of $t$, by (\ref{combo}), we have
\begin{align*}
\Big(\frac{J- W_1^t}{C_2'(t)}\Big)'= \Big(W_2^t+ \frac{C_3'(t) +\xi_j \big(tC_2'(t)-C_2(t)\big)}{C_2'(t)} W_3^t\Big)' 
\end{align*}
changes sign at most once.
\end{proof}

\begin{proof}[Proof of Lemma \ref{monotonic}]
By Lemma \ref{standardmonotone},
\begin{align*}
&\Big(\frac{\partial_sJ_1}{C_2''(s)}\Big)'=-C_2'(u) \Big(\frac{\doublehat C_3''(s)}{C_2''(s)}\Big)'<0,\\
&\Big(\frac{\partial_sJ_2}{C_2''(s)}\Big)'=C_2'(t) \Big(\frac{\hat C_3''(s)}{C_2''(s)}\Big)'>0.
\end{align*}
Thus $\frac{\partial_sJ_1}{C_2''(s)}$ and $\frac{\partial_sJ_2}{C_2''(s)}$ change sign at most once. Therefore $J_1$ and $J_2$ change monotonicity at most once, and hence change sign at most twice. Similarly, $\tilde J_1$ and $\tilde J_2$ change monotonicity at most once and change sign at most twice. $J_3$ is monotonically decreasing as a functions as $s$, and changes sign at most once. $J_4$ is a monotonically decreasing negative function of $s$.

By Lemma \ref{standardmonotone},
\begin{align*}
&\Big(\frac{J_1}{C_2'(s)}\Big)'= -C_2'(u)\frac{C_2''(s)}{C_2'(s)^2}\Big(\frac{\doublehat C_3''(s)}{C_2''(s)}C_2'(s)-\doublehat C_3'(s)\Big), \quad \Big(\frac{\doublehat C_3''(s)}{C_2''(s)}C_2'(s)-\doublehat C_3'(s)\Big)'=C_2'(s) \Big(\frac{\doublehat C_3''(s)}{C_2''(s)}\Big)'>0,\\
&\Big(\frac{J_2}{C_2'(s)}\Big)'=C_2'(t)\frac{C_2''(s)}{C_2'(s)^2} \Big(\frac{\hat C_3''(s)}{C_2''(s)} C_2'(s)-\hat C_3'(s)\Big), \quad \Big(\frac{\hat C_3''(s)}{C_2''(s)} C_2'(s)-\hat C_3'(s)\Big)'=C_2'(s) \Big(\frac{\hat C_3''(s)}{C_2''(s)}\Big)'>0.
\end{align*}
Thus $\frac{\hat C_3''(s)}{C_2''(s)} C_2'(s)-\hat C_3'(s)$ is monotonically increasing and hence changes sign at most once. Therefore $\frac{J_2}{C_2'(s)}$ changes monotonicity at most once and hence changes sign at most twice. Similarly, $\frac{J_1}{C_2'(s)}$  changes monotonicity at most once and changes sign at most twice.

On $E_l$, we have
\begin{align*}
&|J_1| = |W_1^t + 2\xi_j C_2(s)C_2'(s)|\leq2\cdot 2^l|J|,\\
&|J_2|=|W_1^u+2\xi_jC_2(s)C_2'(s)| \leq 2\cdot 2^l|J|,\\
&|J_4|=2\xi_jC_2(s)C_2'(s)\leq 2^l|J|,\\
&|J_3|=|J-J_1-J_2-J_4| \leq 6\cdot 2^l|J|.
\end{align*}
\end{proof}

\begin{proof}[Proof of Lemma \ref{tmonotone}]
Similar to the proof of Lemma \ref{monotonic}.
\end{proof}

\begin{proof}[Proof of Lemma \ref{smonotonic}]
Similar to the proof of Lemma \ref{monotonic}.
\end{proof}

\subsection{Remarks on the omitted cases}\label{4.4}

\begin{remark}\label{modification}
We have focused on the measure $\kappa(z)\,dz$ derived from the region $R_1=\{(t,s,u)\in [2,8]^3: t\leq u\}$ in this section when we prove Proposition \ref{density}. Proofs for measures $\hat \kappa(w)\,dw, \chi(\zeta)\,d\zeta, \hat \chi(\zeta)\,d\zeta$ follow in a similar way. Take $\hat \chi(\zeta)\,d\zeta$ for an example, we want to show for some $\sigma>0$,
$$
\int |\hat \chi(\zeta)|\,d\zeta\lesssim 1, \quad \int \big|\hat \chi(\zeta)-\hat \chi \big(\zeta \odot_j \tilde v^{-1}\big)\big|\,d\zeta \lesssim \delta^\sigma, \quad \forall j\in \mathbb{N}, 0\leq |\tilde v|\lesssim \delta \leq 1, 2\leq s\leq 8.
$$
Recall
$$
R_4=\{(t,u,r)\in [2,8]^3: t\geq r\geq s, t\geq u\}.
$$
In this case, we have $g_{11}=w_{11}-\xi_j\zeta_2w_{31}$ and $g_{13}=w_{13}-\xi_j\zeta_2w_{33}$, in contrast to (\ref{gij}). The $g_{11}$ in this case is independent of $t$ and plays the role of the $g_{13}$ in Subsection \ref{unity}, and the $g_{13}$ in this case depends on $r$ and plays the role of the $g_{11}$ in Subsection \ref{unity}. 
\end{remark}

\begin{remark}\label{change}
We have focused only on the operator $[\mathcal{D}_k^*\mathcal{H}_j\mathcal{H}_j^*\mathcal{H}_j\mathcal{H}_j^*\mathcal{D}_k]_+$ when we show (\ref{cotlarestimate}) in this section. The proofs for omitted cases are similar. In contrast to (\ref{mainequation}), the part of the operator $\mathcal{D}_k\mathcal{H}_j^*\mathcal{H}_j \mathcal{H}_j^* \mathcal{H}_j \mathcal{D}_k^*$ on the positive region is of the form
\begin{align*}
&\quad [\mathcal{D}_k\mathcal{H}_j^*\mathcal{H}_j \mathcal{H}_j^* \mathcal{H}_j \mathcal{D}_k^*]_+f(x) \\
&= \iiint \!\!\! \iiint_{(t,s,u,r)\in ([2,8]\cap[0,2^j])^4} f\Big(x\cdot e^{-y_1X_k - y_2 Y_k -y_3 T_k} e^{C_1(t)X_j+ C_2(t)Y_j + C_3(t)T_j} e^{-C_1(s)X_j- C_2(s)Y_j - C_3(s)T_j}  \\
&\qquad \qquad \qquad \qquad \qquad \qquad  e^{C_1(u)X_j+ C_2(u)Y_j+ C_3(u)T_j} e^{-C_1(r)X_j-C_2(r)Y_j-C_3(r)T_j} e^{v_1X_k+v_2Y_k+v_3T_k} 0\Big)\\
&\qquad \qquad \qquad \qquad \qquad \qquad \qquad \qquad \qquad \qquad \qquad \bar \varsigma_k(v)\phi_j(r)\bar \phi_j(u) \phi_j(s) \bar \phi_j(t) \varsigma_k(y)    \,dv\,dr\,du\,ds\,dt\,dy.
\end{align*}
By the oddness of the functions $C_1(t), C_2(t), C_3(t)$, the parts of the operators $\mathcal{D}_k^*\mathcal{H}_j\mathcal{H}_j^*\mathcal{H}_j\mathcal{H}_j^*\mathcal{D}_k$ and $\mathcal{D}_k\mathcal{H}_j^*\mathcal{H}_j \mathcal{H}_j^* \mathcal{H}_j \mathcal{D}_k^*$ on the non-positive regions are of the form
\begin{align*}
&\iiint \!\!\! \iiint_{(t,s,u,r)\in ([2,8]\cap[0,2^j])^4} f\Big(x\cdot e^{\pm y_1X_k \pm y_2 Y_k \pm y_3 T_k} e^{\pm C_1(t)X_j\pm C_2(t)Y_j \pm C_3(t)T_j} e^{\pm C_1(s)X_j\pm C_2(s)Y_j \pm C_3(s)T_j}  \\
&\qquad \qquad \qquad \qquad \qquad \qquad  e^{\pm C_1(u)X_j\pm C_2(u)Y_j\pm C_3(u)T_j} e^{\pm C_1(r)X_j\pm C_2(r)Y_j\pm C_3(r)T_j} e^{\pm v_1X_k\pm v_2Y_k\pm v_3T_k} 0\Big)\\
&\qquad \qquad \qquad \qquad \qquad \qquad \qquad \qquad \qquad \qquad \qquad \bar \varsigma_k(v)\phi_j(r)\bar \phi_j(u) \phi_j(s) \bar \phi_j(t) \varsigma_k(y)    \,dv\,dr\,du\,ds\,dt\,dy.
\end{align*}
\end{remark}

\section{Boundedness of  \texorpdfstring{$\mathcal{M}_\Gamma$}{TEXT} and \texorpdfstring{$\mathcal{H}_\Gamma$}{TEXT}}\label{finalsection}

In this section we prove Theorem \ref{1}, given Corollary \ref{completelittlewood} and Theorem \ref{main}. 

The boundedness of $\mathcal{M}_\Gamma$ follows by a well-known bootstrapping argument. See Section 11 of Carbery, Wainger, and Wright \cite{CWW}.

Denote $\mathcal{D}_j=\mathcal{H}_j=0$ for $j<0$. For each $k_1, k_2\in \mathbb{Z}$, define the operator
$$
\mathcal{T}_{k_1,k_2}\{f_j\}_{j\in \mathbb{N}} =\{\mathcal{D}_j\mathcal{H}_{j+k_1} \mathcal{D}_{j+k_2} f_j\}_{j\in \mathbb{N}}.
$$
The boundedness of $\mathcal{H}_\Gamma=\sum_{j\in \mathbb{N}} \mathcal{H}_j$, and thus Theorem \ref{1}, follows from the following two Lemmas:
\begin{lemma}[Street and Stein \cite{LP}]
Fix $1<p<\infty$. If there exists $\epsilon>0$ such that
$$
\big\|\mathcal{T}_{k_1, k_2}\big\|_{L^p(\mathbb{H}^1,l^2(\mathbb{N}))\to L^p(\mathbb{H}^1,l^2(\mathbb{N}))} \lesssim 2^{-\epsilon(|k_1|+|k_2|)},
$$
then $\mathcal{H}_\Gamma$ is bounded on $L^p(\mathbb{H}^1)$.
\end{lemma}

\begin{lemma}[Street and Stein \cite{LP}]
For every $1<p\leq 2$, there exists $\epsilon>0$ such that
$$
\big\|\mathcal{T}_{k_1, k_2}\big\|_{L^p(\mathbb{H}^1,l^2(\mathbb{N}))\to L^p(\mathbb{H}^1,l^2(\mathbb{N}))} \lesssim 2^{-\epsilon(|k_1|+|k_2|)}.
$$
\end{lemma}

\bibliographystyle{amsalpha}
\bibliography{ref}

\vspace{2em} \noindent Lingxiao Zhang\\ Department of Mathematics\\ University of Connecticut\\ 341 Mansfield Rd, Storrs, CT, 06269, USA\\ email: \href{mailto:lingxiao.zhang@uconn.edu}{lingxiao.zhang@uconn.edu}\\ MSC: primary 42B20, 42B25

\end{document}